\documentclass{article}[12pt]
\usepackage{a4}
\usepackage[english]{babel}
\usepackage[T1]{fontenc}
\usepackage[all]{xy}
\usepackage{amssymb}
\usepackage{amsmath}
\usepackage{amsthm}
\usepackage{multirow}
\usepackage{graphicx}

\newtheorem{thm}{Theorem}[section]
\newtheorem{prop}[thm]{Proposition}
\newtheorem{lem}[thm]{Lemma}
\newtheorem{coro}[thm]{Corollary}
\newtheorem{defn}[thm]{Definition}
\newtheorem{expl}[thm]{Example}
\newtheorem{rmk}[thm]{Remark}

\numberwithin{equation}{section}

\newcommand{\rbb}{\mathbb{R}}
\newcommand{\kbb}{\mathbb{K}}

\newcommand{\pbb}{\mathbb{P}}

\newcommand{\lra}{\longrightarrow}

\newcommand{\mcalk}{\mathcal{K}}

\newcommand{\mcalo}{\mathcal{O}}

\newcommand{\mcalc}{\mathcal{C}}

\newcommand{\om}{\mathcal{O}(M)}

\newcommand{\okm}{\mathcal{O}_k(M)}

\newcommand{\id}{\mbox{id}}

\newcommand{\okmu}{\mathcal{O}^{(k)}(M)}
\newcommand{\mcala}{\mathcal{A}}
\newcommand{\mcalf}{\mathcal{F}}
\newcommand{\mcald}{\mathcal{D}}
\newcommand{\mcalt}{\mathcal{T}}

\newcommand{\mcalp}{\mathcal{P}}
\newcommand{\mcalb}{\mathcal{B}}

\newcommand{\mcalu}{\mathcal{U}}
\newcommand{\mcalv}{\mathcal{V}}
\newcommand{\mcali}{\mathcal{I}}
\newcommand{\mcalr}{\mathcal{R}}

\newcommand{\fb}{\overline{F}}

\newcommand{\bkmu}{\mathcal{B}^{(k)}(M)}
\newcommand{\bku}{\mathcal{B}^{(k)}}
\newcommand{\btku}{\widetilde{\mathcal{B}}^{(k)}}

\newcommand{\fvect}{f\mbox{Vect}_{\mathbb{K}}}
\newcommand{\bomu}{\mathcal{B}^{(1)}(M)}
\newcommand{\bou}{\mathcal{B}^{(1)}}

\newcommand{\vgvb}{VGVB\left(\mathcal{B}^{(1)}(M)\right)}
\newcommand{\vgvbf}{VGVB\left(\mathcal{B}^{(1)}(F_k(M))\right)}
\newcommand{\vgm}{\mathcal{F}\left(\mathcal{B}^{(1)}(M); \fvect\right)}
\newcommand{\vgfm}{\mathcal{F}\left(\mathcal{B}^{(1)}(F_k(M)); \fvect\right)}
\newcommand{\vgms}{\mathcal{F}^*\left(\mathcal{B}^{(1)}(M); \fvect\right)}

\newcommand{\vfi}{\varphi}
\newcommand{\noy}{\mbox{ker}}
\newcommand{\coker}{\mbox{coker}}
\newcommand{\cp}{\mathbb{C}\mathbb{P}}

\newcommand{\tm}{\mathcal{T}^M}

\newcommand{\utm}{\mathcal{U}(\mathcal{T}^M)}
\newcommand{\butm}{\mathcal{B}_{\mathcal{U}(\mathcal{T}^M)}}

\newcommand{\pio}{\pi_1}

\newcommand{\delk}{\Delta[k]}
\newcommand{\delko}{\Delta[k]^{op}}

\newcommand{\ddelko}{\partial\Delta[k]^{op}}
\newcommand{\phib}{\overline{\phi}}

\newcommand{\tkm}{\mathcal{T}^{M}_k}
\newcommand{\tom}{\mathcal{T}^{M}_1}
\newcommand{\ttm}{\mathcal{T}^{M}_2}

\newcommand{\fbkv}{\mathcal{F}(\mathcal{B}^{(k)}(M); \mcalc)}     
\newcommand{\fkov}{\mathcal{F}_k(\mathcal{O}(M); \mcalc)}
\newcommand{\lv}{\langle v \rangle}
\newcommand{\vi}{\langle v_i \rangle}
\newcommand{\vit}{\langle v_{i+1} \rangle}

\newcommand{\vij}{\langle v_i v_j \rangle}
\newcommand{\vji}{\langle v_j v_i \rangle}

\newcommand{\iso}{\text{Iso}}

\newcommand{\hra}{\hookrightarrow}
\newcommand{\lla}{\longleftarrow}
\newcommand{\ssie}{\subseteq_{ie}}
\newcommand{\frk}{F^{!}}
\newcommand{\ok}{\mathcal{O}_k}
\newcommand{\repg}{\text{Rep}_{\mcalc}(G)}
\newcommand{\repi}{\text{Rep}_{\mcalc}(\pio(M))}
\newcommand{\repif}{\text{Rep}_{\mcalc}(\pio(F_k(M)))}
\newcommand{\fpic}{\mcalf\left(\Pi(M); \mcalc\right)}
\newcommand{\fpifc}{\mcalf\left(\Pi(F_k(M)); \mcalc\right)}
\newcommand{\phig}{\Phi_G}
\newcommand{\phigp}{\Phi_{G'}}
\newcommand{\too}{\mathcal{T}^M_0}
\newcommand{\vzot}{\langle v_0v_1v_2\rangle}
\newcommand{\vot}{\langle v_1v_2\rangle}
\newcommand{\vzo}{\langle v_0v_1\rangle}

\newcommand{\vtz}{\langle v_2v_0\rangle}
\newcommand{\vz}{\langle v_0\rangle}
\newcommand{\vo}{\langle v_1\rangle}
\newcommand{\vt}{\langle v_2\rangle}
\newcommand{\phie}{\Phi_{\eta}}
\newcommand{\vzk}{\langle v_0 \cdots v_k\rangle}
\newcommand{\vok}{\langle v_1 \cdots v_k\rangle}

\newcommand{\phisf}{\Phi_{\Psi_F}}

\newcommand{\rhof}{\rho_F}
\newcommand{\rhofp}{\rho_{F'}}

\newcommand{\ftild}{\widetilde{f}}

\newcommand{\aut}{\text{Aut}}
\newcommand{\mcalh}{\mathcal{H}}
\newcommand{\mcals}{\mathcal{S}}

\title{ \textbf{Very good homogeneous functors in manifold calculus}}
\date{}

\author{ Paul Arnaud Songhafouo Tsopm\'en\'e \\
Donald Stanley}

\begin{document}
\maketitle

\begin{abstract}  
Let $M$ be a smooth manifold, and let $\mathcal{O}(M)$ be the poset of open subsets of $M$. Let  $\mcalc$ be a category that has a zero object and all small limits. A homogeneous functor (in the sense of manifold calculus) of degree $k$ from $\om$ to  $\mcalc$ is called \textit{very good} if it sends isotopy equivalences to isomorphisms. In this paper we show that the category $VGHF_k$ of such functors is equivalent to the category of contravariant functors from the fundamental groupoid of $F_k(M)$ to $\mcalc$, where $F_k(M)$ stands for the unordered configuration space of $k$ points in $M$. As a consequence of this result, we show that the category $VGHF_k$ is equivalent to the category of representations of $\pio(F_k(M))$ in $\mcalc$, provided that $F_k(M)$ is connected.  We also introduce a subcategory of vector bundles that we call \textit{very good vector bundles}, and we show that it is abelian, and equivalent to a certain category of very good  functors. 
\end{abstract}

%\tableofcontents

\section{Introduction}

%We establish a connection between two worlds seemingly distant to each other: \textit{manifold calculus} and \textit{representation theory}. We show that a subcategory from the first world is equivalent to the category of representations of a certain group. Before we properly state our results, we briefly recall what is manifold calculus. %and what is representation theory.

Let $M$  and $\om$ as in the abstract. \textit{Manifold calculus}, due to Goodwillie and Weiss \cite{good_weiss99, wei99}, is  a calculus of functors suitable for studying good contravariant functors $F \colon \om \lra \text{Top}$ from $\om$ to the category of spaces. %It is a variation on Goodwillie's calculus \cite{goodwillie03} of homotopy functors.
 %For some history  and a slow introduction to manifold calculus, see \cite{wei96}. 
 The philosophy of calculus of functors is to take a  functor $F$  and replace it by its Taylor tower $\{T_k(F) \lra T_{k-1}(F)\}_{k \geq 1}$, which converges to the original functor in good cases, very much like  the approximation of a function by its Taylor series. Each $T_k(F)$ is called \textit{polynomial approximation} to $F$ of degree $\leq k$. The \lq\lq difference\rq\rq{} $L_kF$ between $T_kF$ and $T_{k-1}F$, or more precisely the homotopy fiber of the canonical map $T_kF \lra T_{k-1}F$,  belongs to a nice class of functors called \textit{homogeneous functors} of degree $k$. %Contrary to polynomial functors whose explicit description is perhaps too much to hope for, homogeneous functors can be  explicitly described. 
In \cite[Theorem 8.5]{wei99}, Weiss proves a deep result about the classification of  homogeneous functors of degree $k$. More precisely, he shows that any such functor is equivalent to a functor $G$ constructed from a fibration over the unordered configuration space $F_k(M)$ of $k$ points in $M$. 
%We will see below that this result is a kind of \lq\lq topological version\rq\rq{} of ours. 

%On the other hand, groups arise in nature as \lq\lq sets of symmetries (of an object), which are closed under composition
%and under taking inverses\rq\rq{}. \textit{Representation theory}  studies the ways in which a given  group $G$ may act on vector spaces. Such actions enable to represent elements of $G$, which could be an abstract group, by more concrete objects (invertible matrices, or linear isomorphisms), and  group multiplication by matrix multiplication.  So representation theory reduces problems in abstract groups to problems in linear algebra, a subject that is well understood. The basic question in that theory is to classify all representations of a given group, up to isomorphisms. For arbitrary $G$ this is very hard! But for finite groups, a very good general theory exists (see for example \cite{serre77}). We are not going to say more about the subject here. For an introduction to the representation theory, we refer the reader to \cite{ful_har91}. 

%It is almost certainly unique, however, among such clearly delineated subjects, in the breadth of its interest to mathematicians.

In this paper we look at the category $\mathcal{F}_k(\mathcal{O}(M); \mcalc)$   of homogeneous functors $F \colon \om \lra \mcalc$, into a \lq\lq nice\rq\rq{} category $\mcalc$, that send isotopy equivalences to isomorphisms. Such functors, which we call \textit{very good}, have been never considered before. Our main result, Theorem~\ref{main_thm_paper} below, roughly classifies objects of  $\mathcal{F}_k(\mathcal{O}(M); \mcalc)$.

\subsection{Statements of the main results and motivation}

%We are interested in very good homogeneous functors because of the following. First of all, w

Let $\fpifc$ denote the category of contravariant functors from the fundamental groupoid $\Pi(F_k(M))$ of $F_k(M)$ to $\mcalc$.  At first glance, this latter category and $\mathcal{F}_k(\mathcal{O}(M); \mcalc)$  appear quite different, but, somewhat miraculously, they turn out to be related. Specifically, we have the following result.  

 %Our main result is a kind of algebraic version of that of  Weiss stated above. Specifically, it says that the category of very good homogeneous functors of degree $k$ is equivalent to the category of representations of $\pio(F_k(M))$ over $\kbb$. Before stating it, we need to introduce two categories. The first, denoted $\mathcal{F}_k(\mathcal{O}(M); \mcalc)$, is the category of very good homogeneous functors  of degree $k$. Given a group $G$, the second  is the  category, denoted $\kgmd$, of representations of $G$ over $\kbb$ defined as follows. An object of $\kgmd$ is a pair $(V, \rho)$ where $V$ is a $\kbb$-vector space, and $\rho \colon G \lra GL(V)$ is a homomorphism of groups (of course $GL(V)$ denotes the usual linear group). A morphism from $(V, \rho)$ to $(V', \rho')$ consists of a linear map $f \colon V \lra V'$ such that for all $(x, v) \in G \times V$, $f(\rho(x)(v)) = \rho'(x)(f(v))$. 

%is a classification of objects of $\mathcal{F}_k(\mathcal{O}(M); \mcalcn)$ as homomorphisms of groups from $\pio(M)$ to $GL_n$. More precisely, let  $GL_n$ denote the usual group of $\kbb$-linear isomorphisms from $\kbbn$ to $\kbbn$, where $\kbb$ stands for the ground field. Consider the set $\homt(\pio(M), GL_n)$ of homomorphisms from $\pio(M)$ to $GL_n$. Also consider the conjugacy relation $\sim$ defined as $f \sim g$ if and only if there is $\varphi \in GL_n$ such that $f(x) = (\varphi)^{-1} g(x) \varphi$ for all $x \in \pio(M)$. Let $\homt(\pio(M), GL_n)\slash \sim$ be the quotient set, and let $F_k(M)$ denote the (non ordered) configuration space of $k$ points in $M$. 

\begin{thm} \label{main_thm_paper}
Let  $\mcalc$ be a category that has a zero object and all small limits. Then the category $\mathcal{F}_k(\mathcal{O}(M); \mcalc)$ of very good homogeneous functors of degree $k$ is equivalent to the category $\fpifc$. That is,
\[
\mathcal{F}_k(\mathcal{O}(M); \mcalc) \simeq \fpifc. 
\]
\end{thm}

This result has a strong consequence. When $F_k(M)$ is connected, the category $\fpifc$ turns out to be deeply related to a certain category of representations that we now recall.  Let $\repg$ denote the following category of representations of a group $G$ in $\mcalc$. An object of $\repg$ is a pair $(A, \rho)$ where $A$ is an object of $\mcalc$, and $\rho \colon G \lra \aut(A)$ is a homomorphism of groups. A morphism from $(A, \rho)$ to $(A', \rho')$ consists of a morphism $\varphi \colon A \lra A'$ in $\mcalc$ such that for all $x \in G$, $\varphi \rho(x) = \rho'(x)\varphi$. 

\begin{coro} \label{main_coro_paper}
 Let $\mcalc$ as in Theorem~\ref{main_thm_paper}. Assume that $F_k(M)$ is connected. Then the category $\mathcal{F}_k(\mathcal{O}(M); \mcalc)$ is equivalent to the category of representations of the fundamental group $\pio(F_k(M))$ in $\mcalc$. That is,
\begin{eqnarray} \label{fk_repif}
\mathcal{F}_k(\mathcal{O}(M); \mcalc) \simeq \repif.
\end{eqnarray}
\end{coro}

\begin{rmk}
Let $G =\pio(F_k(M))$. If $\mcalc = \mcalr\text{-Mod}$, the category of modules over a ring $\mcalr$, then $\repg$ is nothing but the standard category $\mcalr[G]\text{-Mod}$ of modules over the group ring $\mcalr[G]$, that is, the category of representations of $G$ over $\mcalr$.  In that case (\ref{fk_repif}) becomes 
$
\mathcal{F}_k\left(\mathcal{O}(M); \mcalr\text{-Mod}\right) \simeq \mcalr[\pio(F_k(M))]\text{-Mod}. 
$
\end{rmk}

\sloppy
As a quick consequence of Corollary~\ref{main_coro_paper}, if $F_k(M)$ happens to be simply connected then the category $\mathcal{F}_k(\mathcal{O}(M); \mcalc)$  is equivalent to $\mcalc$. In particular the category $\mathcal{F}_1(\mathcal{O}(S^n); \mcalc)$ of very good linear functors  on the $n$-sphere is equivalent to $\mathcal{F}_1(\mcalo(\rbb^n); \mcalc)$ when $n \geq 2$.

We also prove Theorem~\ref{vgm_equivalence_thm}, which roughly states that the category of very good contravariant functors into finite dimensional vector spaces is equivalent to a nice subcategory of vector bundles, which we call \textit{very good vector bundles}(see Definition~\ref{vgvb_defn}). 
%Roughly speaking, a \textit{vgvb} over $M$ is a traditional vector bundle  endowed with an extra structure, which is nothing but a  very good covariant functor $F \colon \bou(M)  \lra \fvect$. This structure is required to satisfy an extra axiom, which is a kind of compatibility between local trivializations. 
We let  VGVB denote the category of such bundles, and we let VB denote the traditional category of vector bundles over $M$. By definition the category VGVB is a subcategory of VB. It is well known that the latter category is  not abelian as there is a technical issue with the existence of all kernels and cokernels.  

%By definition there is a forgetful functor from VGVB to VB, which is not injective on objects (see Remark~\ref{mobius_rmk}). 
%For instance, let $MB_{-1}$ (respectively $MB_{-2}$) be the Mobius bundle whose attaching function $\rbb \lra \rbb$ is the multiplication by $-1$ (respectively $-2$). As objects of VGVB, they are not isomorphic. But as objects of VB, they are. 
%Looking at (\ref{eqn2_overview}) and Theorem~\ref{vgm_equivalence_thm}, one natural question arises: is VGVB equivalent to the category of very good  homogeneous functors? This question has a negative answer in general since the category $\fvect$ does not have all small limits. We also prove the following result (see also Remark~\ref{abelian_rmk}).

\begin{thm} [Theorem~\ref{vbm_abelian_thm}] \label{vgvb_abelian_thm}
The category VGVB of very good vector bundles is abelian. 
\end{thm}

We are working on a project that consists of studying polynomial functors $F \colon \om \lra \text{Ch}_*$ into chain complexes in the setting of triangulated categories \cite{nee01}. Let $\mcalp$ denote the category of such functors, and let $\mcalh \subset \mcalp$ denote the category of homogeneous functors. It turns out  that \cite{paul_don17-4} the associated ``derived'' categories, denoted $\mcald\mcalp$ and $\mcald\mcalh$, are triangulated categories. If $F$ is polynomial of degree $\leq 2$, then it fits into the triangle  $L_2F \lra F \lra T_1F$, where $L_2F$ and $T_1F$ are indeed objects of $\mcald\mcalh$. So by induction on $k$, one can show that every object of $\mcald\mcalp$ can be written as extension of objects of $\mcald\mcalh$. This reduces the study of polynomial functors to the study of homogeneous functors. Moreover, one can show that \cite{paul_don17-4}  the category $\mcald\mcalh$ is generated, in the triangulated categorical language, by the category VGHF of very good homogeneous functors. This explains why we have started our project by an investigation of $\mathcal{F}_k(\mathcal{O}(M); \mcalc)$.

Our Theorem~\ref{vgm_equivalence_thm} might be helpful to establish a connection between homogeneous functors and sheaves. More precisely,   let $\mcals$ denote the category of sheaves on $M$ with values in (finite) dimensional vector spaces, and let $\mcald \mcals$ denote its ``derived'' category, which is a triangulated category. As mentioned above, one has $\text{VGVB} \subseteq \text{VB}$. It is well known that the the natural functor from VB to $\mcals$ turns out to be an embedding functor. So, by using the obvious inclusion functor $\mcals \subseteq \mcald \mcals$, the category VGVB can be viewed as a subcategory of $\mcald \mcals$. We claim that $\mcald \mcals$ is generated by VGVB. If this is true, then one natural question arises:   since the category $\mcald \mcalh$ is generated by VGHF \cite{paul_don17-4},  and since $\text{VGHF} \simeq \text{VGVB}$ by Theorem~\ref{vgm_equivalence_thm}, one may ask the question to know whether the categories $\mcald \mcals$ and  $\mcald \mcalh$ are equivalent or how they are related. That question is interesting and will be addressed in \cite{paul_don17-4} as well as the claim of course.

\subsection{Overview of the proof of Theorem~\ref{main_thm_paper}}

We first need some notation. For a subposet $\mcala \subseteq \om$, we let $\mcalf(\mcala; \mcalc)$ denote the category of very good contravariant functors $F \colon \mcala \lra \mcalc$. We let $\okm$denote the subposet of $\om$ consisting of open subsets diffeomorphic to the disjoint union of at most $k$ open balls. Let $\mcalb(M)$ be a basis (consisting of open subsets diffeomorphic to a ball) for the topology of $M$. We let $\mcalb^{(k)}(M) \subseteq \om$ denote the subposet whose objects are exactly the disjoint union of $k$ elements of $\mcalb(M)$, and whose morphisms are isotopy equivalences.

The proof goes through three steps in which all of our constructions  are explicit.  

(1) The first thing we need is Theorem~\ref{equiv_cat_thm}, which states that any very good homogeneous functor $F \colon \om \lra \mcalc$ of degree $k$ is determined by its values on $\bkmu$. More precisely, the category $\mcalf_k(\om; \mcalc)$ is equivalent to $\mcalf(\bkmu; \mcalc)$. That is,
\begin{eqnarray} \label{eqn1_overview}
\mcalf_k\left(\om; \mcalc\right) \simeq \mcalf\left(\bkmu; \mcalc\right). 
\end{eqnarray}
%A \lq\lq topological version\rq\rq{} of (\ref{eqn1_overview}) was obtained by Pryor in \cite{pryor15} for polynomial good (see \cite[Section 1]{wei99} for the notion of goodness) contravariant functors $F \colon \om \lra \text{Top}$. In fact, Weiss \cite[Theorem 4.1]{wei99} shows that  any polynomial functor of degree $\leq k$ is characterized by its restriction to $\okm$. Using essentially the same techniques, Pryor generalizes this result by showing that one can replace $\okm$ by a smaller subposet $\mcalb_k(M) \subseteq \okm$, consisting of open subsets diffeomorphic to the disjoint union of at most $k$ elements of $\mcalb(M)$, and still recover the same notion of polynomial functor. 
To prove (\ref{eqn1_overview}), we essentially use the right Kan extension functor $\text{Ran}_i(-)$ along the inclusion $i \colon \bkmu \hra \om$, and show that it is the \lq\lq inverse\rq\rq{} for the restriction functor. One of the key points is to prove that $\text{Ran}_i(-)$ preserves the very goodness property. To do this, 
%we take a very different approach than that of Weiss-Pryor. Our idea is as follows. 
%we proceed as follows. Let $F \colon \bkmu \lra \mcalc$ be very good. Given two  objects $B, B' \in \bkmu$  that are isotopic via an isotopy $L$, we construct an isomorphism  from $F(B')$ to $F(B)$, and show that it is independent of the choice of $L$. To formalize this 
we introduce the concept of \textit{admissible family} of open balls (see Definition~\ref{admissible_def}) \footnote{In the follow up paper \cite[Theorem 1.4]{paul_don17-2}, we use that concept to prove a certain result about the homotopy right Kan extension of good functors.}. As an example, the family $\{B, B_1, A, B_2, B'\}$ from Figure~\ref{adm_family} is admissible. Associated with that family is the isomorphism
\begin{eqnarray} \label{iso_intro}
(F(i_1))^{-1} F(i_2)(F(i_3))^{-1} F(i_4) \colon F(B') \lra F(B),
\end{eqnarray} 
where $i_1, i_2, i_3$, and $i_4$ fit into the poset $\xymatrix{B & B_1 \ar[l]_-{i_1} \ar[r]^-{i_2}  & A  &  B_2   \ar[l]_-{i_3} \ar[r]^-{i_4} & B'}$ of inclusions. 

\begin{figure}[ht!]
\centering
\includegraphics[scale =0.9]{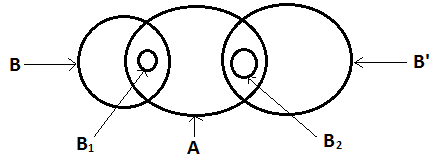}
\caption{An example of an admissible family of open balls (here $k=1$)}  \label{adm_family}
\end{figure}

Isomorphisms like (\ref{iso_intro}) play a crucial role here.   Along the way we use the fact that the category $\mcalc$ has a zero object and all small limits. Note that those requirements about $\mcalc$  are used only in this step.

Equation (\ref{eqn1_overview}) has two nice consequences. If $\bou(F_k(M)) \subseteq \mcalo(F_k(M))$ denotes the subposet whose objects are exactly the product of $k$ elements of $\mcalb(M)$, then $\bou(F_k(M)) \cong \bkmu$. And therefore (\ref{eqn1_overview}) becomes 
\begin{eqnarray} \label{eqn2_overview}
\mcalf_k\left(\om; \mcalc\right) \simeq \mcalf\left(\bou(F_k(M)); \mcalc\right). 
\end{eqnarray}
This latter equation and (\ref{eqn1_overview}) imply that the category of very good homogeneous functors of degree $k$ is equivalent to the category of linear functors $\mcalo(F_k(M)) \lra \mcalc$. (In the follow up paper \cite[Theorem 1.3]{paul_don17-2} we show that the same result holds for good homogeneous functors.) So it is enough to work with $k =1$. The second consequence, which will be used in the next step, is the fact that the righthand side of (\ref{eqn2_overview}) does not depend on the choice of the basis $\mcalb(M)$ for the topology of $M$.

(2) Let $\tm$ be a triangulation of $M$, that is, a simplicial complex homeomorphic to $M$. By first taking two barycentric subdivisions of $\tm$, and then the interior $U_{\sigma}$ of the star of each simplex $\sigma$, we obtain a subposet $\utm:= \{U_{\sigma}\}_{\sigma \in \tm} \subseteq \om$, which covers  $M$. Such a poset is a \textit{very good cover} (see Definition~\ref{gc_defn} and Proposition~\ref{vg_cover_prop}).  It is in particular what Weiss calls a \textit{good $1$-cover}. To any very good cover, one can associate a basis $\mcalb_{\utm}$ for the topology of $M$:
\[
\mcalb_{\utm} = \left\{B \text{ diffeomorphic to an open ball}| \ B \subseteq U_{\sigma} \text{ for some $U_{\sigma} \in \utm$}\right\}. 
\]  
Taking $\bou(M) := \mcalb_{\utm}$, the second thing we need in proving Theorem~\ref{main_thm_paper} is  the following equivalence (Proposition~\ref{equiv_cat_vectn_prop})
\begin{eqnarray} \label{eqn3_overview}
\mcalf\left(\bou(M); \mcalc\right) \simeq \mcalf\left(\utm; \mcalc\right).
\end{eqnarray}
The key point in the proof of (\ref{eqn3_overview}) is the fact that the poset $\utm$ turns out to be a very good cover of $M$. 
%where $\bou(M)$ is a distinguished basis for the topology of $M$, and $\utm \subseteq \om$ is the following subposet. Let $\tm$ be a triangulation of $M$, and let $S(\tm)$ be the set of simplicies of $\tm$. The poset $\utm$ is a collection $\{U_{\sigma}\}_{\sigma \in S(\tm)}$ of open subsets diffeomorphic to a  ball whose each $U_{\sigma}$ is a \lq\lq suitable thickening\rq\rq{}  of $\sigma$ in such a way that the intersection $U_{\sigma} \cap U_{\lambda}$ is either the emptyset or $U_{\sigma \cap \lambda}$. Such a poset is called \textit{very good cover} (see Definition~\ref{gc_defn}). A very good cover is in particular what Weiss calls a \textit{good $1$-cover}. To any very good cover, one can associate a basis $\mcalb_{\utm}$ for the topology of $M$:
%\[
%\mcalb_{\utm} = \left\{B \text{ diffeomorphic to an open ball}| \ B \subseteq U_{\sigma} \text{ for some $U_{\sigma} \in \utm$}\right\}. 
%\] 
%In fact, in the equation (\ref{eqn3_overview}), the poset $\bou(M)$ is nothing but $\mcalb_{\utm}$. 

(3) Lastly, we need the following equivalence of categories (Theorem~\ref{crucial_thm})
\begin{eqnarray} \label{eqn4_overview}
\mcalf\left(\utm; \mcalc\right) \simeq \fpic.
\end{eqnarray}
To prove (\ref{eqn4_overview}), we first construct a functor $\Psi \colon \mcalf\left(\utm; \mcalc\right) \lra \fpic$
as follows. First of all, it is well known that the fundamental groupoid $\Pi(M)$ can be viewed as the category whose objects are vertices of $\tm$, and whose morphisms are homotopy classes of \textit{edge-paths}. (Recall that an \textit{edge-path}  is a chain $f = (v_0, \cdots, v_r)$ of vertices connected by edges in $\tm$.) Now define  $\Psi(F)(v) := F(U_{\lv})$. On morphisms $f$ of $\Pi(M)$, we explain the idea of the definition of $\Psi(F)$ through the following example.  Consider  the triangulation of $M=S^1$ with $three$ vertices ($v_0, v_1$, and $v_2$) and three edges ($\langle v_0v_1\rangle, \langle v_1v_2 \rangle$, and $\langle v_0v_2 \rangle$).  Let $f = (v_0, v_1, v_2)$ be an edge-path from $v_0$ to $v_2$.  
%\begin{figure}[ht!]
%\centering
%\includegraphics[scale =0.9]{triangle}
%\caption{An example of $\tm$}  \label{triangle}
%\end{figure}
Associated with $f$ is the natural poset $\mcalu_f$ as shown (\ref{poset_f}). 
\begin{eqnarray} \label{poset_f}
\mcalu_f = \left\{\xymatrix{U_{\vz} \ar[r]^-{i_1}  & U_{\vzo}  & U_{\vo} \ar[l]_-{i_2} \ar[r]^-{i_3} & U_{\vot} & U_{\vt} \ar[l]_-{i_4} }  \right\}.
\end{eqnarray}
The isomorphism  $\Psi(F)(f) \colon F(U_{\vz}) \lla F(U_{\vt})$ is then defined as the \lq\lq composition along $\mcalu_f$\rq\rq{}. That is, 
$
\Psi(F)(f) := F(i_1)(F(i_2))^{-1}F(i_3)(F(i_4))^{-1}. 
$
On morphisms $\eta \colon F \lra F'$ of $\mcalf\left(\utm; \mcalc\right)$,  we define $\Psi$ as the component of $\eta$ at $U_{\lv}$. 

We also construct a functor $\Phi \colon \fpic \lra \mcalf\left(\utm; \mcalc\right)$, and show that it is the \lq\lq inverse\rq\rq{} for $\Psi$. Given $G \in \fpic$, the idea of the construction of $\Phi (G)$ is to proceed by induction on the skeletons of $\tm$ (see Subsection~\ref{construction_phi_subsection}). %The starting point is to define $\Phi$  on the $1$-skeleton, and then extend the definition on the rest of the triangulation. 
%A good summary of this construction is provided just after the proof of Proposition~\ref{extension_prop}. 

Combining now  (\ref{eqn2_overview}), (\ref{eqn3_overview}), and (\ref{eqn4_overview}), and after replacing $M$ by $F_k(M)$ in   (\ref{eqn3_overview}) and (\ref{eqn4_overview}), we deduce Theorem~\ref{main_thm_paper}.

\subsection{Outline of the paper}

In Section~\ref{notation_section} we fix some notation.  

In Section~\ref{vghf_section} we first define some basic concepts. Next we introduce the notion of \textit{admissible family} in Subsection~\ref{admissible_family_subsection}. The next  subsection defines an important isomorphism out of an isotopy and some other data. A typical example of that isomorphism is given by (\ref{iso_intro}). We show in Subsection~\ref{dependence_subsection} that it does not depend on the choice of the isotopy. Lastly, 
we prove (\ref{eqn1_overview}) or Theorem~\ref{equiv_cat_thm}   at the end of Subsection~\ref{equiv_cat_thm_proof_subsection}. 

In Section~\ref{vgf_section} we first prove (\ref{eqn3_overview}) or Proposition~\ref{equiv_cat_vectn_prop}. Next, in Subsection~\ref{construction_psi_subsection}, we contruct  the functor $\Psi$, while the functor $\Phi$  is constructed in Subsection~\ref{construction_phi_subsection} as mentioned before. In the last subsection, we  prove (\ref{eqn4_overview}) or Theorem~\ref{crucial_thm}, and  Theorem~\ref{main_thm_paper}. 

In Section~\ref{vghf_rep_section} we first prove Proposition~\ref{fpim_rep_prop}, which says that the category $\fpic$ is equivalent to the category of representations of $\pio(M)$ in $\mcalc$ provided that $M$ is connected. Then, combining Theorem~\ref{main_thm_paper} and Proposition~\ref{fpim_rep_prop}, we deduce Corollary~\ref{main_coro_paper}. 

In  Section~\ref{vgvb_section} we first introduce the notion of \textit{very good vector bundles} (vgvb), and provide two examples. In the next subsection, we prove Theorem~\ref{vgvb_abelian_thm} or more precisely  Theorem~\ref{vbm_abelian_thm}. In Subsection~\ref{vgvb_equivalence_subsection}, we prove Theorem~\ref{vgm_equivalence_thm}, which says that the category of vgvb is equivalent to the category of very good contravariant functors into finite dimensional vector spaces.

\textbf{Acknowledgements.} This work has been supported by  Pacific Institute for the Mathematical Sciences (PIMS) and the University of Regina, that the authors acknowledge.

\section{Setup of notation}     \label{notation_section}

In this section we fix some notation.  

$\bullet$ A smooth manifold $M$ will be fixed. As in Weiss's work \cite{wei99}, we let $\om$ denote the poset of open subsets of $M$, morphisms being inclusions. 
%An \textit{open ball} in $M$ means an open subset diffeomorphic to the unit open ball of $\rbb^n$, where $n$ stands for the dimension of $M$.  

$\bullet$ For $k \geq 0$,  we let $\mathcal{O}_k(M) \subseteq \om$ denote  the full subcategory of $\om$ whose objects are open subsets diffeomorphic to the disjoint union of at most $k$ balls. 

$\bullet$ For subsets $A$ and $B$ of $M$ such that $A \subseteq B$, we let $AB \colon A \hra B$ denote the inclusion map. 

$\bullet$ Given two objects $U, V \in \om$ such that $U \subseteq V$,  we  use the notation $U \ssie V$ to mean that the inclusion of $U$ inside $V$ is an isotopy equivalence.   

%We let $I =[0, 1]$.

$\bullet$ We let $F_k(M)$ denote the unordered configuration space of $k$ points in $M$. 

$\bullet$ We let $\mcalt^M$ denote a triangulation of $M$ with the maximal tree denoted $m\mcalt^M$. We write $\mcalt^M_p$ for the $p$-skeleton of $M$. %Recall that a \textit{triangulation} of $M$ is a simplicial complex homeomorphic to $M$. For the computation of the fundamental group $\pio(M) \cong \pio(\tm)$, we take the basepoint, $v$, to be a vertex of $\tm$. 

$\bullet$ An $r$-simplex of $\tm$ generated by vertices $v_0, \cdots, v_r$ is denoted $\langle v_0 \cdots v_r\rangle$. If $r =0$, we will sometimes write $v_0$ for $\langle v_0\rangle$. 

$\bullet$ We let $\fvect$ denote the category of finite dimensional vector spaces over a field $\kbb$. 
 
$\bullet$ Our functors are  contravariant unless stated otherwise.

$\bullet$ If  $\beta \colon F \lra G$ is a natural transformation, we denote by $\beta[A] \colon F(A) \lra G(A)$  the component of $\beta$ at $A$.  

$\bullet$ If $F$ and $G$ are two functors, we will use the notation $F \cong G$ to mean that $F$ is naturally isomorphic to $G$. 

$\bullet$ A category $\mcalc$ that has a zero object, denoted $0$, and all small limits will be fixed. 

$\bullet$ If $\mcala \subseteq \om$ is a subposet, we let $\mcalf(\mcala; \mcalc)$ denote the category of \textit{very good} contravariant functors (see Definition~\ref{verygood_defn}) from $\mcala$ to $\mcalc$. 

$\bullet$ We use the notation $\mcala \simeq \mcalb$ to mean that a category $\mcala$ is equivalent to another category $\mcalb$.

%We use the notation $x := \text{def}$ to state that the left hand side is defined by the right
%hand side.   

\section{Characterization of very good homogeneous functors} \label{vghf_section}

%The philosophy of the manifold calculus, due to Goodwillie and Weiss \cite{good_weiss99, wei99}, is to take a \lq\lq good\rq\rq{} contravariant functor $F \colon \om \lra \text{Top}$ (from the poset $\om$ of open subsets of $M$ to spaces) and replace it by its Taylor tower $\{T_k(F)\}_{k \geq 0}$, where each $T_k(F)$ is the polynomial approximation of $F$ of degree $\leq k$. So the idea is to compute $F$ by its successive polynomial approximations.  

In this section, we show that similar results to those of Weiss-Pryor \cite{wei99, pryor15} hold for very good homogeneous functors (see Definition~\ref{verygood_defn} below). Specifically, we prove Theorem~\ref{equiv_cat_thm}, which states that very good homogeneous functors of degree $k$ are determined by their restriction to the subposet $\mcalb^{(k)}(M) \subseteq \om$ (see Definition~\ref{bku_oku_defn} below).  
%whose objects are exactly the disjoint union $k$ elements of $\mcalb(M)$, morphisms being  isotopy equivalences. 
As a consequence, we prove Corollary~\ref{equiv_cat_coro}, which states that the category of very good homogeneous functors  $\om \lra \mcalc$ of degree $k$ is equivalent to the category of linear functors $\mcalo(F_k(M)) \lra \mcalc$. 

\subsection{Definition of basic concepts}  \label{definition_subsection}

\sloppy

We define the notion of \textit{very good functor} and that of \textit{very good homogeneous functor} of degree $k$. We  also state the main result, Theorem~\ref{equiv_cat_thm}, of the section.

\begin{defn} \label{isotopy_defn}
\begin{enumerate}
\item[(i)] Let $i  \colon U  \hookrightarrow W$ and $i' \colon U' \hra W$ be morphisms of $\om$. We say that $i$ is \emph{isotopic} to $i'$ (or that $U$ is \emph{isotopic} to $U'$) if there exists a continuous map $L \colon U \times [0, 1] \lra W,$  $(x, t) \mapsto L_t(x)$ satisfying the following three conditions: (a) $L_0 = i$, (b) $L_1(U) = U'$, and (c) for all $t$, $L_t \colon U \lra W$ is a smooth embedding. Such a map $L$ is called an \emph{isotopy from $U$ to $U'$}. 
  %\begin{enumerate}
%\item[(a)] $L_0 = i$;
%\item[(b)] $L_1(U) = U'$;
%\item[(c)] for all $t$, $L_t \colon U \lra W$ is a smooth embedding.
%\end{enumerate}
%Such a map $L$ is called an \emph{isotopy from $U$ to $U'$}. 
\item[(ii)]  An inclusion $i \colon U \hra W$ in $M$  is said to be an \emph{isotopy equivalence}, and we denote $U \ssie W$, if $i$ is isotopic to the identity $id \colon W \hra W$. 
\end{enumerate}
\end{defn}

The following  well known  result will be extensively used in this paper.    

\begin{prop}\cite[Chapter 8]{hirsch76}  \label{isotopy_equiv_prop} (i) If  $i  \colon U  \hookrightarrow W$ is isotopic to $i' \colon U' \hra W$, then $\pi_0(U) \cong \pi_0(U')$. 
(ii) Let $U, W \in \om$ be diffeomorphic to a disjoint union of open balls such that $U \subseteq W$. Then the inclusion map  $i \colon U \hra W$ is an isotopy equivalence if and only if the induced map $\pi_0(i)$ is an isomorphism. 
\end{prop}

\begin{defn} \label{verygood_defn} 
For a subcategory $\mcala \subseteq \om$,  a contravariant functor  $F \colon \mcala \lra \mcalc$ is called  \emph{very good} if it sends isotopy equivalences to isomorphisms. 
 \end{defn}

\begin{defn} \label{polynomial_defn}
Let $F \colon \om \lra \mcalc$ be a contravariant functor. 
\begin{enumerate} 
\item[(i)] We say that $F$ is \emph{polynomial of degree $\leq k$} if it is determined by its values on $\okm$. That is, for any $U \in \om$, 
$
F(U) \cong \emph{lim}_{V \in \mathcal{O}_k(U)} F(V). 
$
%for any $V$ in $\om$, for any $A_0, A_1, \cdots, A_k \subset V$ pairwise disjoint nonempty closed subsets, the natural map $F(V) \lra \emph{lim}_{S \neq \emptyset} F(V \setminus \cup_{i \in S} A_i)$ is an isomorphism ($S$ runs over the power set of $\{0, 1, \cdots, k\}$). 
\item[(ii)]  The $k$th \emph{polynomial approximation} of $F$, denoted $T_kF$, is the contravariant functor $T_kF \colon \om \lra \mcalc$ defined as 
$T_kF(U) = \emph{lim}_{V \in \mathcal{O}_k(U)} F(V).$  
\end{enumerate}     
\end{defn}

\begin{defn} \label{homoge_funct_defn}
 (i) A \emph{very good homogeneous functor of degree $k$} is a contravariant functor $F \colon \om \lra \mcalc$ that satisfies the following three conditions. (a) $F$ is very good; (b) $F$ is polynomial of degree $\leq k$; (c) $T_{k-1}F(U)$ is isomorphic to $0$ for all $U$. 
(ii) A \emph{linear functor} is a homogeneous functor of degree $1$. 			
\end{defn}

 For the next definition, we let $\mcalb(M)$ denote a basis, that consists of open subsets diffeomorphic to a ball, for the topology of $M$.

\begin{defn} \label{bku_oku_defn} 
\begin{enumerate}
\item[(i)] Define $\mcalb^{(k)}(M), k \geq 0,$ to be the subcategory of $\om$ whose objects are exactly the disjoint union of  $k$ objects of $\mcalb(M)$, and whose morphisms are isotopy equivalences. 
\item[(ii)] Define $\mcalo^{(k)}(M), k \geq 0,$ to be the subcategory of $\om$ whose  objects are diffeomorphic to the disjoint union of exactly $k$ open balls,  morphisms being isotopy equivalences. 
\end{enumerate}
\end{defn}

By definition $\mcalb^{(k)}(M)$ is a full subcategory of $\mcalo^{(k)}(M)$. To point out the difference between the objects of $\mcalb^{(k)}(M)$ and of $\mcalo^{(k)}(M)$, let us consider the following example. Take $M$ to be a smooth codimension zero submanifold of $\rbb^N$, and take $\mcalb(M)$ to be the subsets of $M$ consisting of open balls (with respect to the euclidean metric). Then an object of $\mcalb^{(k)}(M)$ is exactly the disjoint union of $k$ genuine open balls, while an object of $\mcalo^{(k)}(M)$ is diffeomorphic to the disjoint union of $k$ open balls.  

\begin{defn} \label{cat_vghf_defn}
Define $\fkov$  as the category of very good homogeneous functors $\om \lra \mcalc$ of degree $k$. Also define $\fbkv$ to be the category of very good  functors $\mcalb^{(k)}(M) \lra \mcalc$.
\end{defn}
 
%The following is the main result of the whole section.   

\begin{thm} \label{equiv_cat_thm} Let $\mcalc$ be a category that has a zero object and all small limits.  Then
the category $\fkov$ of very good homogeneous functors of degree $k$ is equivalent to the category $\fbkv$. That is,
$
 \fkov \simeq \fbkv.
$
\end{thm}

We will prove Theorem~\ref{equiv_cat_thm} in Subsection~\ref{equiv_cat_thm_proof_subsection}.  We first need to introduce some terminology. Also we need  to establish a certain amount of intermediate results.

%To prove Theorem~\ref{equiv_cat_thm} we will need seven lemmas, the seventh being a consequence of the others. In order to state them, we need to first introduce some notations.

\subsection{Admissible family of open subsets}  \label{admissible_family_subsection}

We introduce the concept of \textit{admissible family} (see Definition~\ref{admissible_def} below), which is crucial for the paper. We also  derive a couple of results (Propositions~\ref{existence_prop}, \ref{kk_admissible_prop}) that will be used in next subsections.  For this subsection, consider the following data: $W \in \om$, $U, V \in \mcalo_k(W)$, and $L \colon U \times [0, 1] \lra W$ is an isotopy from $U$ to $V$. 
%For subsets $A$ and $B$ of $M$ such that $A \subseteq B$, we let  $AB \colon A \hookrightarrow B$ denote the inclusion map. Let $k \geq 1$.  For this subsection, we consider the following data: 
%\begin{enumerate}
%\item[$\bullet$] $W \in \om$;
%\item[$\bullet$] $U, V \in \mcalo_k(W)$;
%\item[$\bullet$] $L \colon U \times [0, 1] \lra W$ is an isotopy from $U$ to $V$.
%\end{enumerate} 

Recall the notation \lq\lq $\ssie$\rq\rq{} from Definition~\ref{isotopy_defn}-(ii). Also recall the notation \lq\lq AB\rq\rq{} from Section~\ref{notation_section}.

\begin{defn} \label{admissible_def}
Let $K \subseteq U$ be a nonempty compact subset such that $\pi_0(KU)$ is surjective. A family $a =\{a_0, a_1, \cdots, a_m, a_{m+1}\} \subseteq [0, 1]$ such that $a_0 = 0, a_{m+1} =1$, and $a_i \leq a_{i+1}, 0 \leq i \leq m,$ is called \emph{admissible} with respect  to $\{L, K\}$ (or just \emph{admissible}) if  there exists a collection 
$
\mcalu_a = \{U_{01}, U_{12}, \cdots, U_{m(m+1)}\}
$
of objects of $\okm$ such that for all $i$, for all $s \in [a_i, a_{i+1}]$, one has
\begin{eqnarray} \label{admissible_cond}
 L_s(K) \subseteq U_{i(i+1)} \ssie L_s(U),
\end{eqnarray}
and
\begin{eqnarray} \label{admissible_cond2}
\overline{U_{i(i+1)}} \subseteq L_s(U),
\end{eqnarray}
where $\overline{U}_{i(i+1)}$ stands for the closure of $U_{i(i+1)}$.  Such a collection $\mcalu_a$ is said to be \emph{$\{a, K, L\}$-admissible}. 
\end{defn}

\begin{prop} \label{existence_prop}
Let $K$ as in Definition~\ref{admissible_def}. Then there exists an  admissible family   $a = \{a_0, \cdots, a_{m+1}\}$ with respect to $\{K, L\}$. 
\end{prop}

To prove Proposition~\ref{existence_prop} we will need two lemmas, the first  being  a matter of point-set topology. 

\begin{lem} \label{existence1_lem}
Let $X$ be a Hausdorff  space, and let $K, K'$ be two compact spaces. Let $f \colon K \lra X$, and $g \colon K' \lra X$ be continuous such that $f(K) \cap g(K) = \emptyset$. Consider a homotopy $H \colon K \times [0, 1] \lra X$ such that $H_t =f$ for some $t \in [0, 1]$. Then there exists $\epsilon > 0$ such that for all $s \in (t-\epsilon, t+\epsilon), \ H_s(K) \cap g(K') = \emptyset$. 
%\[
%\forall s \in (t-\epsilon, t+\epsilon), \ H_s(K) \cap g(K') = \emptyset. 
%\]
\end{lem}

%\begin{proof}
%Let $p \colon K \times [0, 1] \lra [0, 1]$ be the projection on the second component, and let
%\[
%A = p(H^{-1}(g(K'))).
%\]
%Then $A$ is closed by the compactness of $K, K'$, and by the fact that  $p, H,$ and $g$ are continuous. Furthermore, one can see that
%\[
%A = \{s \in [0, 1] | \ H_s(K) \cap g(K') \neq \emptyset\}. 
%\]
%So $[0, 1] \backslash A$ is open, and it contains $t$ since $H_t =f$ and $f(K) \cap g(K) = \emptyset$ by assumption. This implies the existence of $\epsilon > 0$ such that $(t-\epsilon, t+\epsilon) \subseteq [0, 1] \backslash A$. We thus get  the desired result. 
%\end{proof}

\begin{lem} \label{existence2_lem}
Assume that $U$ is diffeomorphic to an open ball, and let $j \colon U \hra W$ be the inclusion map. Let $K \subseteq U$ as before. Consider an isotopy $H \colon U \times [0, 1] \lra W$ such that $H_t = j$ for some $t \in [0, 1]$. Then there exist $\epsilon > 0$ and  $V_t$ diffeomorphic to an open ball such that  for all $s \in (t-\epsilon, t+\epsilon)$, we have
$H_s(K) \subseteq V_t \subseteq H_s(U)$, and $\overline{V_t} \subseteq H_s(U)$.  
\end{lem}

\begin{proof}
Let $n$ be the dimension of $M$. For $r > 0$, we let 
$
B_r = \{x \in \rbb^n| \ \left\|x\right\| < r\} \quad \text{and} \quad S_r = \overline{B}_r \backslash B_r.
$
Consider a diffeomorphism $\theta \colon B_1 \lra U$. Since $K$ is compact, there exists $\delta > 0$ such that $K \subseteq \theta(B_{\delta})$. Also consider the inclusion maps 
$
f \colon \theta\left(S_{\frac{1+\delta}{2}}\right) \hra W  \mbox{ and } g \colon \theta(S_{\delta}) \hra W.
$
Clearly, one has 
$
f\left(\theta\left(S_{\frac{1+\delta}{2}}\right)\right) \cap g\left(\theta(S_{\delta})\right) = \emptyset.
$
So by applying Lemma~\ref{existence1_lem}, there is $\epsilon' > 0$ such that
$
\forall s \in (t-\epsilon', t+\epsilon'), \ H_s\left(\theta\left(S_{\frac{1+\delta}{2}}\right)\right) \cap \theta(S_{\delta}) = \emptyset.  
$
This implies 
\begin{eqnarray} \label{seps1_eqn}
\forall s \in (t-\epsilon', t+\epsilon'), \ \theta (B_{\delta}) \subseteq H_s\left(\theta\left(B_{\frac{1+\delta}{2}}\right)\right) \subseteq H_s(\theta(B_1)) = H_s(U), 
\end{eqnarray}
and 
\begin{eqnarray} \label{seps1p_eqn}
\forall s \in (t-\epsilon', t+\epsilon'), \ \overline{\theta (B_{\delta})} \subseteq H_s\left(\theta\left(B_{\frac{1+\delta}{2}}\right)\right) \subseteq H_s(\theta(B_1)) = H_s(U). 
\end{eqnarray}
Similarly, by Lemma~\ref{existence1_lem}, there exsits $\epsilon'' > 0$ such that $H_s(K) \cap \theta(S_{\delta}) = \emptyset$ for all $s \in (t-\epsilon'', t+\epsilon'')$. This implies 
\begin{eqnarray} \label{seps2_eqn}
\forall s \in (t-\epsilon'', t+\epsilon''), \ H_s(K) \subseteq \theta (B_{\delta}). 
\end{eqnarray}
Letting
$
\epsilon = \text{min}(\epsilon', \epsilon'') \mbox{ and } V_t = \theta(B_{\delta}),
$
the desired result follows from (\ref{seps1_eqn}), (\ref{seps1p_eqn}) and (\ref{seps2_eqn}).
\end{proof}

\begin{proof}[Proof of Proposition~\ref{existence_prop}]
Let $p$ denote the number of connected components of $U$. Since $U \in \okm$, for $1 \leq r \leq p$ each component, $U^r$, of  $U$ is diffeomorphic to an open ball.  Let 
$
K^r = K \cap U^r,
$
and let $t \in [0, 1]$. ($K^r \neq \emptyset$ since by assumption $\pi_0(KU)$ is surjective.) By Lemma~\ref{existence2_lem}, there exist $\epsilon > 0$ and $V^r_t$ diffeomorphic to an open ball such that for all $s \in (t-\epsilon, t+\epsilon)$,
$
L_s(K^r) \subseteq V^r_t \subseteq L_s(U^r), \text{ and } \overline{V^r_t} \subseteq L_s(U^r).
$
Varying $t$, we get an open cover $\{(t-\epsilon, t+\epsilon)\}_t$ of $[0, 1]$. Using now the compactness of $[0, 1]$, there exists a finite family $a^r = \{a_0^r, \cdots, a_{m_r+1}^r \}$ such that  $a_0^r = 0, a^r_{m_r+1} = 1$,  $a_i \leq a^r_{i+1}$ for all $i$, and  each interval $[a^r_i, a^r_{i+1}]$ is contained in one of the open subsets of the cover. By construction, such a family $a^r$ comes indeed together with a collection 
$
\{U^r_{i(i+1)}\}_{i = 0}^{m_r} \subseteq \{V^r_t\}_t
$
that satisfies (\ref{admissible_cond}). So $a^r$ is admissible with respect to $\{K^r, L\}$. Defining
$
a := \cup_{r = 1}^p a^r, 
$
one can see that $a$ is admissible with respect to $\{\cup_{r = 1}^p K^r, L\} = \{K, L\}$. 
\end{proof}

\begin{prop} \label{kk_admissible_prop}
\begin{enumerate}
\item[(i)] Let $T$ is a finite subset of $[0, 1]$.   If $\{a_0, \cdots, a_{m+1}\}$ is admissible with respect to $\{L, K\}$, then so is $\{a_0, \cdots, a_{m+1}\} \cup T$  with respect to the same set. 
\item[(ii)] Let $K$ and $K'$ be two nonempty compact subsets of $U$ such that $K \subseteq K'$. If a family $\{a_0, \cdots, a_{m+1}\}$ is admissible with respect to $\{L,  K'\}$, then it is also admissible with respect to $\{L, K\}$. 
\end{enumerate}
\end{prop}

\begin{proof}
(i) By induction on the cardinality of $T$. If $T = \{t\}$ for some $t \in [0, 1]$, then there exists $0 \leq j \leq m$ such that $t \in [a_j, a_{j+1}]$. Let $b = a \cup \{t\}$, and define $\mcalu_b = \{U'_{01}, \cdots, U'_{(m+1)(m+2)}\}$ as 
\begin{eqnarray} \label{family_uii}
U'_{i(i+1)} = \left\{ \begin{array}{ccc}
                     U_{i(i+1)} & \text{if}  & i \leq j-1 \\
										 U_{j(j+1)} & \text{if} & i \in \{j, j+1\} \\
										 U_{(i-1)i} & \text{if}  & i \geq j+2, 
                     \end{array} \right. \quad  0 \leq i \leq m+1.
\end{eqnarray}
Manifestly $\mcalu_b$ satisfies (\ref{admissible_cond}), which proves the base case. The inductive step is handled in the same way  as the base case. 
(ii) This follows directly from the definition. 
%Assume that $a = \{a_0, \cdots, a_{m+1}\}$ is admissible with respect to $\{L, K'\}$. Then there exists a collection $\mcalu_a$, which is $\{a, K', L\}$-admissible. Since  $K \subseteq K'$, it follows that  $L_s(K) \subseteq L_s(K')$ for any $s$. Therefore the same collection $\mcalu_a$ is also $\{a, K, L\}$-admissible. 
\end{proof}

\subsection{The isomorphism $\iso(\mcalu_a, a, K, L)$} \label{isomorphism_isouaakl_subsection}

We continue to use the same data ($W, U, V, L$) as in Subection~\ref{admissible_family_subsection}. The very first goal here is  to define an important isomorphism, $\iso(\mcalu_a, a, K, L) \colon F(U) \lla F(V)$, out of an admissible family and a very good functor $F \colon \okm \lra \mcalc$.  Next we  show that   this isomorphism is independent of the choice of $\mcalu_a, a, $ and $K$ in Lemmas~\ref{uii_upii_lem}, \ref{isoab_isoab}, and \ref{admissible_kkp_lem} respectively. These lemmas are part of ingredients needed for the proof of Theorem~\ref{equiv_cat_thm}.   

%In the next subsection, we will show that  under certain conditions $\iso(\mcalu_a, a, K, L)$ does not depend on the choice of $L$ as well.    

Let $c(U)$ denote the number of components of $U$. Of course $c(U) = c(V)$ since $U$ is isotopic to $V$. Let  $\{U^1, \cdots, U^{c(U)}\}$ be the set of components of $U$. 
%Note that each $U^r$ is diffeomorphic to an open ball by Definition~\ref{polynomial_defn}-(i) of $\okm$.  
For each $1 \leq r \leq c(U)$, we let $K^r \subseteq U^r$ denote a nonempty compact subset, and $K := \cup_{r =1}^{c(U)} K^r$. Also let $a = \{a_0, \cdots, a_{m+1}\}$ be an admissible family with respect to $\{K, L\}$, and   $F \colon \okm \lra \mcalc$ be a  very good  functor. We want to define an isomorphism  $F(U) \stackrel{\cong}{\longleftarrow} F(V)$ out of these data. First of all, for each $0 \leq i \leq m+1$, define $U_i$ as the image of $U$ under $L_{a_i}$. That is,
$
U_i :=  L_{a_i}(U).
$
Clearly, one has $U_0 = U$ and $U_{m+1} = V$. Since $a$ is admissible, there is a collection $\{U_{i(i+1)}\}_{i=0}^m$ which is $\{a, K, L\}$-admissible in the sense of Definition~\ref{admissible_def}. 
%\begin{defn} \label{aluvwk_family_defn}
%For each $1 \leq r \leq c(U)$, for every $0 \leq i \leq m$, choose $U_{i(i+1)}^r \subseteq U_i^r \cap U_{i+1}^r$ such that 
%\begin{enumerate}
%\item[(a)] $U^r_{i(i+1)}$ is diffeomorphic to an open ball, and 
%\item[(b)] $U^r_{i(i+1)}$ lies in the same path-component as $L_{a_i}(K^r)$, which is viewed as a subspace of the intersection $U_i^r \cap U_{i+1}^r$ \footnote{Notice that $U_i^r \cap U_{i+1}^r$ is not necessarily path-connected}. 
%\end{enumerate}
%Set $U_{i(i+1)} = \cup_{r =1}^p U^r_{i(i+1)}$. Such a collection $\{U_{i(i+1)}\}_{i=0}^m$ is called a \emph{$\{a, L, U, V, W, K\}-$family}.
%\end{defn}
Certainly, for every $i$, the inclusions $U_{i(i+1)}U_i \colon U_{i(i+1)} \hra U_i$  and $U_{i(i+1)}U_{i+1} \colon U_{i(i+1)} \hra U_{i+1}$ are isotopy equivalences by (\ref{admissible_cond}). So  $F(U_{i(i+1)}U_i)$ and $F(U_{i(i+1)}U_{i+1})$ are both isomorphisms for all $i$ since $F$ is very good.  Applying now $F$ to the zigzag 
\[
\xymatrix{U_0  & U_{01} \ar[l] \ar[r]  & U_1 & \cdots \ar[l] \ar[r]  & U_m & U_{m(m+1)} \ar[l] \ar[r]  & U_{m+1},}
\]
we get the following  diagram of isomorphisms.
\[
\xymatrix{F(U_0) \ar[r]  & F(U_{01})   & F(U_1) \ar[l] \ar[r] &  \cdots & F(U_m) \ar[l] \ar[r] & F(U_{m(m+1)})  & F(U_{m+1}) \ar[l]}
\]
%in which every morphism is of course a linear isomorphism. 

\begin{defn} \label{iso_luvwk_defn}
Define 
$
\iso(\mcalu_a, a, K, L) \colon F(U) \lla F(V)
$
as the composition
\begin{align*}
\iso(\mcalu_a, a, K, L) & = & \\
(F(U_{01}U_0))^{-1} \cdots (F(U_{i(i+1)}U_i))^{-1} F(U_{i(i+1)}U_{i+1}) \cdots F(U_{m(m+1)}U_{m+1}). &  &
\end{align*}
\end{defn}

\sloppy

\begin{lem} \label{uii_upii_lem}
Let $\mcalu'_a = \{U'_{i(i+1)}\}_{i=0}^m$ be another collection which is $\{a, K, L\}$-admissible. Then  
$
\iso(\mcalu_a, a, K, L) = \iso(\mcalu'_a, a, K, L)
$
\end{lem}

\begin{proof}
Since 
$
L_{a_i}(K) \subseteq U_{i(i+1)} \text{ and } L_{a_i}(K) \subseteq U'_{i(i+1)}, 0 \leq i \leq m,
$
by (\ref{admissible_cond}) it follows that $L_{a_i}(K) \subseteq U_{i(i+1)} \cap U'_{i(i+1)}$. Since $K$ is nonempty, there exists $D_{i(i+1)} \in \okm$ contained in $L_{a_i}(K)$ such that 
$
D_{i(i+1)} \ssie U_{i(i+1)} \quad \text{and} \quad D_{i(i+1)} \ssie U'_{i(i+1)}.
$
Consider now the following commutative diagram of isotopy equivalences. 
\[
\xymatrix{U_0 \ar@{=}[d] & U_{01} \ar[l] \ar[r]  & U_1 \ar@{=}[d] & \cdots \ar[l] & U_{m(m+1)} \ar[l] \ar[r] & U_{m+1} \ar@{=}[d] \\
U_0 \ar@{=}[d] & D_{01} \ar[l] \ar[r] \ar[u] \ar[d] & U_1 \ar@{=}[d] & \cdots \ar[l] & D_{m(m+1)} \ar[l] \ar[r] \ar[u] \ar[d] & U_{m+1} \ar@{=}[d]\\
U_0  & U'_{01} \ar[l] \ar[r] &  U_1 & \cdots \ar[l]  & U'_{m(m+1)} \ar[l] \ar[r]  & U_{m+1}.}  
\]
By applying $F$ to it, and by using the fact that each square of the resulting diagram commutes, and the fact that every vertical map $F(U_i) \lra F(U_i)$ is the identity, we get 
$
\iso(\mcalu_a, a, K, L) = \iso(\mcalu'_a, a, K, L),
$
which is the required result.  
\end{proof}

\begin{lem} \label{isoa_isoba}
Let $T \subseteq [0, 1]$ be a finite set, and let $a = \{a_0, \cdots, a_{m+1}\}$ be admissible with respect to $\{K, L\}$. Consider another family $b = a \cup T$, which is indeed admissible with respect to $\{K, L\}$ by Proposition~\ref{kk_admissible_prop}-(i). Then for any collection $\mcalu_b$ which is $\{b, K, L\}$-admissible, one has
$
\iso(\mcalu_a, a, K, L) =  \iso(\mcalu_b, b, K, L). 
$
\end{lem}

\begin{proof}
By induction on the cardinality of $T$.  Assume that $T = \{t\}$ for some $t \in [0, 1]$. Then there exists $j \in \{0, \cdots, m\}$ such that $a_j \leq t \leq a_{j+1}$. Since $\iso(\mcalu_b, b, K, L)$ does not depend on the choice of $\mcalu_b$ by Lemma~\ref{uii_upii_lem}, to make things easier, we take $\mcalu_b = \{U'_{i(i+1)}\}_{i=0}^{m+1}$ as defined in (\ref{family_uii}). Note that the inclusions $U'_{j(j+1)} \hra L_t(U)$ and $U'_{(j+1)(j+2)} \hra L_t(U)$ coincide since 
$U'_{j(j+1)} = U'_{(j+1)(j+2)} = U_{j(j+1)}$
 by (\ref{family_uii}). This implies that $F(U'_{j(j+1)}L_t(U)) \circ \left(F(U'_{(j+1)(j+2)}L_t(U))\right)^{-1} = id.$
Using Definition~\ref{iso_luvwk_defn} and this latter equation, one can easily see that $\iso(\mcalu_b, b, K, L) = \iso(\mcalu_a, a, K, L)$, which proves the base case. 
%Recalling now Definition~\ref{iso_luvwk_defn}, one has 
%\begin{align*}
%\iso(\mcalu_b, b, K, L)  & = &  \\
%\cdots \left(F(U'_{j(j+1)}U_j)\right)^{-1} \underbrace{F(U'_{j(j+1)}L_t(U)) \left(F(U'_{(j+1)(j+2)}L_t(U))\right)^{-1}}_{id} F(U'_{(j+1)(j+2)}U_{j+1}) \cdots  & = & \\
%\cdots \left(F(U_{j(j+1)}U_j)\right)^{-1} F(U_{j(j+1)}U_{j+1}) \cdots & = & \\
%\iso(\mcalu_a, a, K, L). 
%\end{align*}
%This proves the base case. 
The inductive step works in the same way as the base case. 
\end{proof}

The following lemma follows directly from Lemma~\ref{isoa_isoba}.  %says that $\iso(\mcalu_a, a, K, L)$ does not depend on the choice of an admissible family with respect to $\{K, L\}$.  

\begin{lem}  \label{isoab_isoab}
Let $a = \{a_0, \cdots, a_{m+1}\}$ and $b = \{b_0, \cdots, b_{n+1}\}$ be admissible with respect to the same set $\{K, L\}$. Then 
$
\iso(\mcalu_a, a, K, L) = \iso(\mcalu_b, b, K, L). 
$
\end{lem}

%\begin{proof}  
%It immediately follows from Lemma~\ref{isoa_isoba}. 
%\end{proof}

%We end this subsection with the following lemma that states that $\iso(\mcalu_a, a, K, L)$ is also independent of the choice of $K$. 

\begin{lem} \label{admissible_kkp_lem}
Let $K$ and $K'$ be nonempty compact subsets of $U$ such that $\pi_0(KU)$ and $\pi_0(K'U)$ are both surjective. Let $a$ (respectively $b$) be admissible with respect to $\{K, L\}$ (respectively $\{K', L\}$). Then one has
$
\iso(\mcalu_a, a, K, L) = \iso(\mcalu_{a'}, a', K', L). 
$
\end{lem}

\begin{proof}
By Proposition~\ref{existence_prop}, there exists an admissible family $b$ with respect to $\{K \cup K', L\}$. Let 
$
c := b \cup a \cup a'.
$
Certainly $c$ is admissible with respect to $\{K \cup K', L\}$ by Proposition~\ref{kk_admissible_prop}~-(i). Again by the same proposition (but part (ii)), $c$ is admissible with respect to both $\{K, L\}$ and $\{K', L\}$. So, by Definition~\ref{iso_luvwk_defn}, one has
$
\iso(\mcalu_c, c, K \cup K', L) = \iso(\mcalu_c, c, K, L)$ and  $\iso(\mcalu_c, c, K \cup K', L) = \iso(\mcalu_c, c, K', L).$
Moreover, by Lemma~\ref{isoa_isoba}, one has 
$
\iso(\mcalu_c, c, K, L) = \iso(\mcalu_a, a, K, L)$ and  $\iso(\mcalu_c, c, K', L) = \iso(\mcalu_{a'}, a', K', L).$
Combining these equations, we get the desired result. 
\end{proof}

\subsection{Dependence of $\iso(\mcalu_a, a, K, L)$ on the choice of the isotopy $L$}  \label{dependence_subsection}

In Subsection~\ref{isomorphism_isouaakl_subsection}, we showed that the isomorphism $\iso(\mcalu_a, a, K, L)$ introduced in Definition~\ref{iso_luvwk_defn} does not depend on $\mcalu_a, a,$ and $K$. Here the goal is to prove Proposition~\ref{isotopy_invariance_prop}, which says that under certain conditions $\iso(\mcalu_a, a, K, L)$ is independent of the choice of $L$ as well. This result is one of the key ingredients in proving Theorem~\ref{equiv_cat_thm}.

Let $\mcalb(M)$ as in Subsection~\ref{definition_subsection}. Recall the categories $\okmu$ and $\bku(M)$ from Definition~\ref{bku_oku_defn}. We continue to use the same data as in Subsection~\ref{admissible_family_subsection} except that here we make the following restrictions:  $W \in \okmu$,  $U, V \in \bku(M)$ are such that the inclusions $U \hra W$ and $V \hra W$ are both isotopy equivalences,  $L \colon U \times [0, 1] \lra W$ is an isotopy from $U$ to $V$ such that for all $t$, $L_t(U) \in \mcalb(M)$, and  $F \colon \bku(M) \lra \mcalc$ is a very good  functor.

%As in the previous subsection,  $W = \cup_{r=1}^k W^r$ where $W^r$ is a component of $W$ (which is indeed diffeomorphic to an open ball). Similarly we have $U = \cup_{i=1}^k U^r$ and $V = \cup_{r=1}^k V^r$ where  $U^r, V^r \in \mcalb(M)$. Since $U \hra W$ and $V \hra W$ are isotopy equivalences, it follows (by Proposition~\ref{isotopy_equiv_prop})  that $U^r \subseteq W^r$ and $V^r \subseteq W^r$ for every $r$. 

%Given a compact subset $K = \cup_{r=1}^k K^r$ of $U$ such that $K$ is path-connected and $K^r \subseteq U^r, 1 \leq r \leq k$, we showed in Lemma~\ref{isoab_isoab} that the isomorphism 
%\[
%\iso[a, L, U, V, W, K] \colon F(U) \lla F(V)
%\]
%does not depend on the choice of an admissible family $a = \{a_0, \cdots, a_{m+1}\}$. Because of this, we will sometimes drop the letter \lq\lq $a$ \rq\rq{} in the notation $\iso[a, L, U, V, W, K]$. That is,
%\[
%\iso[L, U, V, W, K] := \iso[a, L, U, V, W, K].
%\]
%In this subsection, the goal is to prove that the same isomorphism is also independent of the choice of the isotopy $L$ from $U$ to $V$ in $W$ (see Lemma~\ref{isotopy_invariance_lem} below). 

To state  Proposition~\ref{isotopy_invariance_prop}, we need to make a definition.  As in the preceding subsection, let $\{U^r\}_{r = 1}^k$, $\{V^r\}_{r =1}^k$, and   $\{W^r\}_{r = 1}^k$ denote the set of components of $U, V,$ and $W$ respectively. For each $1 \leq r \leq k$, let $x_r \in U^r$ be a point, and consider the compact set $K = x := \{x_1, \cdots, x_k\}$.  Let  $\lambda^r_{Lx} \colon [0, 1] \lra W^r$ be the path defined as  $\lambda^r_{Lx}(t) := L(x_r, t)$,
and let $\lambda_{Lx} := \{\lambda^r_{Lx}\}_r$. Consider an admissible family $a =\{a_0, \cdots, a_{m+1}\}$  with respect to $\{x, L\}$. By (\ref{admissible_cond}), there exists a collection $\mcalu_{a} = \{U_{i(i+1)}\}$ such that 
\begin{eqnarray} \label{lambda_lxai}
\lambda^r_{Lx} (a_i) = L_{a_i}(x_r)   \in U_{i(i+1)}, 1 \leq i \leq m
\end{eqnarray}

\begin{rmk} \label{uii_rmk}
Looking closer at the proof of Lemma~\ref{existence2_lem}, and using the fact that $\mcalb(M)$ is a basis for the  topology of $M$, one can always assume that each $U^r_{i(i+1)}$ is an object of $\mcalb(M)$ since each component, $K^r$, of $K$ is a single point here. 
\end{rmk} 

By Lemmas~\ref{uii_upii_lem}, \ref{isoab_isoab}, \ref{admissible_kkp_lem}, the isomorphism $\iso(\mcalu_{a}, a, x, L)$ is independent of the choice of $\mcalu_{a}, a,$  and $x$ so that one can rewrite it just in term of $L$.   

\begin{defn} \label{iso_lambda_defn}
Define $\iso(\lambda_{Lx}) = \iso(\lambda_L) \colon F(U) \lla F(V)$ as 
$
 \iso(\lambda_L) := \iso(\mcalu_{a}, a, x, L). 
$
\end{defn}

%The following is the main result of this subsection, and it says that $\iso(\lambda_L)$ does not depend on $L$. 

\begin{prop} \label{isotopy_invariance_prop}
Let $L' \colon U \times [0, 1] \lra W$ be another isotopy from $U$ to $V$ such that $L'_t(U) \in \mcalb(M)$ for all $t$. Then 
$
\iso(\lambda_L) = \iso(\lambda_{L'}).
$
\end{prop}

To  prove this result, we need two  lemmas. 
%\begin{defn} \label{gamma_family_defn}
%A \emph{$\gamma$-family} is a collection  
%\[
%\{A_i, A_{j(j+1)}, \ 0 \leq i \leq m+1, \ 0 \leq j \leq m\}
%\] 
%of objects of $\bku(M)$ that satisfy the following four conditions:
%\begin{enumerate}
%\item[(a)] For any $r, i, j$,  $A^r_i \subseteq W^r$ and $A^r_{j(j+1)} \subseteq W^r$; 
%\item[(b)] $A_0 = U$ and $A_{m+1} = V$;
%\item[(c)] $A^r_{i(i+1)} \subseteq A^r_i \cap A^r_{i+1}, 1 \leq r \leq k, 0 \leq i \leq m$;
%\item[(d)] For each $r$, the collection $\{A_i^r, A_{j(j+1)}^r\}_{i, j}$ forms a cover of $\gamma^r$, and $A_{j(j+1)}^r \cap \gamma^r \neq \emptyset$.  
%\end{enumerate} 
%\end{defn}
%As an example, the collection $\{U_i, U_{j(j+1)}\}_{ij}$, where $U_i = L_{a_i}(U)$, is a $\lambda_{Lx}$-family by (\ref{lambda_lxai}) and Remark~\ref{uii_rmk}. To any $\gamma$-family one can associate an isomorphism from $F(V)$ to $F(U)$ denoted 
%\[
%\iso(A_0, A_{01}, \cdots, A_{m(m+1)}, A_{m+1}) \quad \text{or just} \quad  \iso(\gamma),
%\]
%  and define as 
%\begin{eqnarray}  \label{fao_iso}
%\begin{array}{ccc}
%\iso(\gamma) = \iso(A_0, A_{01}, \cdots, A_{m(m+1)}, A_{m+1}) & = &   \\
% (F(A_{01}A_0))^{-1} \cdots (F(A_{i(i+1)}A_i))^{-1} F(A_{i(i+1)}A_{i+1}) \cdots F(A_{m(m+1)}A_{m+1}). &  & 
% \end{array}
%\end{eqnarray}
%Notice that  $\iso(\gamma)$ is defined in the same way as the isomorphism from Definition~\ref{iso_luvwk_defn}.  
\begin{lem} \label{abcd_lem}
Let $A, B, C, D \in \bku(M)$ such that for any $1 \leq r \leq k$, $C^r \subseteq A^r \cap B^r$, and $D^r \subseteq A^r \cap B^r$. Suppose there is a path $\beta^r \colon [0, 1] \lra A^r \cap B^r$ such that $\beta^r(0) \in C^r$ and $\beta^r(1) \in D^r$. Then $\iso(A, C, B) = \iso(A, D, B)$, where $\iso(A, C, B) \colon F(A) \lla F(B)$ is defined as 
\begin{eqnarray} \label{iso_acb}
\iso(A, C, B) = (F(CA))^{-1} \circ F(CB). 
\end{eqnarray} 
\end{lem}

\begin{proof}
By the compactness of $[0, 1]$, there exists a collection $\{X_i^r\}_{i=0}^{n+1}$ of objects of $\mcalb(M)$ such that  $n$ is independent of $r$, $X^r_i \subseteq A^r \cap B^r$ for all $r, i$,  $X^r_0 = C^r, X^r_{n+1} = D^r, $ and $X^r_i \cap X^r_{i+1} \neq \emptyset$ for any $r, i$, and  the family $\{X_i^r\}_{i=0}^{n+1}$ forms an open cover of $\beta^r$ for all $r$. 
By the fact that  $\mcalb(M)$ is a basis for the topology of $M$, there exists another collection $\{X^r_{i(i+1)}\}_{i=0}^n$ of objects of $\mcalb(M)$ such that $X^r_{i(i+1)} \subseteq X^r_i \cap X^r_{i+1}$ for all  $r, i$.  Define 
$
X_i := \cup_{r=1}^k X^r_i,  \text{ and } X_{i(i+1)} := \cup_{r=1}^k X^r_{i(i+1)},
$
and consider the following commutative diagram of isotopy equivalences. 
\[
\xymatrix{A \ar@{=}[r]  & A \ar@{=}[r]  & A  & \cdots  \ar@{=}[l] \ar@{=}[r] & A & A \ar@{=}[r] \ar[l] & A  \\
   C \ar[u] \ar[d]  &  X_{01} \ar[l] \ar[r] \ar[u] \ar[d]  &  X_1 \ar[u] \ar[d] & \cdots \ar[l] \ar[r] \ar@{.>}[u] \ar@{.>}[d] & X_n \ar[u] \ar[d] & X_{n(n+1)} \ar[u] \ar[d] \ar[l] \ar[r] &  D \ar[u] \ar[d] \\
	B \ar@{=}[r]  & B \ar@{=}[r] & B   & \cdots  \ar@{=}[l]  \ar@{=}[r] & B \ar@{=}[r] & B \ar@{=}[r]  & B.}
\]
Applying $F$ to it, and using the fact that each square of the resulting diagram commutes, we get the desired result.  
\end{proof}

One can extend the definition of $\iso(\lambda_L)$ to any path (which does not necessarily come from an isotopy) in the following way. 

\begin{defn} \label{gammac_adm_defn}
For $1 \leq r \leq k$, let $\gamma^r \colon [0, 1] \lra W^r$ be a path such that $\gamma^r(0) \in U^r$ and $\gamma^r(1) \in V^r$, and let $\gamma := \{\gamma^r\}_r$. Let $c = \{c_0, \cdots, c_{m+1}\}$ be a family of points of $[0, 1]$ such that $c_0 = 0, c_{m+1} = 1, $ and $c_i \leq c_{i+1}$ for all $i$. A collection 
$
\{A_i, \ A_{j(j+1)}, \ 0 \leq i \leq m+1, \ 0 \leq j \leq m\}
$
of objects of $\bku(M)$ is called \emph{admissible with respect to $\{\gamma, c\}$} or just \emph{$\{\gamma, c\}$-admissible} if the following four conditions are satisfied for all $r$: (a)  $A^r_0 = U^r$ and $A^r_{m+1} = V^r$; (b) $A^r_{i(i+1)} \subseteq  A^r_i \cap A^r_{i+1}$; (c) $\gamma^r([c_{i-1}, c_i]) \subseteq A^r_i$;  (d) $\gamma^r(c_i) \in A^r_{i(i+1)}$.
\end{defn}
 
\begin{expl}
Consider the collection of paths $\lambda_{Lx} = \{\lambda^r_{Lx}\}_r$  and  the family $\mcalu_{a} = \{U_{i(i+1)}\}$ as before. If we set $U_i = L_{a_i}(U)$, then one can easily check that the collection 
$
\{U_i, U_{j(j+1)}\}_{i, j}
$
is $\{\lambda_{Lx}, a_x\}$-admissible. 
\end{expl}

Associated with an $\{\gamma, c\}$-admissible family $\{A_i, A_{j(j+1)}\}_{i, j}$ is an  isomorphism from $F(V)$ to $F(U)$ denoted 
$\iso\left(\gamma, c, \{A_i\}_{i=0}^{m+1}, \{A_{i(i+1)}\}_{i=0}^m\right)$  or just  $ \iso(\gamma, c, A_i, A_{i(i+1)})$ and defined as 
\begin{align}
\iso(\gamma, c, A_i, A_{i(i+1)}) & := &    (F(A_{01}A_0))^{-1} \cdots (F(A_{i(i+1)}A_i))^{-1} F(A_{i(i+1)}A_{i+1}) \cdots F(A_{m(m+1)}A_{m+1})  \label{fao_iso}. 
\end{align}
Notice that  $\iso(\gamma, c, A_i, A_{i(i+1)})$ is defined in the same way as the isomorphism from Definition~\ref{iso_luvwk_defn}.  
The following lemma says that $\iso(\gamma, c, A_i, A_{i(i+1)})$ does not depend on the choice of $\{c, A_i, A_{i(i+1)}\}$. 

\sloppy

\begin{lem} \label{isoaii_lem}
\begin{enumerate}
\item[(i)] If 
$
\{A_i, A_{j(j+1)}\}_{i, j} \text{ and } \{A_i, A'_{j(j+1)}\}_{i, j}
$
are both $\{\gamma, c\}$-admissible, then one has
$
\iso(\gamma, c, A_i, A_{i+1}) =  \iso(\gamma, c, A_i, A'_{i(i+1)}).
$
\item[(ii)]  If 
$
\{A_i, A_{j(j+1)}\}_{i, j}  \text{ and } \{B_i, B_{j(j+1)}\}_{i, j}
$
are both admissible with respect to $\{\gamma, c\}$, then 
$
\iso(\gamma, c, A_i, A_{i+1}) =  \iso(\gamma, c, B_i, B_{i+1}).
$
\item[(iii)] Let $T \subseteq [0, 1]$ be finite. Let $\{A_i, A_{j(j+1)}\}_{i, j}$ be an $\{\gamma, c\}$-admissible family. Let $\{B_s, B_{(l(l+1))}\}_{s, l}$ be admissible with respect to $\{\gamma, c \cup T\}$. Then we have 
$
\iso(\gamma, c, A_i, A_{i+1}) =  \iso(\gamma, c \cup T, B_s, B_{s(s+1)}).
$
\item[(iv)] Let $d = \{d_0, \cdots, d_{n+1}\}$ be another family of points of $[0, 1]$ such that $d_0 = 0, d_{n+1} = 1$, and $d_i \leq d_{i+1}$ for all $i$. If $\{A_i, A_{j(j+1)}\}_{i, j}$ is $\{\gamma, c\}$-admissible, and if $\{B_s, B_{l(l+1)}\}_{s, l}$ is $\{\gamma, d\}$-admissible, then 
$
\iso(\gamma, c, A_i, A_{i+1}) =  \iso(\gamma, d, B_s, B_{s(s+1)}).
$
\end{enumerate}
\end{lem}

\begin{proof}
The proof of (i) works exactly in the same way as that of Lemma~\ref{uii_upii_lem}. For (iii), its proof follows  from (i) and (ii) by  induction on the cardinality of $T$ (this is similar to the proof of Lemma~\ref{isoa_isoba}). The last part, (iv), is an immediate consequence of (iii). Now we prove the second part.  For all $r, i$ one has $A^r_{i(i+1)} \cap B^r_{i(i+1)} \neq \emptyset$ by the property (d) above. So, since $\mcalb(M)$ is a basis for the topology of $M$, there exists $C^r_{i(i+1)} \in \mcalb(M)$  such that  $\gamma^r(c_i) \in C^r_{i(i+1)}$ and 
$
C^r_{i(i+1)} \subseteq A^r_{i(i+1)} \cap B^r_{i(i+1)}.  
$
%Applying $F$ to the inclusions $A_{i(i+1)} \hookrightarrow A_i$, $A_{i(i+1)} \hookrightarrow A_{i+1}$, $B_{i(i+1)} \hookrightarrow B_i$, $B_{i(i+1)} \hookrightarrow B_{i+1}$, $C_{i(i+1)} \hookrightarrow A_{i(i+1)}$, $C_{i(i+1)} \hookrightarrow B_{i(i+1)}$, which are isotopy equivalences of course, and 
Using Lemma~\ref{abcd_lem}, we get the following three equations. 
$\iso(A_i, A_{i(i+1)}, A_{i+1}) = \iso(A_i, C_{i(i+1)}, A_{i+1})$,
and
$\iso(B_i, B_{i(i+1)}, B_{i+1}) = \iso(B_i, C_{i(i+1)}, B_{i+1})$ for  $0 \leq i \leq m$, 
and $\iso(A_i, C_{(i-1)i}, B_i) = \iso(A_i, C_{i(i+1)}, B_i)$, for $1 \leq i \leq m$. 
The desired result easily follows from those equations. 
%If $m =0$, the desired result easily follows from the first two equations. If $m \geq 1$, it follows from all those equations.
%We need another equation, which is obtained as follows. First of all, for all $r$, $1 \leq i \leq m$ one has $C^r_{(i-1)i} \subseteq A^r_i \cap B^r_i$ by (\ref{eqn0_iso}) and (b).  For the same reason, we have $C^r_{i(i+1)} \subseteq A^r_i \cap B^r_i$. Now, by (c) $\gamma^r([c_{i-1}, c_i])$ is a path in the intersection $A^r_i \cap B^r_i$ from $\gamma^r(c_{i-1}) \in C^r_{(i-1)i}$ to $\gamma^r(c_i) \in C^r_{i(i+1)}$. Applying now Lemma~\ref{abcd_lem}, we get 
%\begin{eqnarray}  \label{eqn3_iso}
%\iso(A_i, C_{(i-1)i}, B_i) = \iso(A_i, C_{i(i+1)}, B_i), \ 1 \leq i \leq m. 
%\end{eqnarray}
%Combining (\ref{eqn1_iso}), (\ref{eqn2_iso}), (\ref{eqn3_iso}), and the fact that the top and bottom arrows from the diagram are both the identities, we obtain the desired result \footnote{Notice that (\ref{eqn3_iso}) holds when $m \geq 1$. The case $m =0$ is readily obtained from (\ref{eqn1_iso}) and (\ref{eqn2_iso})}.  
\end{proof}

We can now prove the main result of this subsection. 

\begin{proof}[Proof of Proposition~\ref{isotopy_invariance_prop}]
%Let $x = \{x_1, \cdots, x_k\}$ be a configuration of $k$ points such that $x_r \in U^r, 1 \leq r \leq k$. Let $a_x = \{a_0, \cdots, a_{m+1}\}$ be admissible (in the sense of Definition~\ref{admissible_def}) with respect to  $\{x, L\}$. One can then consider the paths $\lambda^r_{Lx} \colon [0, 1] \lra W^r, 1 \leq r \leq k$,  as before (see (\ref{lambdar_lx})). By Definition~\ref{iso_lambda_defn}, one has  
%\[
%\iso(\lambda_L) = \iso(\lambda_{Lx})  = \iso(\mcalu_{a_x}, a_x, x, L). 
%\]
%In similar fashion, given another isotopy  $L' \colon U \times [0, 1] \lra W$ such that $L'_t(U) \in \mcalb(M)$ for all $t$, one has the isomorphism 
%\[
%\iso(\lambda_{L'}) = \iso(\lambda_{L'x})  = \iso(\mcalu_{a'_x}, a'_x, x, L).  
%\]
Our goal is to show that  $\iso(\lambda_L) = \iso(\lambda_{L'})$. 
%First of all, note that the paths $\lambda^r_{Lx}$ and $\lambda^r_{L'x}$  have the same starting point since 
%$
%\lambda^r_{Lx}(0) = \lambda^r_{L'x}(0) = x_r.
%$
%However,  $\lambda^r_{Lx}$ and $\lambda^r_{L'x}, 1 \leq r \leq k$ does not necessarily have the same ending point. But we can assume without loss of generality that they have because of the path-connectedness of every $V^r$.  
Since $W^r$ is diffeomorphic to an open ball, and then contractible, there exists a homotopy 
$
H^r \colon [0, 1] \times [0, 1] \lra W^r,  (s, t) \mapsto H_t^r(s)
$
such that $H^r_0 = \lambda^r_{Lx}$ and $H^r_1  = \lambda^r_{L'x}$. Let $H_t := \{H^r_t\}_{r=1}^k$.  Certainly the data involved in the definition of  $\iso(\lambda_L)$ determine a family $(U_0, U_{01}, \cdots, U_{m(m+1)}, U_{m+1})$, which is $\{\lambda_{Lx}, a\}$-admissible.    So
$
\iso(\lambda_L) = \iso(H_0, a, U_s, U_{s(s+1)})$ and $\iso(\lambda_{L'}) = \iso(H_1, a', U'_l, U'_{l(l+1)}),
$
where $a'$ is an admissible family with respect to $\{x, L'\}$. By the compactness of $[0, 1] \times [0, 1]$, for each $1\leq r \leq k$, there exist  a  subdivision  $\left\{[c_{i-1}, c_i] \times [d_j, d_{j+1}]\right\}_{i, j}$, $1 \leq i \leq m_r+1$, $0 \leq j \leq n_r$, of the rectangle $[0,1] \times [0, 1]$ (of course $c_0 = d_0 =0,$ and $c_{m_r+1} = d_{n_r+1} =1$), and   a family  $\{A^{rq}_p\}_{p, q}$, $0 \leq p \leq m_r+1$,  $0 \leq q \leq n_r+1$ of objects of $B(W^r)$ that satisfy the following two conditions:
(a)  for all $q$, $A_0^{rq} = U^r$ and $A_{m_r+1}^{rq} = V^r$; (b) $H^r\left([c_{i-1}, c_i] \times [d_j, d_{j+1}]\right) \subseteq A_i^{rj}$ for any $i, j$.
%The second condition implies that for any $i, j$, one has
%\begin{eqnarray} \label{observation}
%H_{d_j}^r(c_i) \in A_i^{rj} \cap A_{i+1}^{rj} \quad \text{and} \quad H_{d_{j+1}}^r(c_i) \in A_i^{rj} \cap A_{i+1}^{rj}.
%\end{eqnarray}
Since the objects of $B(M)$ form a basis for that topology of $M$, there exists $A^{rj}_{i(i+1)} \in B(W^r)$ such that 
$A_{i(i+1)}^{rj} \subseteq A_i^{rj} \cap A_{i+1}^{rj},$   $A_{i(i+1)}^{r(j+1)} \subseteq A_i^{rj} \cap A_{i+1}^{rj},$ and  $H^r_{d_j}(c_i) \in A_{i(i+1)}^{rj}$.
We can assume that all $m_r$'s are equal since one can refine the subdivision as many times as possible.  The same assumption can be made for $n_r$. So from now on, we will write $m$ for $m_r$ and $n$ for $n_r$. 
%We can assume that $m_r$ and $n_r$ are both independent of $r$ (one can refine the subdivision if needed). So from now on, we let $m_r = m$ and $n_r =n$. 
As usual, let 
$
A_i^j := \cup_{r =1}^k A_i^{rj}, \text{ and } A_{i(i+1)}^j := \cup_{r =1}^{k} A_{i(i+1)}^{rj}.  
$
Also let $c = \{c_0, \cdots, c_{m+1}\}$.  Clearly, for each $0 \leq j \leq n+1$, the collection 
$
\left\{A_i^j, A^j_{l(l+1)}\right\}_{i, j}$ with $0 \leq i \leq m+1$  and $0 \leq l \leq m$, is $\{H_{d_j}, c\}$-admissible, and one has the equation %Let $0 \leq j \leq n$. The rest of the proof is devoted to show that 
\begin{eqnarray} \label{iso_hdjc}
\iso\left(H_{d_j},c, A^j_i, A^j_{i(i+1)}\right) = \iso\left(H_{d_{j+1}},c, A^{j+1}_i, A^{j+1}_{i(i+1)}\right),
\end{eqnarray} 
which is obtained  by combining the equations
$\iso\left(A_i^j, A_{i(i+1)}^j, A_{i+1}^j\right) = \iso\left(A_i^j, A_{i(i+1)}^{j+1}, A_{i+1}^j\right)$, $0 \leq i \leq m,$
$\iso\left(A_i^j, A_{(i-1)i}^{j+1}, A_i^{j+1}\right) = \iso\left(A_i^j, A_{i(i+1)}^{j+1}, A_i^{j+1}\right),$ $1 \leq i \leq m,$ 
$\iso\left(A_0^j, A_{01}^{j+1}, A_0^{j+1}\right) = id$, and  $\iso\left(A_{m+1}^j, A_{m(m+1)}^{j+1}, A_{m+1}^{j+1}\right) =id.$
The first two equations come from Lemma~\ref{abcd_lem}, while the other ones come from condition (a) above. Now the desired result follows from  (\ref{iso_hdjc}) and Lemma~\ref{isoaii_lem}-(iv). 
%Combining now (\ref{iso1_eqnn}), (\ref{iso2_eqnn}), (\ref{iso3_eqnn}), and using the definition of $\iso(A_0^j, A_{01}^j, \cdots, A_{m(m+1)}^j, A_{m+1}^j)$  from (\ref{fao_iso}), it is straightforward to get (\ref{iso_hdjc}),
%which holds for  any $0 \leq j \leq n$. Thus 
%\[
%\iso\left(H_0,c, A^0_i, A^0_{i(i+1)}\right) = \iso\left(H_1,c, A^{n+1}_i, A^{n+1}_{i(i+1)}\right).
%\] 
%Since
%\[
%\iso\left(H_0,c, A^0_i, A^0_{i(i+1)}\right) = \iso\left(H_0, a_x, U_s, U_{s(s+1)}\right),
%\]
%and  
%\[
%\iso\left(H_1,c, A^{n+1}_i, A^{n+1}_{i(i+1)}\right) = \iso\left(H_1, a'_x, U'_l, U'_{l(l+1)}\right),
%\]
%by Lemma~\ref{isoaii_lem}-(iv),  the desired result follows by (\ref{iso_lambda_llp}) \footnote{Notice that (\ref{iso2_eqnn}) holds when $m \geq 1$. In the case $m=0$, the diagram (\ref{big_diagramn}) is reduces to 
%\[
%\xymatrix{F(A_0^j) \ar[rd] \ar[r] \ar@{=}[d] & F(A_{01}^j)  & F(A_1^j) \ar[l]  \ar[ld] \ar@{=}[d] \\
%F(A_0^{j+1}) \ar[r] & F(A_{01}^{j+1}) &  F(A_1^{j+1}), \ar[l]   }
%\]
%and the proof of the lemma immediately comes from (\ref{iso1_eqnn}), and (\ref{iso3_eqnn})}.
\end{proof}

\subsection{Characterization of very good homogeneous functors} \label{equiv_cat_thm_proof_subsection}

The aim here is to prove Theorem~\ref{equiv_cat_thm} announced earlier at the end of Subsection~\ref{definition_subsection}.

We will need the results obtained  in the previous subsections, and  three more lemmas. For the first one,  we need the following definition. A category $\mathcal{I}$ is said to be \textit{connected} if for any  objects $a, b \in \mcali$ there exists a zigzag 
$
\xymatrix{b=b_0 & b_1 \ar[l] \ar[r] & \cdots & b_m \ar[l]  \ar[r] & a=b_{m+1}}
$
of morphisms of $\mathcal{I}$. %which is denoted $(b, i_0, b_1, \cdots, b_m, i_m, a)$. 
 If $F \colon \mathcal{I} \lra \mcalc$ is a contravariant functor that sends every morphism to an isomorphism, then to any such a zigzag  one can associate an isomorphism $\text{Iso}(b, b_1, \cdots, b_m, a) \colon F(a) \lra F(b)$ defined in the same way as (\ref{fao_iso}). 
%\begin{eqnarray} \label{iso_ba}
%\text{Iso}(b, i_0, b_1, \cdots, b_m, i_m, a) = (F(i_0))^{-1} F(i_1)  \cdots  (F(i_{m-1}))^{-1}F(i_m).
%\end{eqnarray}
%If the category $\mathcal{I}$ is such that for any objects $c, d \in \mcali$, $\text{Hom}_{\mcali}(c, d)$ is either the emptyset or a single morphism, then $\text{Iso}(b, i_0, b_1, \cdots, b_m, i_m, a)$ is just denoted $\text{Iso}(b, b_1, \cdots, b_m, a)$. This is the case for the categories $\okmu$ and $\bkmu$. 
The following lemma is straightforward. 

\begin{lem} \label{lemma6_iso}
If for any $a, b \in \mathcal{I}$, $\text{Iso}(b, b_1, \cdots, b_m, a)$ does not depend on the choice of the zigzag between $a$ and $b$, then the limit of the  $F$ over $\mcali$ is isomorphic to $F(c)$ for any $c \in \mathcal{I}$. 
%Let $\mathcal{I}$ be a connected small category, and let $\mathcal{C}$ be a category that has all small limits. Consider a contravariant functor $F \colon \mathcal{I} \lra \mcalc$ that satisfies the following two properties: (a) $F$ sends any morphism to an isomorphism; (b) For every $a, b \in \mathcal{I}$, for every zigzag between $a$ and $b$, the associated isomorphism 
%\[\text{Iso}(b, i_0, b_1, \cdots, b_m, i_m, a) \colon F(a) \lra F(b)\]
%is unique.
%\footnote{The uniqueness means that given another zigzag $(b, j_1, c_1, \cdots, c_s, j_s, a)$, one has $\text{Iso}(b, i_1, b_1, \cdots, b_r, i_r, a) = \text{Iso}(b, j_1, c_1, \cdots, c_s, j_s, a)$}  
%Then for any $a \in \mcali$, the limit of the diagram $F$ over $\mcali$ is isomorphic to $F(a)$. That is,
%$
%\underset{b \in \mcali}{\text{lim}}\; F(b) \cong F(a). 
%$
\end{lem}

%\begin{proof}
%Let $a$ be an object of $\mcali$, and let $b \in \mcali$ be another object. Then, since $\mcali$ is connected by assumption, there exists a zigzag like (\ref{zigzag_ini}). Let $\lambda_b$ denote the $\text{Iso}(b, i_1, b_1, \cdots, b_m, i_m, a)$ from (\ref{iso_ba}). By the property (ii) $\lambda_b \colon F(a) \lra F(b)$ is the unique isomorphism from $F(a)$ to $F(b)$. If $f \colon b' \lra b$ is a morphism of $\mcali$, and if $\lambda_{b'} \colon F(a) \lra F(b')$ denotes the unique isomorphism from $F(a)$ to $F(b')$, then the obvious triangle  
%\[
%\xymatrix{                  &    & F(b) \ar[dd]^-{F(f)} \\
 %         F(a) \ar[rru]^-{\lambda_b} \ar[rrd]_-{\lambda_{b'}}  &   &     \\
	%				                 &      &     F(b')}
%\]
%commutes by (ii). So the collection $\left\{\lambda_b \colon F(a) \lra F(b)\right\}_{b \in \mcali}$ is a cone in $\mcalc$. Certainly, by (ii), this cone satisfies the universal property for limits. Thus $\underset{b \in \mcali}{\text{lim}} \;F(b) \cong F(a)$. 
%\end{proof}

\begin{lem} \label{lemma7_iso}
Let $F \colon \bkmu \lra \mcalc$ be a very good  functor, and let $W \in \okmu$. Consider the full subcategory $\btku(W) \subseteq \bku(W)$ whose objects $U$ have the property that the canonical inclusion $U \hookrightarrow W$ is an isotopy equivalence. Then the limit of the restriction $F| \btku(W)$ is isomorphic to $F(\widetilde{U})$ for any $\widetilde{U} \in \btku(W)$. 
%That is, 
%\[
%\forall \widetilde{U} \in \btku(W), \ \  \underset{U \in \btku(W)}{\text{lim}} F(U) \cong F(\widetilde{U}). 
%\]
Of course the same result holds when the domain of $F$ is replaced by $\okmu$.  
\end{lem}

\begin{proof}
Thanks to Lemma~\ref{abcd_lem} and Proposition~\ref{isotopy_invariance_prop} one can see that the restriction $F| \btku(W)$ satisfies the hypothesis of Lemma~\ref{lemma6_iso}, which completes the proof. 
%Manifestly, the category $\btku(W)$ is connected since each component of $W$ is path-connected. So the property (i) from Lemma~\ref{lemma6_iso} is satisfied. Now we want to check (ii) from the same lemma. Let $U, V \in \btku(W)$, and let $\{U_i, U_{j(j+1), \ 0 \leq i \leq m+1,\  1 \leq j \leq m}\}$ be a family of objects of $\btku(W)$
%such that 
%\[
%U_0 = U, \quad U_{m+1} = V, \quad \text{and} \quad U_{j(j+1)} \subseteq U_j \cap U_{j+1}.
%\]
% Of course this latter inclusion is required to be an isotopy equivalence. Such a family is precisely a typical zigzag 
%\begin{eqnarray} \label{zigzag_bbot}
%\xymatrix{U & U_{01} \ar[l] \ar[r] & U_1 &\cdots \ar[l] \ar[r] & U_m & U_{m(m+1)} \ar[l] \ar[r] & V}
%\end{eqnarray}
%between $U$ and $V$ in the category $\btku(W)$. One can then consider the isomorphism 
%\[
%\iso(U_0, U_{01}, \cdots, U_{m(m+1)}, U_{m+1}) \colon F(U) \lla F(V), 
%\]
%as define in (\ref{fao_iso}). Thanks to Proposition~\ref{isotopy_invariance_prop}, this isomorphism does not depend on the choice of the zigzag from $U$ to $V$, which completes the proof.  
\end{proof}

\begin{lem} \label{very_good_lem}
Let $F \colon \okm \lra \mcalc$ be a very good functor. Then  the functor $\frk \colon \om \lra \mcalc$ defined as 
$
\frk (U) = \underset{V \in \okm}{\lim} F(V) 
$
is also very good \footnote{The functor $\frk$ is nothing but the right Kan extension of $F$ along the inclusion $\okm \hra \om$}.
\end{lem}

\begin{proof}
Let $f \colon U \hra U'$ be  an isotopy equivalence of $\om$. Then there exists an isotopy $L \colon U \times [0, 1] \lra U'$ such that $L_0 = f$ and $L_1 \colon U \lra U'$ is a diffeomorphism. Our goal is to show that the canonical map
$
\psi \colon  \underset{V \in \ok(U')}{\lim} F(V) \lra  \underset{V \in \ok(U)}{\lim} F(V)
$ 
is an isomorphism. To do this, we will write $\psi$ as a composition $\psi = \lambda\phi$ of two isomorphisms:  
\[
\phi \colon  \underset{V \in \ok(U')}{\lim} F(V) \lra \underset{V \in \ok(U)}{\lim} F(L_1(V)) \quad \text{and} \quad \lambda \colon \underset{V \in \ok(U)}{\lim} F(L_1(V)) \lra \underset{V \in \ok(U)}{\lim} F(V). 
\]
We proceed in three steps. 

$\bullet$ Construction of $\phi$. Let $\iota \colon \ok(U) \lra \ok(U')$ be the functor defined as $\iota(V) = L_1(V)$. Clearly $\iota$ has an inverse since $L_1$ is a diffeomorphism.  %So $\iota$ is an isomorphism of categories. 
This implies that the  triangle
\[
\xymatrix{\ok(U) \ar[rr]^-{G} \ar[d]_{\iota}^{\cong}  &  &  \mcalc  \\
                \ok(U')  \ar[rru]_-{F},   &    &  }
\]
in which $G := F \iota$, induces an  isomorphism from the limit of $F$ to that of $G$. This isomorphism is nothing but  $\phi$. 

$\bullet$  The map $\lambda$ is induced by a natural isomorphism $\beta \colon G \lra F | \ok(U)$ defined in the following way. Let $V \in \ok(U)$, and let  $K \subseteq V$ be a compact subset such that $\pi_0(KV)$ is surjective.   Then, by Proposition~\ref{existence_prop}, there exists an admissible family $a = \{a_0, \cdots, a_{m+1}\}$ with respect to $\left\{K, L|V\times[0, 1]\right\}$.  By Definition~\ref{admissible_def}, such a family comes together with a collection $\mcalv_a = \{V_{01}, \cdots, V_{m(m+1)}\}$ that satisfies (\ref{admissible_cond}). Now  define $\beta[V] \colon F(L_1(V))  \lra F(V)$ as 
$
\beta[V] := \iso(\mcalv_a, a , K, L)
$
where $\iso(\mcalv_a, a , K, L)$ is the isomorphism introduced in Definition~\ref{iso_luvwk_defn}. 
%is independent of the choice of $\mcalv_a, a,$ and $K$ by Lemmas~\ref{uii_upii_lem}, \ref{isoab_isoab}, \ref{admissible_kkp_lem}. 
The naturality of $\beta$ is rather technical. Let $g \colon V \hra V'$ be a morphism of $\ok(U)$. The idea is to find an open cover, $\{(s-\epsilon, s+\epsilon_s)\}_{s \in I}$ of $I$, for which there is a commutative square 
\begin{eqnarray} \label{fvs_fvps0}
\xymatrix{F(V_{s-\epsilon_s})    &   &  & F(V_{s+\epsilon_s}) \ar[lll]_-{\cong}  \\
   F(V'_{s-\epsilon_s})   \ar[u]    & & &    F(V'_{s+\epsilon_s})  \ar[lll]^-{\cong} \ar[u] }
\end{eqnarray}
for each open interval $(s-\epsilon, s+\epsilon)$. Let $s \in [0, 1]$ and let $V_s:=L(V, s)$. Also let $0 \leq i \leq m+1$ such that $s \in [a_{i}, a_{i+1}]$, and let  
$
K'_s := \overline{V_{i(i+1)}}.
$
Certainly $K'_s$ is a compact subset of $V'_s$ by (\ref{admissible_cond2}) and the fact that $V \subseteq V'$. %Moreover $V_s \subseteq V'_s:= L(V', s)$ since $V \subseteq V'$. Therefore $K'_s$ is a compact subset of $V'_s$. 
Without loss of generality we can assume that $\pi_0(K'_sV'_s)$ is surjective (otherwise, if a component of $V'_s$ does not intersect $K'_s$, it suffices to add to $K'_s$ one point from that component.) Now let $b = \{b_s, \cdots, b_{r+1}\}$, with $b_s =s$ and $b_{r+1} =1$, be an admissible family with respect to $\left\{K'_s, L|V'_s \times [s, 1]\right\}$, and let   $\mathcal{V}' = \{V'_{s(s+1)}, \cdots, V'_{r(r+1)}\}$ be an associated collection satisfying (\ref{admissible_cond}). %which implies that $K'_s \subseteq V'_{s(s+1)}$. Since $K'_s$ contains $V_{i(i+1)}$, it follows that $V_{i(i+1)} \subseteq V'_{s(s+1)}$.  
Define
$
\epsilon_1 := \text{min}(b_{s+1}, a_{i+1}) - s.
$
By applying $F$ to the commutative diagram %\begin{eqnarray} \label{vsvps}
\[
\xymatrix{V_s \ar[d]  &  V_{i(i+1)} \ar[l] \ar[r] \ar[d]  & V_{s+\epsilon_1} \ar[d] \\
          V'_s      &  V'_{s(s+1)}  \ar[l]   \ar[r]  &   V'_{s+\epsilon_1},}
\]
and  by recalling the definition of $\iso(A, C, B)$ from (\ref{iso_acb}), we get the following commutative square
\begin{eqnarray} \label{fvs_fvps}
\xymatrix{F(V_s)    &   &  &  & F(V_{s+\epsilon_1}) \ar[llll]_-{\iso(V_s, V_{i(i+1)}, V_{s+\epsilon_1})}  \\
   F(V'_s)   \ar[u]    & & &  &  F(V'_{s+\epsilon_1}).  \ar[llll]^-{\iso(V'_s, V'_{s(s+1)}, V'_{s+\epsilon_1})} \ar[u] }
\end{eqnarray}
Similarly, there exist $\epsilon_2$ and a commutative square.
\begin{eqnarray} \label{fvs_fvps2}
\xymatrix{F(V_{s-\epsilon_2})    &   &  &  & F(V_{s}) \ar[llll]_-{\iso(V_{s-\epsilon_2}, V_{i(i+1)}, V_{s})}  \\
   F(V'_{s-\epsilon_2})   \ar[u]    & & &  &  F(V'_{s}).  \ar[llll]^-{\iso(V'_{s-\epsilon_2}, V'_{(s-1)s}, V'_{s})} \ar[u] }
\end{eqnarray}
Taking $\epsilon_s = \text{min}(\epsilon_1, \epsilon_2)$, and merging (\ref{fvs_fvps}) and (\ref{fvs_fvps2}), we get (\ref{fvs_fvps0}). Now, by using the compactness of $I$, we have a finite subcover  of $I$ and this produces a finite sequence of squares. Merging these squares, we get the obvious commutative square involving $F(V), F(V'), F(L_1(V)),$ and $F(L_1(V'))$, which proves the naturality of $\beta$.

$\bullet$ By construction, it is straightforward to check that $\psi = \lambda \phi$, which completes the proof. 
%For all  $V \in \ok(U)$,   one has $V_i \subseteq U'$ and $V_{i(i+1)} \subseteq U'$ for any $i$. This implies that the following triangle commutes. 
%\[
%\xymatrix{\frk(V)   &     &   \frk(L_1(V))  \ar[ll]_-{\beta[V]}   \\
  %                    &    \frk (U') \ar[lu]  \ar[ru] &  }
%\]
%Furthermore,  $\frk(V) \cong F(V)$ since $V$ is a terminal object of $\ok(V)$. For the same reason, we have $\frk(L_1(V)) \cong F(L_1(V))$.   Passing to the limit when $V$ runs over $\ok(U)$, we then deduce that $\psi = \lambda \phi$, which completes the proof. 
\end{proof}

\begin{rmk}
In \cite{wei99}, Lemma 3.8 asserts  the same thing as our Lemma~\ref{very_good_lem}  but for good functors $\om \lra \text{Top}$ into spaces instead. One might then ask the question to know why we provided another proof here, or why we did not adapt the proof of Weiss to our case.  The main reason is the fact that Weiss' proof uses geometric realizations of categories, which  lie naturally in spaces (and not in $\mcalc$!). 
\end{rmk}

We can now prove the main result of the section.

\begin{proof}[Proof of Theorem~\ref{equiv_cat_thm}]
%Recall the categories $\mcalf(\bkmu; \mcalc)$ and $\mcalf_k(\om, \mcalc)$ from Definition~\ref{cat_vghf_defn}. 
We want to prove that the categories $\mcalf(\bkmu; \mcalc)$ and $\mcalf_k(\om, \mcalc)$ are equivalent. Our strategy consists of doing that through two new categories. The first one, denoted   $\mcalf(\okmu; \mcalc)$, is the category of very good functors from $\okmu$ to $\mcalc$. And the second, denoted $\mcalf_k(\okm; \mcalc)$, is the category of very good functors  $F \colon \okm \lra \mcalc$ such that  $F| \mcalo_{k-1}(M) =0,$
where $0$ denotes the zero object of $\mcalc$. These categories fit into the diagram
\[
\xymatrix{\mcalf(\bkmu; \mcalc) \ar@<1ex>[r]^-{\psi_1} &  \mcalf(\okmu; \mcalc) \ar@<1ex>[l]^-{\phi_1} \ar@<1ex>[r]^-{\psi_2} &  \mcalf_k(\okm; \mcalc) \ar@<1ex>[l]^-{\phi_2} \ar@<1ex>[r]^-{\psi_3} &  \mcalf_k(\om, \mcalc) \ar@<1ex>[l]^-{\phi_3} }
\]
in which
 $\phi_1, \phi_2,$ and $\phi_3$ are the restriction functors,
 $\psi_1$ and $\psi_3$ are defined as 
$\psi_1(F)(U) = \underset{B \in \btku(U)}{\lim} F(B)$ and  $\psi_3(F)  = \frk,$
where $\frk$ is the functor defined in the statement of Lemma~\ref{very_good_lem},    and 
$\psi_2$ is defined as 
\[
\psi_2(F)(U)  =  \left\{ \begin{array}{cc}
                          F(U) & \text{ if }  U \in \okmu \\
													0   & \text{otherwise}. 
                         \end{array} \right.
\] 
Here  $\btku(U)$ is the category introduced in the statement of Lemma~\ref{lemma7_iso}.  From now on, our goal is to prove the following three claims. For $1 \leq i \leq 3$, the $i$th claim says that the functor $\psi_i$ is an equivalence of categories with inverse $\phi_i$. 

For the first claim, let $F \colon \bkmu \lra \mcalc$ be an object of $\mcalf(\bkmu; \mcalc)$. We first need  to show that $\psi_1(F)$ is very good. This easily follows from two applications of Lemma~\ref{lemma7_iso}. Certainly one has $\phi_1\psi_1 \cong id$ and $\psi_1\phi_1 \cong id$. 
The second claim follows immediately from the definitions. For the the third one,  
%Given $F \in \mcalf(\okmu; \mcalc)$, the functor $\psi_2(F)$ is very good because of the fact that if $i \colon U \hookrightarrow V$ is a morphism of $\okm$ then $i$ is an isotopy equivalence if and only if $\pi_0(i)$ is an isomorphism by Proposition~\ref{isotopy_equiv_prop}-(ii). The rest of the proof is obvious since by definition one easily has $\phi_2 \psi_2 =id$ and $\psi_2 \phi_2 =id$.
let $F \in  \mcalf_k(\okm; \mcalc)$. By Lemma~\ref{very_good_lem},  the functor $\psi_3(F) = \frk$ is very good. Moreover $\psi_3(F)$ is polynomial of degree $\leq k$ by Definition~\ref{polynomial_defn}-(ii). Furthermore, recalling the functor $T_k$ from Definition~\ref{polynomial_defn}-(iii), one has  
\[
\begin{array}{ccc}
T_{k-1} \frk (U)  = \underset{V\in \mcalo_{k-1}(U)}{\lim} \frk (V)  &  =  & \underset{V\in \mcalo_{k-1}(U)}{\lim} \left( \underset{W \in \ok(V)}{\lim} F(W)\right) \\
                                           &   \cong  & \underset{V\in \mcalo_{k-1}(U)}{\lim} F(V) \   \  \text{ since $V$ is a terminal object of $\ok(V)$} \\
                                           & \cong  & 0 \  \  \text{ since $F| \mcalo_{k-1}(M)= 0$.}
\end{array}
\]
Hence $\psi_3(F)$ is a very good homogeneous functor of degree $k$. Certainly one has natural isomorphisms $\phi_3\psi_3 \cong id$ and $\phi_3\psi_3 \cong id$, which completes the proof of the theorem. 
\end{proof}

%\begin{rmk}
%Notice that Theorem~\ref{equiv_cat_thm} also holds when $\mcalc$ is replaced by any category, with a zero object, that has all small limits. 
%\end{rmk}

As a consequence of  Theorem~\ref{equiv_cat_thm}, we have the following in which $F_k(M)$  denotes the unordered configuration space of $k$ points in $M$ with the subspace topology. 

\begin{coro} \label{equiv_cat_coro}
The category of very good homogeneous functors of degree $k$ is equivalent to the category of very good linear functors $\mcalo(F_k(M)) \lra \mcalc$. That is, there is an equivalence between  $\fkov$ and $\mcalf_1(\mcalo(F_k(M)); \mcalc)$, provided that $\mcalc$ has a zero object and all small limits. 
\end{coro}

\begin{proof}
Recall   $\mcalb(M)$ from the paragraph just before Definition~\ref{bku_oku_defn}. Let $\mathcal{B}^{(1)}(F_k(M)) \subseteq \mcalo^{(1)}(F_k(M))$ be the full subcategory  whose objects are the product of exactly $k$ pairwise disjoint objects of $\mcalb(M)$. Clearly objects of $\mathcal{B}^{(1)}(F_k(M))$ form a basis for the topology of $F_k(M)$. It is also clear that the category $\bkmu$ is canonically isomorphic to $\mathcal{B}^{(1)}(F_k(M))$. %the isomorphism sending $\coprod_{i=1}^k B_i$ to $\prod_{i=1}^k B_i$. 
This implies that the categories $\mcalf\left(\bkmu; \mcalc\right)$ and $\mcalf\left(\mathcal{B}^{(1)}(F_k(M)); \mcalc\right)$ are equivalent. Moreover, the categories $\mcalf\left(\bkmu; \mcalc\right)$ and $\mcalf\left(\mathcal{B}^{(1)}(F_k(M)); \mcalc\right)$ are equivalent as well as $\mcalf\left(\mathcal{B}^{(1)}(F_k(M)); \mcalc\right)$ and $\mcalf_1\left(\mcalo(F_k(M)); \mcalc\right)$ 
%That is,
%\begin{eqnarray} \label{equiv1_cat}
%\mcalf\left(\bkmu; \mcalc\right) \simeq \mcalf\left(\mathcal{B}^{(1)}(F_k(M)); \mcalc\right).
%\end{eqnarray}
%On the other hand, one has  
%\begin{eqnarray} \label{equiv2_cat}
%\mcalf\left(\bkmu; \mcalc\right) \simeq \mcalf_k\left(\om; \mcalc\right) 
%\end{eqnarray}
%and
%\begin{eqnarray} \label{equiv3_cat}
%\mcalf\left(\mathcal{B}^{(1)}(F_k(M)); \mcalc\right) \simeq \mcalf_1\left(\mcalo(F_k(M)); \mcalc\right)
%\end{eqnarray}  
by Theorem~\ref{equiv_cat_thm}. This proves the corollary. 
%(\ref{equiv1_cat}),  (\ref{equiv2_cat}), and  (\ref{equiv3_cat}). 
\end{proof}

%In the follow up paper \cite{paul_don17-2}, the target category of our functors is a model category, and we  work with \textit{good functors} instead (namely, functors that send isotopy equivalences to weak equivalences). We  show \cite[Theorem 1.3]{paul_don17-2}, which can be viwed as the \lq\lq good version\rq\rq{} of Corollary~\ref{equiv_cat_coro}. 

\section{Very good  functors}  \label{vgf_section}

The goal of this section is to prove the main result of the paper: Theorem~\ref{main_thm_paper}. The proof, which will be done at the end of Subsection~\ref{proof_main_thm_subsection}, goes through two big steps. The first step (Theorem~\ref{equiv_cat_thm}) has been already accomplished in Section~\ref{vghf_section}. The second one is  Theorem~\ref{crucial_thm} below, which roughly says that the category of very good functors is equivalent to the category of functors from the fundamental groupoid of $M$ to $\mcalc$. Both steps are connected by Proposition~\ref{equiv_cat_vectn_prop}, which involves the concept of \textit{very good covers} that we now explain. 

 %We begin with the notion of \textit{very good cover}, which is will be needed in the proof of Theorem~\ref{main_thm_paper}. 

%Let $M$ be a smooth manifold endowed with its canonical topology. 

\subsection{Very good covers} \label{vgc_subsection}

In this subsection, we introduce the notion of \textit{very good cover} (see Definition~\ref{gc_defn} below), $\mcalu$, of $M$. We show that such a cover produces a natural basis, $\mcalb_{\mcalu}$, of open balls for the topology of $M$. As posets, $\mcalu$ is smaller than $\mcalb_{\mcalu}$, but the categories $\mcalf(\mcalu; \mcalc)$ and  $\mcalf(\mcalb_{\mcalu}; \mcalc)$ of very good  functors are equivalent by Proposition~\ref{equiv_cat_vectn_prop} whose proof is the main goal here. 
%At the end of the subsection, we state Theorem~\ref{crucial_thm}. 

\begin{defn} \label{gc_defn}
An open cover $\mcalu = \{U_{\sigma}\}_{\sigma}$ of $M$ is called \emph{very good} if it satisfies the following four conditions.
\begin{enumerate}
\item[(C0)] Each $U_{\sigma}$ is diffeomorphic to an open ball.
\item[(C1)] For every $\sigma, \lambda$, the intersection $U_{\sigma} \cap U_{\lambda}$ is either the emptyset or a finite union of elements of $\mcalu$.
\item[(C2)] For every $\lambda$, the set $\left\{U_{\sigma}| \ U_{\sigma} \subseteq U_{\lambda}\right\}$ is finite.
\item[(C3)] For every open subset $B$ diffeomorphic to an open ball such that $B$ is contained in some $U_{\lambda}$, there exists a smallest (with respect to the order $U_{\sigma} \leq U_{\lambda}$ if and only if $U_{\sigma} \subseteq U_{\lambda}$) $U_{\sigma_B} \in \mcalu$ such that $B \subseteq U_{\sigma_B}$. 
\end{enumerate}
\end{defn}

\begin{rmk} \label{vg_cover_rmk}
If one replaces (C1) by

$(\widetilde{C}1)$ for every $\sigma, \lambda$, the intersection $U_{\sigma} \cap U_{\lambda}$ is either the emptyset or  an element of $\mcalu$, 

then (C2) will imply (C3). Indeed, suppose we have $(\widetilde{C}1)$ and (C2), and let $B \subseteq U_{\lambda}$ for some $\lambda$. Then the set $A = \{\alpha | \ B \subseteq U_{\alpha} \subseteq U_{\lambda}\}$ is finite because of (C2). Take then $U_{\sigma_B} = \cap_{\alpha_i \in A} U_{\alpha_i}$. This latter intersection lies in $\mcalu$ because of $(\widetilde{C}1)$. We thus get another definition of a very good cover with only three axioms: (C0),  $(\widetilde{C}1)$ and (C2). 
\end{rmk}

\begin{expl} \label{very_good_cover_expl}
Take $M = S^1 = \{(x, y) \in \rbb^2| \ x^2 + y^2 =1\}$.
\begin{enumerate}
 \item[(i)] Consider the open subsets $U_{\sigma_1} = S^1 \backslash \{(0, 1)\}, U_{\sigma_2} = S^1 \backslash \{(0, -1)\}, U_{\sigma_3} = \{(x, y) \in S^1| \ -1 \leq x < 0\}$ and $U_{\sigma_4} = \{(x, y) \in S^1| \ 0 < x \leq 1\}$. Then it is straightforward to check that the family $\mcalu  = \{U_{\sigma_i}\}_i$ is a very good cover of $S^1$. Note that the axiom $(\widetilde{C}1)$ does not hold here since $U_{\sigma_1} \cap U_{\sigma_2} = U_{\sigma_3} \cup U_{\sigma_4}$. 
\item[(ii)] Another example of a very good cover of $S^1$ is a cover with six open arcs $\mcalu' = \{U_{\lambda_i}\}_{i=1}^3 \cup \{U_{\lambda_{ij}}\}_{i < j}$ such that $U_{\lambda_{ij}} = U_{\lambda_i} \cap U_{\lambda_j}$. Contrary to $\mcalu$, the cover $\mcalu'$ satisfies $(\widetilde{C}1)$. 
\end{enumerate}
\end{expl}

Notice that the cover $\mcalu'$ %from Example~\ref{very_good_cover_expl} 
comes from a triangulation of $S^1$ with three $0$-simplices, and three $1$-simplices, while the cover $\mcalu$ 
%from the same example 
is not determined by a triangulation. In general, given a \textit{triangulation} $\tm$ 
%\footnote{By \textit{triangulation} of $M$, we mean a simplicial complex homeomorphic to $M$.} 
of a smooth manifold $M$, %if $S(\tm)$ denotes the set of simplices of $\tm$, 
one can always define a cover 
$
\utm = \{U_{\sigma}\}_{\sigma \in \tm}
$ 
of $M$ such that each $U_{\sigma}$ is diffeomorphic to an open ball, and 
\begin{eqnarray} \label{axiom_triangulation_cover}
\sigma_1 \cap \sigma_2 = \sigma \quad \mbox{if and only if} \quad U_{\sigma_1} \cap U_{\sigma_2} = U_{\sigma}.
\end{eqnarray}
Such a cover can be obtained in the following way. First take two barycentric subdivisions of $\tm$, and then define $U_{\sigma}$ as the interior of the star of $\sigma$. This  is indeed homeomorphic to an open ball $B^n, n:= \mbox{dim}(M)$  \cite[Proposition 6.3]{gallier08}. To see that $U_{\sigma}$ is diffeomorphic to $B^n$,   we need to deal with two cases. If $n \neq 4$, then there exists a unique smooth structure on $\rbb^n$ (see \cite[Section 2.4]{buo03}), and  this implies that the smooth manifold $U_{\sigma}$ is diffeomorphic to $B^n$. If $n = 4$, one can see that the smooth manifolds  $U_{\sigma}$ and $B^n$ are \textit{combinatorially equivalent} \footnote{Two smooth manifolds are said to be \textit{combinatorially equivalent} if they possess isomorphic $C^2$ triangulations.}, which implies the desired result by  Corollary 6.6 from \cite{mun60}. 
%by suitably thickening each simplex by a fixed $0<\epsilon <<1$ small enough. By convention $U_{\emptyset} = \emptyset$. 
From now on, a triangulation $\tm$ is fixed once and for all.  

\begin{prop} \label{vg_cover_prop}
The cover $\utm$ is very good.
\end{prop}

\begin{proof}
The axioms (C0) and $(\widetilde{C}1)$ are satisfied by construction, while (C2) follows from (\ref{axiom_triangulation_cover}). 
\end{proof}

%For the next proposition, recall that the topology of $M$ is the canonical one. Namely, the topology generated by the domains of charts of the maximal atlas. 
%(which defines the manifold structure of $M$).  

\begin{prop} \label{basis_prop}
Let $\mcalu$ be any very good cover of $M$. Then the set
\[
\mcalb_{\mcalu} = \left\{B \mbox{ diffeomorphic to an open ball such that } B \subseteq U_{\sigma} \mbox{ for some } U_{\sigma} \in \mcalu \right\}
\]
forms a basis for the topology of $M$. 
\end{prop}

\begin{proof}
This follows immediately from the axiom (C1), and the fact that $\mcalu \subseteq \mcalb_{\mcalu}$. 
%By definition it is clear that $\mcalu \subseteq \mcalb_{\mcalu}$. So $\mcalb_{\mcalu}$ is a cover of $M$. Moreover, for $A, B  \in \mcalb_{\mcalu}$, there exist $U_{\sigma}, U_{\lambda} \in \mcalu$ such that $A \subseteq U_{\sigma}$ and $B \subseteq U_{\lambda}$. This implies that $A \cap B \subseteq U_{\sigma} \cap U_{\lambda}$. But by the axiom $(C1)$ there is $U_{\alpha_1}, \cdots, U_{\alpha_k} \in \mcalu$ such that $U_{\sigma} \cap U_{\lambda} = \cup_i U_{\alpha_i}$. So for any point $x \in A \cap B$, one can find an open ball $D$ containing $x$ and contained in some $U_{\alpha_i}$ (it suffices to take $D$  small enough). This completes the proof. 
\end{proof}

%If one thinks $\mcalu$ and  $\mcalb_{\mcalu}$ as categories (morphisms being inclusions),  then one can easily see that $\mcalu$ is a subcategory of $\mcalb_{\mcalu}$ by definition. Consider now the very good cover $\utm$ of $M$ as above, and let $\butm$ be its associated basis as in Proposition~\ref{basis_prop}.  By definition, any morphism of $\butm$ is an inclusion of one open subset diffeomorphic to an open ball inside another one. This implies every morphism of $\butm$ is an isotopy equivalence. 

\begin{rmk} \label{bofm_butm}
Recall the notation ``$\mcalb(M)$'' introduced right before Definition~\ref{bku_oku_defn}. If one takes $\mcalb(M) = \butm$, then  by Definition~\ref{bku_oku_defn}-(i) it is easy to see that the posets $\bou(M)$ and $\butm$ coincide. That is, 
$
\bou(M) = \butm. 
$
\end{rmk}

Recall the notation $\mcalf(\mcala; \mcalc)$ from Section~\ref{notation_section}. 

%Let $\mcalc$ be a category. If $\mcala$ is a subcategory of $\om$, we will write $\mcalf(\mcala; \mcalc)$ to denote the category of very good contravariant functors (see Definition~\ref{verygood_defn}) from $\mcala$ to $\mcalc$.  Of course, morphisms of  $\mcalf(\mcala; \mcalc)$ are natural transformations. 

\begin{prop} \label{equiv_cat_vectn_prop}
Let $\mcalc$ be any category. Then the categories $\mcalf(\butm; \mcalc)$ and $\mcalf(\utm; \mcalc)$ are equivalent. 
\end{prop}

\begin{proof}
Define $\phi \colon \mcalf(\butm; \mcalc) \lra \mcalf(\utm; \mcalc)$ as the restriction to $\utm$. That is,  $\phi(G):= G|\utm$.  
To define a functor $\psi$ in the other way, let $F \colon \utm \lra \mcalc$ be an object of $\mcalf(\utm; \mcalc)$. For  $B \in \butm$ define $\psi(F)(B):= F(U_{\sigma_B})$, where $U_{\sigma_B}$ is provided by the axiom (C3) from Definition~\ref{gc_defn}. Again from the same axiom, one can easily define $\psi(F)$ on morphisms. If $\eta \colon F \lra F'$ is a morphism of $\mcalf(\utm; \mcalc)$, we define $\psi(\eta)(B):= \eta[U_{\sigma_B}]$. It is straightforward to check that $\phi \psi = id$ and $\psi \phi \cong id$. 

\end{proof}

We close this subsection with the statement of Theorem~\ref{crucial_thm}. 
%On the one hand, consider the category $\mcalf(\utm; \mcalc)$  of very good contravariant functors from $\utm$ to $\mcalc$. On the other hand, we have the following category. 

\begin{defn} \label{fpivect_defn}
Let $\Pi(M)$ be the fundamental groupoid of $M$. Define $\fpic$ as the category whose objects are contravariant functors $F \colon \Pi(M) \lra \mcalc$, and whose morphisms are  natural transformations. 
\end{defn}

%The following result says that these two categories are equivalent.  

%Let $\mcalf(\butm; \mcalc)$ be the category of very good contravariant functors. On the other let $\kbb$ be the ground field, and let $GL(F(U_{\lv}))$ denote the usual group of $\kbb$-linear isomorphisms from $\kbbn$ to $\kbbn$, 
%\[
%GL(F(U_{\lv})) = \{\varphi \colon \kbbn \lra \kbbn| \ \mbox{$\varphi$ is a $\kbb$-linear isomorphism}\},
%\]
%and let $\homt(\pio(M), GL(F(U_{\lv})))$ be the set of homomorphisms of groups between $\pio(M)$ and $GL(F(U_{\lv}))$.  In that set, define the conjugacy relation $\sim$ by  $f \sim g$ if and only if there exists $\varphi \in GL(F(U_{\lv}))$ such that $f(x) = \varphi^{-1} g(x) \varphi$ for all $x \in \pi_1(M)$, and consider the quotient $\homt(\pi_1(M), GL(F(U_{\lv})))\slash \sim$, which is just a set.  

\begin{thm} \label{crucial_thm}
Let $\mcalc$ be any category.  Then the categories $\mcalf\left(\utm; \mcalc\right)$ and $\fpic$ are equivalent. 
%That is, 
%\[
%\mcalf\left(\utm; \mcalc\right) \simeq \fpic.
%\]
\end{thm}  

To prove this theorem, we will construct two functors: 
$
\Psi \colon \mcalf\left(\utm; \mcalc\right) \lra \fpic
$
in Subsection~\ref{construction_psi_subsection},  and 
$
\Phi \colon \fpic \lra \mcalf\left(\utm; \mcalc\right)
$
in Subsection~\ref{construction_phi_subsection}.  Next, in Subsection~\ref{proof_main_thm_subsection}, we will show that $\Phi \Psi \cong id$ and $\Psi \Phi \cong id$.

\subsection{Construction of the functor $\Psi \colon \mcalf\left(\utm; \mcalc\right) \lra \fpic$}  \label{construction_psi_subsection}

To construct $\Psi$, we first need to replace the fundamental groupoid $\Pi(M)$ by another category easier to work with.

%We continue to use the same notations as in Subsection~\ref{vgc_subsection}.  As said the title, the goal of this subsection is to construct a functor 
%\[
%\Psi \colon \mcalf\left(\utm; \mcalc\right) \lra \fpic
%\] 
%from the category of very good contravariant functors to the category $\fpic$ introduced in Definition~\ref{fpivect_defn}. To do this we first need the notion of \textit{edge-path}, which will allow us to get a simpler definition of the fundamental groupoid $\Pi(M)$. Recall the classical definition of $\Pi(M)$. It is a category whose objects are points in $M$, morphisms being homotopy classes of paths. 

\begin{defn} \label{edge_loop_defn}
(i) An \emph{edge-path} is a sequence $(v_0, \cdots, v_r), r \geq 0,$ of vertices of $\tm$ such that two adjacent vertices, $v_i$ and $v_{i+1}$, are faces of the same $1$-simplex  in $\tm$. 
%When $r =0$,  $(v_0)$ is called the \emph{constant edge-path} at $v_0$.  
(ii) An \emph{edge-loop} is an edge-path starting and ending at the same vertex. 
%\item[(iii)] An edge-path $(v_0, v_1, \cdots, v_r, v_{r+1})$ is called \emph{reduced} if it does not contain any sub-sequence like $(v_i, v_{i+1}, v_i)$. An edge-loop  is \emph{reduced} if the underlying edge-path $(v, v_1, \cdots, v_r)$ is reduced.  
\end{defn}

 Let $P\tm$ denote the set of edge-paths in which we define  the equivalence relation $\sim$ generated by $(v_i, v_j, v_i, v_l) \sim (v_i, v_l)$, and $(v_i, v_j, v_k) \sim (v_i, v_k)$ if and only if $v_i, v_j,$ and  $v_k$ are the vertices of a $2$-simplex of $\tm$. 
%\[
%\mbox{Type1:} \quad  (v_0, \cdots, v_h, v_i, v_j, v_k, v_l, \cdots, v_r) \simeq (v_0, \cdots, v_h, v_i, v_l, \cdots, v_r)
%\]
%if and only if $v_k = v_i$,
%and
%\[
%\mbox{Type2:} \quad (v_0, \cdots, v_i, v_j, v_k, \cdots, v_r) \simeq (v_0, \cdots, v_i, v_k, \cdots, v_r) 
%\]
%if and only if $v_i, v_j, v_k$ determine a $2$-simplex in $\tm$. 
Let $P\tm \slash \sim$ be the category  whose objects are vertices of $\tm$, and whose morphisms are equivalence classes $[f]$ of edge-paths. By abuse of notation we will write $f$ for $[f]$. Compositions are induced by the  concatenation operation.  The following result is well known, and it is readily obtained from \cite[Proposition 1.26]{hatcher02}. 

\begin{thm} \label{pim_thm}
The fundamental groupoid of $M$ is isomorphic to the category $P\tm \slash \sim$. 
%That is,
%\[
%\Pi(M) \cong P\tm \slash \simeq. 
%\]
\end{thm}

From now on, one should think $\Pi(M)$ as the category $P\tm \slash \sim$. Consider the poset $\utm$ from Subsection~\ref{vgc_subsection}, and let $f = (v_0, \cdots, v_r)$ be an edge-path. For $0 \leq i \leq r-1$, one has the inclusions $U_{\vi} \subseteq U_{\langle v_iv_{i+1}\rangle}$ and $U_{\vit} \subseteq  U_{\langle v_iv_{i+1}\rangle}$, which we denote $p_{i(i+1)} \colon U_{\vi} \hookrightarrow U_{\langle v_iv_{i+1}\rangle}$ and $q_{i(i+1)} \colon U_{\vit} \hookrightarrow U_{\langle v_iv_{i+1}\rangle}.$ 
%Now we need to introduce some notation before defining $\Psi$. First, recall  the subposet \[\utm = \{U_{\sigma}\}_{\sigma \in S(\tm)} \subseteq \om\] from (\ref{cover_utm}), where $S(\tm)$ is the set of simplices of $\tm$. An object $U_{\sigma}$ is a suitable thickening (diffeomorphic to an open ball) of the simplex $\sigma$ such that $U_{\sigma} \cap U_{\lambda}$ is either the emptyset or $U_{\sigma \cap \lambda}$. Given an edge-path $f = (v_0, \cdots, v_r)$, we let  $\langle v_iv_{i+1} \rangle$ denote the edge connecting $v_i$ to $v_{i+1}$. For every $0 \leq i \leq r-1$, one has two canonical inclusions 
%\begin{eqnarray} \label{piqi}
%p_{i(i+1)} \colon U_{\vi} \hookrightarrow U_{\langle v_iv_{i+1}\rangle} \quad \mbox{and} \quad q_{i(i+1)} \colon U_{\vit} \hookrightarrow U_{\langle v_iv_{i+1}\rangle} 
%\end{eqnarray} 

\begin{rmk} \label{orf_rmk} The edge-path $f = (v_0, \cdots, v_r)$ is canonically oriented from $v_0$ to $v_r$. (i) If $(v_i, v_j, v_i)$ is an edge-path, then the corresponding  edges $\vij$ and $\vji$ differ  by their orientations. But, by the definition of $U_{\sigma}$, one has $U_{\vij} = U_{\vji}$. (ii) Clearly, one has $p_{ij} = q_{ji}$ and $q_{ij} = p_{ji}$ since $U_{\vij} = U_{\vji}$. So $p_{ij} \neq p_{ji}$ and $q_{ij} \neq q_{ji}$ whenever $i \neq j$. 
\end{rmk}

The following  definition is that of $\Psi$ on objects. We will define it on morphisms at the end of this subsection. 

\begin{defn} \label{psi_objects_defn}
Define 
%$
%\Psi \colon \mcalf\left(\utm; \mcalc\right) \lra \fpic
%$ 
%as  
$
\Psi(F)(v) := F(U_{\lv}).
$
If $f = (v_0, \cdots, v_r)$ is an edge-path, define $\Psi(F)(f) \colon \Psi(F)(v_r) \lra \Psi(F)(v_0)$  as the composite
\begin{eqnarray} \label{psif_defn}
\Psi(F)(f) := F(p_{01})(F(q_{01}))^{-1} \cdots F(p_{i(i+1)})(F(q_{i(i+1)}))^{-1} \cdots F(p_{(r-1)r})(F(q_{(r-1)r})^{-1}.
\end{eqnarray}
%on objects as follows. Let $F \colon \utm \lra \mcalc$ be very good.
%\begin{enumerate}
%\item[$\bullet$] On objects, that is, on vertices of $\tm$,
%\begin{eqnarray} \label{psifv}
%\Psi(F)(v) := F(U_{\lv}).
%\end{eqnarray}
%\item[$\bullet$] On morphisms, if $f = (v_0, \cdots, v_r)$ is an edge-path, define $\Psi(F)(f) \colon \Psi(F)(v_0) \lla \Psi(F)(v_r)$ as 
%\begin{eqnarray} \label{psif_defn}
%\Psi(F)(f) := F(p_{01})(F(q_{01}))^{-1} \cdots F(p_{i(i+1)})(F(q_{i(i+1)}))^{-1} \cdots F(p_{(r-1)r})(F(q_{(r-1)r})^{-1}
%\end{eqnarray}
\end{defn}

For the sake of simplicity, we will write $\Psi_F$ for $\Psi(F)$. To check that $\Psi_F(f)$ is well defined, we  need the following lemma.

\begin{lem} \label{circlular_lem}
Let $v_0, v_1$, and $v_2$ be the vertices of a $2$-simplex $\langle v_0v_1v_2\rangle$, and let $(v_0, v_1, v_2, v_0)$ be an edge-loop. Then 
\[
F(p_{01})(F(q_{01}))^{-1} F(p_{12})(F(q_{12}))^{-1} F(p_{20})(F(q_{20}))^{-1} = id
\]
\end{lem}

\begin{proof}
Consider the inclusions 
$d^0 \colon U_{\vot} \hookrightarrow U_{\vzot},$   $d^1 \colon U_{\vtz} \hookrightarrow U_{\vzot},$ and $d^2 \colon U_{\vzo} \hookrightarrow U_{\vzot}.$
Also consider the diagram 
\[
\xymatrix{   &      &   F(U_{\vt})    &    &      \\
                &  F(U_{\vtz}) \ar[ru]^-{F(p_{20})} \ar[ldd]_-{F(q_{20}}  &   &   F(U_{\vot})  \ar[lu]_-{F(q_{12})} \ar[rdd]^-{F(p_{12})}  &    \\
						&       &  F(U_{\vzot}) \ar[lu]_-{F(d^1)} \ar[ru]^-{F(d^0)}  \ar[d]_-{F(d^2)}  &   &       \\	
                F(U_{\vz})   &    &    F(U_{\vzo})  \ar[rr]_-{F(q_{01})} \ar[ll]^-{F(p_{01})} &      &   F(U_{\vo}) }
\]
The desired result follows from the fact that  each square of that diagram commutes. 
%the underlying diagram without $F$ (remember that $F$ is contravariant) is a diagram of inclusions, which is obviously commutative,  it follows that $(\ref{triangle_diagram})$ is commutes. Equivalently,  every square of that diagram commutes, and this gives the following  equalities. 

%\begin{align*}
%F(p_{01})(F(q_{01}))^{-1} F(p_{12})(F(q_{12}))^{-1} F(p_{20})(F(q_{20}))^{-1}  &  =   &  \\
% F(p_{01})F(d^2) (F(d^0))^{-1} F(d^0) (F(d^1))^{-1}  (F(q_{20}))^{-1} &  =  &  \\
%F(p_{01})F(d^2) (F(d^1))^{-1}  (F(q_{20}))^{-1}   &    =   &    \\
%F(p_{01})  (F(p_{01}))^{-1} F(q_{20}) (F(q_{20}))^{-1}   &  =  &   \\
%id. & &                                                                                                                                 
%\end{align*}
%The first equality comes from $(F(q_{i_1}))^{-1} F(p_{i_2}) = F(l_3)(F(l_1))^{-1}$ and $(F(q_{i_2}))^{-1} F(p_{i_3}) = F(l_1)(F(l_2))^{-1}$. And the third equality follows from $F(l_3)(F(l_2))^{-1} = (F(p_{i_1}))^{-1} F(q_{i_3})$. 
\end{proof}

\begin{prop} \label{psif_well_define_prop}
$\Psi_F(f) = \Psi_F(g)$ when $f \sim g$. 
\end{prop}

\begin{proof}
This follows immediately from the definition of $\Psi_F,$ Remark~\ref{orf_rmk}, and Lemma~\ref{circlular_lem}. 
\end{proof}

Now we define $\Psi$ on morphisms.
\begin{defn}  
If $\eta \colon F \lra F'$ is a morphism of $\mcalf(\utm; \mcalc)$, define $\Psi_{\eta} \colon \Psi_F \lra \Psi_{F'}$ as 
$
\Psi_{\eta}[v]  := \eta[U_{\lv}]. 
$
\end{defn}
%Now let $\eta \colon F \lra F'$ be a natural transformation between two very good functors $F, F' \colon \utm \lra \mcalc$. For any object $v \in \Pi(M)$, define $\Psi(\eta)[v] := \eta[U_{\lv}]$. 
The map $\Psi_{\eta}$ is  a member of $\underset{\fpic}{\text{hom}}(\Psi_F, \Psi_{F'})$ because of the following. For  
%\begin{prop} \label{psi_morphisms_prop}
%The collection $\{\eta[U_{\lv}]\}_{v \in \Pi(M)}$ is a natural transformation from $\Psi_F$ to $\Psi_{F'}$. 
%\end{prop}
an edge-path $f = (v_0, v_1, \cdots, v_r)$, one can consider the following diagram
\[
\xymatrix{F(U_{\langle v_0\rangle})  \ar[d]_-{\eta[U_{\lv}]}  &  F(U_{\langle v_0v_1\rangle}) \ar[l]_-{F(p_{01})} \ar[r] \ar[d]^-{\eta[U_{\langle v_0v_1 \rangle}]} &    \cdots &  F(U_{\langle v_{r-1}v_r\rangle}) \ar[rr]^-{F(q_{(r-1)r})} \ar[d]_-{\eta[U_{\langle v_{r-1}v_r\rangle}]} \ar[l]  & & F(U_{\langle v_r\rangle}) \ar[d]^-{\eta[U_{\langle v_r\rangle}]}  \\
F'(U_{\langle v_0\rangle})  &  F'(U_{\langle v_0v_1\rangle}) \ar[l]^-{F'(p_{01})} \ar[r]  &  \cdots &  F'(U_{\langle v_{r-1}v_r\rangle}) \ar[rr]_-{F'(q_{(r-1)r})} \ar[l] &  & F'(U_{\langle v_r\rangle})   }
\]
in which the horizontal arrows are obtained by applying $F$ and $F'$ to $p_{i(i+1)}$ and 
$q_{i(i+1)}$. Since  each square of that diagram commutes by the naturality of $\eta$, it follows that $\Psi_{\eta}$ is a natural transformation. 

\subsection{Construction of the functor $\Phi \colon \fpic \ \lra \mcalf\left(\utm; \mcalc\right)$}  \label{construction_phi_subsection}

%In Subsection~\ref{construction_psi_subsection} we defined a functor 
%$
%\Psi \colon \mcalf\left(\utm; \mcalc\right) \lra \fpic.
%$ 
%Here the aim is to construct a functor in the other way, namely
%$
%\Phi \colon \fpic \ \lra \mcalf\left(\utm; \mcalc\right). 
%$
%In Subsection~\ref{proof_main_thm_subsection} we will show that $\Phi$ is the \lq\lq inverse\rq\rq{} for $\Psi$. 

We begin with some notation. 
%At the beginning of the section, before Proposition~\ref{vg_cover_prop}, we have fixed a triangulation $\tm$ of $M$. In this subsection we continue to use the same triangulation.  We let $\tkm$ denote the $k$-skeleton of $\tm$. 
For a subcomplex $\mathcal{K} \subseteq \tm$, we will write $\mcalu(\mcalk) \subseteq \utm$ for the full subcategory  whose objects are $U_{\sigma}$ with $\sigma$ running over the set of simplices of $\mcalk$. For a category $\mcala$ we will write $ob(\mcala)$ for the collection of objects, and $mor(\mcala)$ for the collection of morphisms of $\mcala$. Recall the notation $\mcalt^M_p$ from Section~\ref{notation_section}.

Let $G \colon \Pi(M) \lra \mcalc$ be a contravariant functor, that is, $G \in \fpic$. We wish to define a very good functor 
$\Phi_G \colon \utm \lra \mcalc,$
where $\phig$ stands for $\Phi(G)$. Our strategy is to first define $\phig$ on $\mcalu(\too)$, then on $\mcalu(\tom)$, and $\mcalu(\ttm)$, and lastly, by  using Proposition~\ref{extension_prop} below, we will extend the definition on the rest of the triangulation.

On the $0$-skeleton $\too$, define $\phig(U_{\lv}):=G(v)$.  On the $1$-skeleton $\tom$, let $\langle v_0v_1\rangle$ be an edge, and let 
$d^0 :=q_{01} \colon U_{\langle v_1\rangle} \hra U_{\langle v_0v_1\rangle}$ and  $d^1 := p_{01} \colon U_{\langle v_0\rangle} \hra U_{\langle v_0v_1\rangle}.$
%Then the sequence $(v_0, v_1)$ represents the canonical edge-path from $v_0$ to $v_1$.  
Define 
\[
\phig(U_{\langle v_0v_1\rangle}) := G(v_1), \quad \text{and} \quad
\phig(d^i) := \left\{ \begin{array}{ccc}
                      id & \text{if}  & i = 0 \\
											G((v_0, v_1)) & \text{if} & i =1.
                     \end{array} \right.
\]
Of course $G((v_0, v_1))$ is an isomorphism  since every morphism of $\Pi(M)$ is invertible. So $\phig$ thus defined is a very good  functor. By definition it satisfies the following two conditions: (a) $\phig(U_{\lv}) = G(v)$ for all $v \in \too$; (b) $\phig(d^1)(\phig(d^0))^{-1}  = G((v_0, v_1))$ for any edge $\langle v_0v_1\rangle$. 

\begin{prop}
Given any other very good  functor $F \colon \mcalu(\tom) \lra \mcalc$ satisfying (a) and (b), 
there exists a natural isomorphism $\beta \colon \phig \lra F$.
\end{prop}  

\begin{proof}
Define $\beta$ as 
\[
\beta[U_{\sigma}] = \left\{ \begin{array}{ccc}
                                           id &  \text{if}  & \sigma \in \too \\
																					(F(d^0))^{-1}  & \text{if} & \sigma = \langle v_0v_1\rangle.
                                           \end{array}  \right.
\]
It is straightforward to check the naturality of $\beta$. 
%One can easily check the naturality of $\beta$ as shown the following diagram.
%\[
%\xymatrix{G(v_0) \ar[d]_{id}  & &  G(v_1)  \ar[ll]_-{G((v_0, v_1))} \ar[rr]^-{id}  \ar[d]_-{(F(d^0))^{-1}}  & &  G(v_1) \ar[d]^-{id}  \\
%  G(v_0)     &  &   F(U_{\langle v_0v_1 \rangle}) \ar[ll]^-{F(d^1)}  \ar[rr]_-{F(d^0)} & &   G(v_1).}
%\] 
%The lefthand square commutes by the property (b), and the righthand square is obviously commutative. 
\end{proof}

Now we define $\phig$ on the $2$-skeleton $\ttm$. Let $\tau = \langle v_0v_1v_2\rangle$ be a $2$-simplex, and let $\partial \mcalu(\tau) \subseteq \mcalu(\tau)$ denote the full subposet whose $ob(\partial \mcalu(\tau)) = ob(\mcalu(\tau)) \backslash \{U_{\tau}\}$. Let $F \colon \partial \mcalu(\tau) \lra \mcalc$ be a very good functor, and let $\iota \colon \partial \mcalu(\tau) \hookrightarrow \mcalu(\tau)$ be the inclusion functor. A contravariant functor $\fb \colon \mcalu(\tau) \lra \mcalc$ is called an \textit{extension} of $F$ if (i) $\fb$ is very good, and (ii) $\fb \circ \iota = F$. 

\begin{prop} \label{extension2_prop}
Such an extension $\fb$ exists if and only if 
\begin{eqnarray} \label{extension_cond}
F(p_{01}) \left(F(q_{01})\right)^{-1}F(p_{12}) \left(F(q_{12})\right)^{-1} F(p_{20}) \left(F(q_{20})\right)^{-1} = id. 
\end{eqnarray}
Moreover $\fb$ is unique up to isomorphism. 
\end{prop} 

\begin{proof}
If $\fb$ exists, then the equation (\ref{extension_cond}) holds by Lemma~\ref{circlular_lem}. Now assume that we have (\ref{extension_cond}), and define  $\fb := F$ on the poset $\partial \mcalu(\tau)$. Also define  $\fb(U_{\tau}) := F(U_{\langle v_1v_2\rangle}).$
On morphisms $d^i \colon U_{\langle v_0 \cdots \widehat{v}_i \cdots v_2\rangle} \hra U_{\langle v_0 \cdots v_2\rangle}$, define 
\[
\fb(d^0) := id, \ \ \  \fb(d^1) := \left( F(p_{20})\right)^{-1}F(q_{12}), \ \ \mbox{ and } \ \ \fb(d^2) := \left( F(q_{01})\right)^{-1}F(p_{12}). 
\]
If $a_1$ and $a_2$ are two composable morphisms of $\mcalu(\tau)$, define $\fb(a_2a_1) := \fb(a_1)\fb(a_2)$. To see that $\fb$ is well defined on compositions, we need to show that equations
$
\fb(p_{12})\fb(d^0) = \fb(q_{01}) \fb(d^2),  
$
$
\fb(q_{12})\fb(d^0) = \fb(p_{20}) \fb(d^1),
$
and
$
\fb(p_{01})\fb(d^2) = \fb(q_{20}) \fb(d^1)
$
hold. The first two follow directly from the definition of $\fb(d^0), \fb(d^1)$ and $\fb(d^2)$, while the latter follows from (\ref{extension_cond}).
%is equivalent to $F(p_{01})\fb(d^2) = F(q_{20}) \fb(d^1)$ since $p_{01}$ and $q_{20}$ are morphisms of $\partial \mcalu(\tau)$. We check this latter equation. 
%\begin{align*}
%F(p_{01})\fb(d^2)  &  = & F(p_{01})\left(F(q_{01})\right)^{-1}F(p_{12}) \mbox{  by the definition of $\fb(d^2)$} \\
 %                      & = & F(q_{20}) \left(F(p_{20})\right)^{-1}F(q_{12}) \mbox{ by (\ref{extension_cond})} \\
	%										& = & F(q_{20})\fb(d^1) \mbox{ by the definition of $\fb(d^1)$}. 
%\end{align*}
Thus $\fb$ is a well defined  functor, which is clearly very good and satisfies $\fb \circ \iota = F$. 

To prove the uniqueness, let $F' \colon \mcalu(\tau) \lra \mcalc$ be another extension of $F$. Define $\beta \colon \fb \lra F'$ as  
\[
\beta[U_{\sigma}] = \left\{ \begin{array}{cc}
                             \left(F'(d^0)\right)^{-1} & \mbox{ if $\sigma = \tau$} \\
														 id  & \mbox{ otherwise}.
                             \end{array} \right.
\]
By the definitions, it is straightforward to check that $\beta$ is a natural isomorphism. 
%By definition, every $\beta[U_{\sigma}]$ is an isomorphism. For the naturality of $\beta$, since $F' \circ \iota = F$ and $\fb \circ \iota = F$, we only need to check it on $d^0, d^1$, and $d^2$. 
%\begin{enumerate}
%\item[$\bullet$] On $d^0 \colon U_{\vot} \hookrightarrow U_{\tau}$ the resulting square,
%\[
%\xymatrix{F(U_{\vot}) \ar[rrr]^-{\beta[U_{\vot}] = id} &  & &  F(U_{\vot}) \\
 %         \fb(U_{\tau}) \ar[u]^-{id} \ar[rrr]_-{\beta[U_{\tau}] = (F'(d^0))^{-1}} & &  & F'(U_{\tau}), \ar[u]_-{F'(d^0)}}
%\]
%is clearly commutative. 
%\item[$\bullet$] On $d^1 \colon U_{\vzt} \hookrightarrow U_{\tau}$ the resulting square,
%\[
%\xymatrix{F(U_{\vzt}) \ar[rr]^-{id} &   &  F(U_{\vzt}) \\
 %         \fb(U_{\tau}) \ar[u]^-{\fb(d^1)} \ar[rr]_-{(F'(d^0))^{-1}} &  & F'(U_{\tau}), \ar[u]_-{F'(d^1)}}
%\]
%is commutative.  Indeed, in the category $\mcalu(\tau)$, one has the equality $d^0q_{12} = d^1p_{20}$. Applying $F'$ to  that equality we get $F(q_{12})F'(d^0) = F(p_{20})F'(d^1)$, which implies that $$F'(d^1) \left(F'(d^0)\right)^{-1} = \left(F(p_{20})\right)^{-1}F(q_{12}) = \fb(d^1).$$ The latter equality comes from the above definition of $\fb(d^1)$. 
%\item[$\bullet$] On $d^2 \colon U_{\vzo} \hookrightarrow U_{\tau}$ we have a similar proof as in the previous case by using \lq\lq$d^2q_{01} = d^0p_{12}$\rq\rq{} instead of \lq\lq$d^0q_{12} = d^1p_{20}$\rq\rq{}. 
%\end{enumerate}
%So $\fb$ is the unique extension  of $F$ up to isomorphism. 
\end{proof}

The functor $\phig \colon \mcalu(\tom) \lra \mcalc$ certainly satisfies (\ref{extension_cond}) because of the following reason. Since $\tau =\langle v_0v_1v_2 \rangle$ is a $2$-simplex, it follows that $(v_0, v_1, v_2, v_0) \sim (v_0)$. Therefore one has 
$
G((v_0, v_1)) G((v_1, v_2)) G((v_2, v_0)) = id.
$
Hence, thanks to Proposition~\ref{extension2_prop}, we can extend $\phig$ to $\mcalu(\mcalt^M_2)$ up to isomorphism, and the new functor is still denoted $\phig$.  

Now we define $\phig$ on the rest of the triangulation. Let $\delk, k \geq 0,$ denote the poset whose objects are nonempty subsets of $\{0, 1, \cdots, k\}$, and whose morphisms are inclusions. Consider the dual category $\delko$. Given an object $S$ of $\delk$ and $i \notin S$, the  inclusion $S \hookrightarrow S \cup \{i\}$ gives rise to a unique morphism $d_i \colon S \cup \{i\} \lra S$ in $\delko$. In fact $d_i$ consists of taking out $i$.   We will write $\{a_1, \cdots, \widehat{i}, \cdots, a_p\}$ for $\{a_1, \cdots, a_p\} \backslash \{i\}$.  The main observation here is the fact that every morphism of $\delko$ can be written as a composition of $d_i$'s. 
%given two objects $S$ and $S'$ of $\delko$, there is only one morphism between them, which can be written in different many ways.  For instance, if $S'$ is on the form $S' = S\backslash \{i_1, \cdots, i_p\}$ then the composition $d_{i_{\sigma(1)}} \circ \cdots \circ d_{i_{\sigma(p)}} \colon S \lra S'$ does not depend on the permutation $\sigma \in \Sigma_p$ we choose. This composition is the morphism from $S$ to $S'$. 
Let $\ddelko \subseteq \delko$ be the full subposet of all objects except $\{0, 1 \cdots, k\}$, and let $\iota \colon \ddelko \hookrightarrow \delko$ be the inclusion functor. The following result gives us a way to extend $\phig$ to $\mcalu(\tkm)$ when $k \geq 3$. 
%is a generalization of Proposition~\ref{extension2_prop} to high dimensional simplices. %We gives us a way to extend $\phig$ to $\mcalu(\tkm)$ when $k \geq 3$.  

\begin{prop} \label{extension_prop}
Assume $k \geq 3$, and let $\phi \colon \ddelko \lra \mcalc$ be a covariant functor that sends every  morphism to an isomorphism. Then there exists a unique functor (up to isomorphism) $\phib \colon \delko \lra \mcalc$ such that (i) the image of any morphism under $\phib$ is an isomorphism, and (ii) $\phib \circ \iota = \phi$. 
%the obvious triangle 
%\[
%\xymatrix{\ddelko \ar[rr]^-{\phi} \ar[d]_{\iota}  &  &  \mcalc \\
%          \delko  \ar[rru]_-{\phib} &    & }
%\]
%commutes.    
\end{prop}

\begin{proof}
We begin by showing that $\phib$ exists. On the poset $\ddelko$, define  $\phib := \phi$. Also define $\phib(\{0, \cdots, k\}) := \phi(\{1, \cdots, k\})$. On morphisms, we first define  $\phib(d_0) := id$, where $d_0 \colon \{0,  \cdots, k\} \lra \{1, \cdots, k\}$. Next define $\phib(d_i), 1 \leq i \leq k$, in such a way that the following square commutes. 
%On objects of $\delko$, we only need to define $\phib(\{0, \cdots, k\})$ since  by definition, the set of objects of $\ddelko$ is obtained from the set of objects of $\delko$ by removing only $\{0, \cdots, k\}$.  Set 
%\[
%\phib(\{0, \cdots, k\}) = \phi(\{1, \cdots, k\}).
%\]
%On morphisms,  define first $\phib(d_0)$ by  $\phib(d_0) = id$ where $d_0 \colon \{0, 1,  \cdots, k\} \lra \{1, \cdots, k\}$.
%Next define $\phib(d_i), 1 \leq i \leq k$, as follows. Consider the following commutative diagram in the category $\delko$. 
\[
\xymatrix{\phib(\{0, \cdots, k\}) \ar[rr]^-{\phib(d_i)} \ar[d]_-{\phib(d_0) = id}  &  &  \phib(\{0, \cdots, \widehat{i}, \cdots, k\}) \ar[d]^-{\phib(d_0) = \phi(d_0)} \\
           \phib(\{1, \cdots, k\}) \ar[rr]_-{\phib(d_i) = \phi(d_i)} &  &  \phib(\{1, \cdots, \widehat{i}, \cdots, k-1\}).}
\] 
Namely, define
 %Notice that the top  and the lefthand vertical morphisms belong to  $\delko$, while the bottom  and the righthand vertical  morphisms lie in $\ddelko$. Define 
$
\phib (d_i) := \left(\phi(d_0)\right)^{-1} \phi (d_i) \phib (d_0). 
$
On the compositions, we define $\phib$ in the obvious way. Since there could be many different ways to go from one object of $\delko$ to another one, one needs to check that $\phib$ is well defined on compositions. To do that, it is enough to show that the  equations  
$\phi(d_i)\phib(d_j)  = \phi(d_j)\phib(d_i)$, $0 \leq i < j \leq k,$ hold. The case $i=0$ follows immediately  from the definition of $\phib(d^i)$, while the other cases follow from the case $i=0$ and the equations 
$\phi(d_i)\phi(d_j) = \phi(d_j)\phi(d_i), i \neq j$.
%we have \[\phi(d_0)\phib(d_j) = \phi(d_0) \left(\phi(d_0)\right)^{-1} \phi(d_j)\phib(d_0) = \phi(d_j)\phib(d_0).\] In the first equality, we have replaced $\phib(d_j)$ by its definition (\ref{phib_di_fornula}). So (\ref{phiphib_equations}) holds when $i =0$. That is, 
%\begin{eqnarray} \label{stepo_eqn}
%\phi(d_0)\phib(d_j)  = \phi(d_j)\phib(d_0), 1 \leq j \leq k. 
%\end{eqnarray}
%\item[$\bullet$] Using the fact that $\phi$ is a functor by assumption, for any $d_i, d_j, i \neq j,$ in $\ddelko$,   we have 
%\begin{eqnarray} \label{phi_equations}
%\phi(d_i)\phi(d_j) = \phi(d_j)\phi(d_i).
%\end{eqnarray}
% For any $0 \leq i < j \leq k$,
%\begin{align*}
%\phi(d_0)\phi(d_j)\phib(d_i) & = & \phi(d_j)\phi(d_0)\phib(d_i) \ \ \mbox{ by (\ref{phi_equations})} \\
 %                            & = & \phi(d_j)\phi(d_i)\phib(d_0) \ \ \mbox{ by (\ref{stepo_eqn})} \\
	%													&  = & \phi(d_i)\phi(d_j) \phib(d_0)  \ \ \mbox{ by (\ref{phi_equations})}  \\
		%												& = & \phi(d_i) \phi(d_0)\phib(d_j) \ \ \mbox{ by (\ref{stepo_eqn})} \\
		%												& = & \phi(d_0) \phi(d_i) \phib(d_j)  \ \ \mbox{ by (\ref{phi_equations})}.
%\end{align*}
Thus $\phib$ is a functor which by definition satisfies conditions (i) and (ii) of the proposition. 

To prove the uniqueness part, let $F \colon \delko \lra \mcalc$ be another functor satisfying (i) and (ii).  By the definitions, it is straightforward to show that the map $\beta \colon \phib \lra F$ defined as  
\[
\beta[S] := \left\{ \begin{array}{ccc}
									 (F(d_0))^{-1} & \mbox{if} & S = \{0, \cdots, k\} \\
									 id & \mbox{if} & S \in \ddelko 
                  \end{array} \right.
\]
is a natural isomorphism. This ends the proof. 
\end{proof}

Now we define by induction $\Phi(\eta):= \Phi_{\eta}$ where  $\eta \colon G \lra G'$ is a morphism of $\fpic$. 
%To define $\Phi$ on morphisms, let $\eta \colon G \lra G'$ be a morphism of $\fpic$. We  define $\Phi(\eta):=\phie \colon \phig \lra \phigp$ by induction on the skeletons.  
On the $1$-skeleton, define 
\[
\phie[U_{\sigma}] := \left\{ \begin{array}{ccc}
                                       \eta[v]  & \text{if} & \sigma = v \in \too  \\
                                       \eta[v_1] & \text{if}  & \sigma = \vzo.                                     
                                        \end{array}  \right.
\]
 Assume that  we have defined $\phie$ on $\mcalu(\mcalt^M_{k-1}),$ for some $k \geq 2$. For a $k$-simplex $\vzk$, we define
$
\phie[U_{\vzk}] := \phie[U_{\vok}]. 
$
%\begin{eqnarray}  \label{phie_ob2}
%\phie[U_{\vzk}] := \phie[U_{\vok}]. 
%\end{eqnarray} 
By induction on the skeletons, one can easily verify that the collection 
 $
 \left\{ \phie[U_{\sigma}] \right\}_{U_{\sigma} \in \utm}
 $
 is a natural transformation from $\phig$ to $\phigp$.

\subsection{Proof of the main result of the paper}   \label{proof_main_thm_subsection}

The goal of this subsection is to prove the main result of the paper: Theorem~\ref{main_thm_paper}. Before we do this, we will first prove Theorem~\ref{crucial_thm} announced earlier at the end of Subsection~\ref{vgc_subsection}. We will need  two lemmas.

%\textbf{Summary of Subsections~\ref{vgc_subsection}, \ref{construction_psi_subsection}, \ref{construction_phi_subsection}.}

%\begin{lem} \label{psiphi_id_lem}
% $\Psi \Phi = id$. 
%\end{lem} 

%\begin{proof}
%This follows immediately from the definitions.  
%\end{proof}

%Now we show that the composition $\Phi\Psi$ is naturally isomorphic to the identity through the following two lemmas. 

\begin{lem} \label{phipsi_lem}
For any very good functor $F \colon \utm \lra \mcalc$, there is a natural isomorphism 
$
\beta[F] \colon \phisf \stackrel{\cong}{\lra} F.
$
\end{lem}

\begin{proof}
By induction. On the $1$-skeleton, define 
\[
\beta[F][U_{\sigma}] := \left\{ \begin{array}{ccc}
                             id  & \text{if}  & \sigma = v \in \too \\
														 (F(d^0))^{-1} & \text{if}  & \sigma = \vzo.
                               \end{array} \right.
\]
Assuming that $\beta[F]$ is defined on $\mcalu(\tm_{k-1})$ for some $k \geq 2$,  we define  
$
\beta[F][U_{\vzk}] := (F(d^0))^{-1} \beta[F][U_{\vok}], 
$
where $d^0 \colon U_{\vok} \hookrightarrow U_{\vzk}$. By the definitions, one can easily check that $\beta[F]$ is a natural isomorphism.            
\end{proof}
 
Again by induction, one can prove the following lemma just by using the definitions. 

\begin{lem} \label{phipsi_natiso_lem}
The collection $\beta = \left\{\beta[F]\right\}_{F \in \mcalf(\utm; \mcalc)}$ defines a natural isomorphism from $\Phi\Psi$ to $id$. 
\end{lem}

We can now prove Theorem~\ref{crucial_thm}, and  Theorem~\ref{main_thm_paper}. 

\begin{proof}[Proof of Theorem~\ref{crucial_thm}]
From  the definitions, one has $\Psi\Phi = id$, and from  Lemma~\ref{phipsi_natiso_lem}, we have $\Phi\Psi \cong id$.  
\end{proof}

\sloppy

\begin{proof}[Proof of Theorem~\ref{main_thm_paper}]
First of all, by Corollary~\ref{equiv_cat_coro} the category $\mcalf_k(\om; \mcalc)$ of very good homogeneous functors $\om \lra \mcalc$ of degree $k$ is equivalent to the category $\mcalf_1(\mcalo(F_k(M)); \mcalc)$ of very good linear functors $\mcalo(F_k(M)) \lra \mcalc$. %That is, $\mcalf_k\left(\om; \mcalc\right) \simeq \mcalf_1\left(\mcalo(F_k(M)); \mcalc\right).$
Let $\mcalt^{F_k(M)}$ denote a triangulation of $F_k(M)$. Consider the associated cover $\mcalu(\mcalt^{F_k(M)})$, which is very good by Proposition~\ref{vg_cover_prop}. Also consider the basis $\mcalb_{\mcalu(\mcalt^{F_k(M)})}$
%\[
%\mcalb_{\mcalu(\mcalt^{F_k(M)})} = \left\{B \subseteq F_k(M) \text{ diffeomorphic to an open ball| \ $B \subseteq U_{\sigma}$ for some $U_{\sigma} \in \mcalu(\mcalt^{F_k(M)})$}\right\},
%\]
for the topology of $F_k(M)$ (see Proposition~\ref{basis_prop}). By Theorem~\ref{equiv_cat_thm}, one has the equivalence  
$
\mcalf_1\left(\mcalo(F_k(M)); \mcalc\right) \simeq \mcalf\left(\bou(F_k(M)); \mcalc\right).
$
Since $\bou(F_k(M)) = \mcalb_{\mcalu(\mcalt^{F_k(M)})}$ by Remark~\ref{bofm_butm}, this latter equivalence becomes 
$
\mcalf_1\left(\mcalo(F_k(M)); \mcalc\right) \simeq \mcalf\left(\mcalb_{\mcalu(\mcalt^{F_k(M)})}; \mcalc\right).
$
Applying now Proposition~\ref{equiv_cat_vectn_prop}, we get 
$
\mcalf\left(\mcalb_{\mcalu(\mcalt^{F_k(M)})}; \mcalc\right) \simeq \mcalf\left(\mcalu(\mcalt^{F_k(M)}); \mcalc\right).
$
Lastly,  we have the equivalence 
%\begin{eqnarray} \label{main_thm_eqn5}
$
\mcalf\left(\mcalu(\mcalt^{F_k(M)}); \mcalc\right) \simeq \fpifc
$
by Theorem~\ref{crucial_thm}. Combining all these equivalences, we get the desired result. 
\end{proof}

\section{Connection to representation theory} \label{vghf_rep_section}

The goal of this short  section is to prove Corollary~\ref{main_coro_paper}, which establishes a connection between very good homogeneous functors and representation theory. Specifically, it says that the category of very good homogeneous functors of degree $k$ is equivalent to that of representations of $\pio(F_k(M))$ provided that $F_k(M)$ is connected.

To prove Corollary~\ref{main_coro_paper}, we will  need Proposition~\ref{fpim_rep_prop} below in which it is important to view $\Pi(M)$ as in Section~\ref{vgf_section} (see Theorem~\ref{pim_thm}). It is also important to view the fundamental group $\pio(M)$ as the set of equivalence classes of edge-loops starting and ending at the same vertex, namely $w$. For a group $G$, recall the category $\text{Rep}_{\mcalc} (G)$ from the introduction. 

%Before we state it, we need to recall some notation.  Let $m\tom$ denote the maximal tree of $\tm$. As we saw in Theorem~\ref{pim_thm}, the fundamental groupoid $\Pi(M)$ of $M$ can be viewed as the category whose objects are vertices of $\tm$, and whose morphisms are equivalence classes of edge-paths with respect to the relation \lq\lq$\simeq$\rq\rq{} introduced just after Definition~\ref{edge_loop_defn} (of \textit{edge-path} and of \textit{edge-loop}). Of course, the composition in $\Pi(M)$ is defined by the concatenation operation. In similar fashion, the fundamental group $\pio(M)$ can be viewed as the set of equivalence classes of edge-loops starting and ending at the same vertex, namely $w$. From now on, we fix $w$ once and for all. 

%We also have to recall some definitions. Let $\fpic$ be the category of contravariant functors from $\Pi(M)$ to $\mcalc$. The following is the definition of the category of representations. 

%\begin{defn} \label{repi_defn}
%For a group $G$, define $\repi$ as the category whose objects are  pairs $(A, \rho)$ where $A$ is an object of $\mcalc$, and $\rho \colon G \lra \aut(A)$ is a homomorphism of groups from $G$ to the group of automorphisms $\aut(A)$. A morphism  from $(A, \rho)$ to $(A', \rho')$ consists of a morphism $\varphi \colon A \lra A'$ in $\mcalc$ such that for all $x \in G$, the following square commutes. 
%\begin{eqnarray} \label{action_cond}
%\xymatrix{A   \ar[d]_-{\rho(x)} \ar[rr]^-{\varphi}  &  &  A' \ar[d]^-{\rho'(x)}  \\
 %        A \ar[rr]_-{\varphi}   &   &   A' }
%\end{eqnarray}
%\end{defn}
\begin{prop} \label{fpim_rep_prop}
Assume that $M$ is connected. Then the category $\fpic$ from Definition~\ref{fpivect_defn} is equivalent to the category $\repi$ of representations of $\pio(M)$ in $\mcalc$. That is,
$
\fpic \simeq \repi. 
$
\end{prop}

\begin{proof}
We  define two functors 
\[
\xymatrix{\fpic \ar@<1ex>[rr]^-{\Theta} & & \repi \ar@<1ex>[ll]^-{\Lambda}}
\]
%and then we will show that $\Theta \Lambda \cong id$ and $\Lambda \theta \cong id$. 
as $\Theta(F) := (F(w), \rhof)$ where $\rhof(f) := F(f)$ for any edge-loop $f$. If $\eta \colon F \lra F'$ is a natural transformation, we define 
$\Theta(\eta) \colon (F(w), \rhof) \lra (F'(w), \rhofp)$ as  $\Theta(\eta) := \eta[w]$. Now define $\Lambda(V, \rho)(v) := V$. 
 To define $\Lambda(V, \rho)$ on morphisms of $\Pi(M)$, we need to introduce the following definition. An edge-path $f= (v_0, \cdots, v_r)$ is called \textit{reduced} if $v_i \neq v_j$ whenever $i \neq j$. Given two different vertices $v$ and $v'$, there exists a unique reduced edge-path, denoted $g_{vv'} = (v, v_1, \cdots, v_{r-1}, v')$, in the maximal tree $m\tom$.  This is because $m\tom$ is contractible since $M$ is connected by assumption. If $f = (v_0, \cdots, v_r)$  is an edge-path we define $\Lambda(V, \rho)(f):= \rho(\ftild)$  where $\ftild := g_{wv_0} fg_{v_rw}$.  On morphisms of $\repi$, we define $\Lambda$ in the obvious way. That is, $\Lambda(\varphi)[v] := \varphi$. By the definitions, it is straightforward to verify that  $\Theta\Lambda = id$ and $\Lambda\Theta \cong \id$, which completes the proof.  
\end{proof}

%We are now ready to prove Corollary~\ref{main_coro_paper}. 

\begin{proof}[Proof of Corollary~\ref{main_coro_paper}]
This follows immediately from Theorem~\ref{main_thm_paper} and Proposition~\ref{fpim_rep_prop}. 
\end{proof}

\section{Very good vector bundles}  \label{vgvb_section}

%In Section~\ref{generator_homogeneous_functor_section} we studied homogeneous functors  of degree $k$ from $\om$ to the category $\ch$ of chain complexes  in an abelian category $\mathcal{A}$. One of the most important results we derived  is Theorem~\ref{generated_thm}, which says that the category of homogeneous  functors (which is actually a triangulated category by Theorem~\ref{thick_sub_thm}) is generated by a certain subcategory $\mathcal{I}_k$ of \lq\lq very good functors\rq\rq{}. In this section we prove that $\mathcal{I}_k$, with $\mathcal{A} = \mcalc$ (the category of finite dimensional vector spaces), is equivalent to a particular class of vector bundles that we call \textit{very good vector bundles} (see Definition~\ref{very_good_vb_defn} below). More precisely, we prove Theorem~\ref{vgm_equivalence_thm}. We also prove Theorem~\ref{vbm_abelian_thm}, which asserts that our category of very good vector bundles is abelian.  Notice that the classical category of vector bundles over a fixed base is not abelian. 

We prove Theorem~\ref{vgm_equivalence_thm} below which states that  the category of very good functors, studied in  the previous sections, is equivalent to a particular class of vector bundles (which we call \textit{very good vector bundles} (see Definition~\ref{vgvb_defn} below)). We also prove Theorem~\ref{vbm_abelian_thm}, which states that our category of very good vector bundles is abelian.  %Notice that the traditional category of vector bundles over a fixed base is not abelian. 
Throughout this section, we will write $\fvect$ for the category of finite dimensional vector spaces over a field $\kbb$. 

%show that the category of very good functors (that we studied in previous sections) is equivalent to a particular class of vector bundles that we call \textit{very good vector bundles} (see Definition~\ref{vgvb_defn} below). Specifically, we prove Theorem~\ref{vgm_equivalence_thm}. We also prove Theorem~\ref{vbm_abelian_thm}, which asserts that our category of very good vector bundles is abelian.  Notice that the traditional category of vector bundles over a fixed base is not abelian. 

%Here the target category for our functors  is no longer $\mcalc$, as in previous sections, but $\fvect$ instead. 

% We will write  $\bomu$ to denote the category whose  objects are the open subsets of $M$ diffeomorphic to an open ball. Morphisms of $\bomu$ are inclusions. The objects of $\bomu$ are required to form a basis for the topology of $M$. Fix $\bomu$ once and for all in this section. 

\subsection{Definition and examples}  \label{vgvb_examples_subsection}

%\begin{defn}
%A functor $F \colon \om \lra \mcalc$ from $\om$ to finite dimensional vector spaces is called \emph{very good} if it sends isotopy equivalences to isomorphisms. 
%\end{defn}

Roughly speaking, a \textit{very good vector bundle} is a vector bundle in the classical sense endowed with an extra structure (which is a very good covariant functor) that satisfies the axioms for a vector bundle, and an additional axiom (which is some kind of compatibility between local trivializations). To be more precise, we have the following definition. 
%Recall that a functor from a subcategory of $\om$ to vector spaces is called \textit{very good} if it sends isotopy equivalences to isomorphisms.

\begin{defn} \label{vgvb_defn}
Let $\bomu$ as in Definition~\ref{bku_oku_defn}. The category  $\vgvb$  of \emph{very good vector bundles} is defined as follows. An object of  $\vgvb$ is a triple $(E, \pi, F_{\pi})$ where $E$ is a topological space,  $\pi \colon E \lra  M$ is a continuous surjection, and $F_{\pi} \colon \bomu \lra \fvect$ is a very good covariant functor. Such a triple is endowed with a family of homeomorphisms $\left\{\varphi^{\pi}_U \colon U \times F_{\pi} (U) \lra \pi^{-1}(U)\right\}_{U \in \bomu}$ that satisfy the following four axioms:
    \begin{enumerate}
		     \item[(VB0)] for every $x \in M$, the preimage $\pi^{-1}(x)$ is a vector space;
		     \item[(VB1)] for all $U \in \bomu$, and for all $(x, v) \in U \times F_{\pi} (U)$, $(\pi \circ \vfi_U^{\pi})(x, v) = x$;
				 \item[(VB2)]  for all $U \in \bomu$, and for all $x \in U$, the map $\vfi^{\pi}_{xU} \colon F_{\pi}(U) \lra \pi^{-1}(x)$ defined by $\vfi^{\pi}_{xU} (v) = \vfi^{\pi}_U (x, v)$ is an isomorphism of vector spaces. 
				 \item[(VB3)] For any morphism $j \colon U \hookrightarrow V$ of $\bomu$, the following square commutes.
				\[
				  \xymatrix{U \times F_{\pi}(U) \ar[rr]^-{\vfi^{\pi}_U} \ar[d]_{j \times F_{\pi}(j)} & & \pi^{-1}(U) \ar@{^{(}->}[d] \\
					          V \times F_{\pi}(V) \ar[rr]^-{\vfi^{\pi}_V} &  &  \pi^{-1}(V). }
				\]
		\end{enumerate}
  A morphism	from $(X, p,  F_p)$ to $(Y, q, F_q)$ consists of a pair $(f, \eta)$ where $f \colon X \lra Y$ is a map making the obvious triangle commute, and $\eta \colon F_p \lra F_q$ is a natural transformation such that for every $U \in \bomu$, the following square commutes. 
					\[
				  \xymatrix{U \times F_p(U) \ar[rr]^-{id \times \eta[U]} \ar[d]_{\vfi^p_U} & & U \times F_q(U)  \ar[d]^-{\vfi^{q}_U} \\
					       p^{-1}(U)    \ar[rr]^-{f} &  &  q^{-1}(U). }
				  \]
				%Here the restriction of $f \colon X \lra Y$ to $p^{-1}(U)$ is still denoted $f$, and $\eta[U]$ is the component of $\eta$ at $U$.  
%An object of $\vgvb$ is  called a \emph{very good vector bundle}. So  $\vgvb$ is called the \emph{category of very good vector bundles with respect to $\bomu$} or simply the \emph{category of very good vector bundles}. 
\end{defn}

By definition, any very good vector bundle is a vector bundle in the classical sense. But the converse is not true as shown Example~\ref{vb_not_verygood} below. 

\begin{expl} \label{vb_verygood}
Consider the Mobius bundle $(E, \pi)$ where $E = [0, 1] \times \rbb \slash \sim$ and \lq\lq $\sim$\rq\rq{} is the equivalence relation $\sim$ is generated by $(0, t) \sim (1, -t)$ for all $t \in \rbb$. Of course $\pi \colon E \lra M=S^1 \cong [0, 1]\slash 0 \sim 1$ is the canonical projection.  Also consider the convariant functor $F_{\pi} \colon \bomu \lra \fvect$ defined as $F_{\pi}(U) = \rbb$, and 
\[
F_{\pi}(U \hookrightarrow V) = \left\{ \begin{array}{cc}
                                        -id_{\rbb} & \text{if $U \subseteq (0, 1)$ and $[0] = [1] \in V$} \\
																				id_{\rbb} &  \text{otherwise}. 
                                       \end{array} \right.                            
\]
One can easily verify that $(E, \pi, F_{\pi})$ is a very good vector bundle. 
\end{expl}

Many other examples of  classical vector bundles are very good including line bundle over the real projective space $\rbb \pbb^n$. %As an example of a vector bundle which is not very good, we have the following.  

\begin{expl} \label{vb_not_verygood}
The canonical complex line bundle over $\cp^1 \cong S^2 =M$ is not a very good vector bundle essentially because of the fact that the second component of the local trivializations   $\vfi_U^{\pi}(x, z)$ depends on both variables $x$ and $z$.   
\end{expl}

\begin{rmk}  \label{mobius_rmk}
Consider the Mobius bundle $(E', \pi')$ where $E' = [0, 1]\times \rbb \slash (0, t) \sim (1, -2t)$, and $\pi' \colon E' \lra S^1$ is defined by $\pi'([x, t]) = [x]$. Also consider the vector bundle $(E, \pi)$ from  Example~\ref{vb_verygood}. Certainly the bundle $(E', \pi')$  is very good as in Example~\ref{vb_verygood}. The point is that $(E', \pi', F_{\pi'})$ and $(E, \pi, F_{\pi})$ are isomorphic as vector bundles, but not as very good vector bundles. 
\end{rmk}

\subsection{Proving that the category $\vgvb$  is abelian}  \label{vgvb_abelian_subsection}

To  prove that  $\vgvb$  is abelian, we need the following lemma.

\begin{lem} \label{ker_coker_lem}
Let $(f, \eta) \colon (X, p, F_p) \lra (Y, q, F_q)$ be a morphism of $\vgvb$. Let $U, V \in \bomu$ such that $U \cap V \neq \emptyset$. Then for any $x \in U$, and $y \in V$ ($x$ and $y$ need not lie in $U \cap V$), there are  isomorphisms $\noy f_x \cong F(U) \cong F(V) \cong \noy f_y,$ and $\coker f_x \cong G(U) \cong G(V) \cong \coker f_y,$ of vector spaces. Here $f_x \colon p^{-1}(x) \lra q^{-1}(x)$ is the canonical linear map induced by $f$, $F(U)$ and $G(U)$ are the kernel and  cokernel of $\eta[U] \colon F_p(U) \lra F_q(U)$ respectively. 
\end{lem}

\begin{proof}
We will prove the first set of isomorphisms; the proof of the second set is similar. The first isomorphism, $\noy f_x \cong F(U)$,  is readily obtained from the following commutative diagram in which the isomorphisms $\vfi^p_{xU}$ and $\vfi^q_{xU}$ are provided by the axiom (VB2). 
\begin{eqnarray} \label{crucial_ab_diag}
\xymatrix{F(U) =\noy \eta[U] \ar@{^{(}->}[r] \ar@{.>}[d]_-{\lambda_{xU}}^-{\cong} & F_p(U) \ar[r]^-{\eta[U]} \ar[d]_-{\vfi^p_{xU}}^-{\cong} & F_q(U) \ar[d]_-{\vfi^q_{xU}}^-{\cong} \\ 
       \noy f_x \ar@{^{(}->}[r] & p^{-1}(x) \ar[r]^-{f_x} & q^{-1}(x).}
\end{eqnarray}
%In that diagram, the isomorphisms $\vfi^p_{xU}$ and $\vfi^q_{xU}$ are provided by the axiom (VB2), the  square on the right commutes because of the commutative property of (\ref{axiom2_morphism}), and  $\lambda_{xU}$ is  the obvious isomorphism, which is indeed the restriction of $\vfi^p_{xU}$ to $\noy \eta[U]$. 
The second isomorphism follows from the  facts (i) the intersection of $U$ and $V$ is not empty, (ii) objects of $\bomu$ form a basis for the topology of $M$, (iii) every morphism of $\bomu$ is an isotopy equivalence, and (iv) the functor $U \mapsto F(U)$ is very good. The last isomorphism is obtained in the same way as the first one. 
\end{proof}

\begin{thm} \label{vbm_abelian_thm}
The category $\vgvb$ of very good vector bundles over a manifold $M$ is an abelian category. 
\end{thm}

\begin{proof}
First recall that a category is \emph{abelian} if it satisfies the following four axioms: (Ab1) it has a zero object, (Ab2) it has all binary products and binary coproducts, (Ab3) it has all kernels and cokernels, and (Ab4)  every monomorphism (respectively epimorphism) is a kernel (respectively cokernel) of a map. 

$\bullet$ Axiom (Ab1). The zero object is $(X_0, p_0, F_{p_0})$ where $X_0 = M \times \{0\}$, for all $x \in M, p_0(x,0) = x$, and for all $U \in \bomu, F_{p_0}(U) = \{0\}$, the trivial vector space.   
%Given another object $(E, \pi, F_{\pi})$ of $\vgvb$, one can easily check the following two items. 
  %\begin{enumerate}
%	\item[$\bullet$] The unique morphism from $(X_0, p_0, F_{p_0})$ to $(E, \pi, F_{\pi})$ is the pair $(f_0, \alpha_0)$ defined by $f_0(x, 0) = 0 \in \pi^{-1}(x)$, and $\alpha_0(U) \colon \{0\} \lra F_{\pi}(U)$ is indeed the zero morphism.
%	\item[$\bullet$] The unique morphism from $(E, \pi, F_{\pi})$ to $(X_0, p_0, F_{p_0})$ is the pair $(g_0, \beta_0)$ defined by $g_0(x) = (\pi(x), 0)$, and $\beta_0(U) \colon F_{\pi}(U) \lra \{0\}$ is indeed  the zero morphism too. 
%	\end{enumerate}

$\bullet$ Axiom (Ab2). Let $(X, p, F_p)$ and $(Y, q, F_q)$ be two objects of $\vgvb$. Define a new object $(E, \pi, F_{\pi})$ as follows. The total space is  the pullback of $\xymatrix{X \ar[r]^-{p} & M & Y \ar[l]_-{q}}$. The projection is $\pi := pf_p = qf_q$, where $f_p \colon E \lra X$ and $f_q \colon E \lra Y$ are the projections on the first and second component respectively.  The covariant functor $F_{\pi}$ is defined as the product $F_{\pi} = F_p \times F_q$. For every $U$, define $\vfi^{\pi}_U (x, v_1, v_2) := (\vfi^p_U(x, v_1), \vfi^q_U(x, v_2))$.  It is clear that $\vfi^{\pi}_U$ is an homeomorphism because so are $\vfi^p_U$ and $\vfi^q_U$. It is also clear that $(E, \pi,  F_{\pi})$ satisfies the axioms of a very good vector bundle. One can easily verify that $(E, \pi,  F_{\pi})$ is the  product as well as the coproduct of $(X, p, F_p)$ and $(Y, q, F_q)$.  

$\bullet$ Axiom (Ab3). Let $(f, \eta) \colon (X, p, F_p) \lra (Y, q, F_q)$ be a morphism of $\vgvb$. We wish to construct a new object $(E, \pi, F_{\pi})$, which will represent the kernel of $(f, \eta)$. First of all, define $E$ as 
$E = \left\{(x, v)\ | \ x \in M, v \in \noy f_x \right\}.$
Next define $\pi \colon E \lra M$ by $\pi(x, v) = x$, and $F_{\pi} \colon \bomu \lra \fvect$ by $F_{\pi}(U) = \noy \eta[U]$.  Recalling the first isomorphism of Lemma~\ref{ker_coker_lem} or more precisely the isomorphism $\lambda_{xU}$ from (\ref{crucial_ab_diag}), we define the local trivialization $\vfi^{\pi}_{U} \colon U \times F_{\pi}(U) \lra \pi^{-1}(U)$  by the formula   $\vfi^{\pi}_{U}(x, v) = (x, \lambda_{xU}(v)).$
It is clear that the triple $(E, \pi, F_{\pi})$ thus defined is an object of $\vgvb$.  Moreover  $(E, \pi, F_{\pi})$ is the desired kernel.  Similarly, one has the cokernel of $(f, \eta) \colon (X, p, F_p) \lra (Y, q, F_q)$, which is obtained  by replacing in the preceding construction \lq\lq ker\rq\rq{} by \lq\lq coker\rq\rq{}. 

$\bullet$ Axiom (Ab4). Since we defined the kernel and the cokernel fibrewise, and since the category of vector spaces is abelian, it follows that every monomorphism in $\vgvb$ is the kernel of its cokernel, and every epimorphism is the cokernel of its kernel. 
\end{proof}

\begin{rmk}  \label{abelian_rmk}
The traditional category of vector bundles over a fixed base is not abelian as there is an issue with the kernel of a morphism, which is not a bundle in any natural way. For instance, consider the trivial bundle $(X, p)$ over $\rbb$ with $X = \rbb \times \rbb$, and $p(x, y) = x$. Also consider the map $f \colon X \lra X$ defined by $f(x, y) = (x, xy)$. Fibrewise, the kernel of $f$ is $\rbb$ over $0$, and is reduced to the trivial vector space otherwise. So the map $x \mapsto \mbox{dim}(\noy f_x)$ is not locally constant, and therefore the kernel of $f$ is not a vector bundle neither in the classical sense nor in  the sense of Definition~\ref{vgvb_defn}.  A similar issue happens to the cokernel. The crucial thing that turns our category of vector bundles  into an abelian category is Lemma~\ref{ker_coker_lem}. 
\end{rmk}

\subsection{Equivalence between $\vgvbf$ and $\vgfm$}   \label{vgvb_equivalence_subsection}

The goal of this subsection is to prove Theorem~\ref{vgm_equivalence_thm}, which says that the category of very good vector bundles is equivalent to the category of very good functors. Recall the notation $F_k(M)$ from Section~\ref{notation_section}.

%We use the same notations as those introduced in Section~\ref{vghf_section}. Thus we let $F_k(M)$ denote the unordered configuration space of $k$ points in $M$.  Also we let $\bou(F_k(M))$ denote the category whose objects are exactly the product of $k$ objects of $\bomu$, where $\bomu$ is as in Definition~\ref{vgvb_defn}.  Of course a morphism in  that category is an inclusion. Note that every object of $\bou(F_k(M))$ is diffeomorphic to an open ball in $F_k(M)$. Note also that objects of $\bou(F_k(M))$ form a basis for the canonical topology of $F_k(M)$.  Now we let $\vgfm$ denote the category of very good contravariant functors $F \colon \bou(F_k(M)) \lra \fvect$ from $\bou(F_k(M))$ to finite dimensional vector spaces, morphisms being indeed natural transformations. 

\begin{thm} \label{vgm_equivalence_thm}
The category $\vgvbf$ of very good vector bundles over $F_k(M)$ is equivalent to the category $\vgfm$ of very good contravariant functors from $\bou(F_k(M))$ to $\fvect$. 
%\begin{eqnarray} \label{vgvbf_vgfm}
%\vgvbf \simeq \vgfm. 
%\end{eqnarray}
\end{thm}

\begin{proof}
It suffices to work with $k =1$; the proof remains unchanged for other values of $k$. Let $\vgms$ denote the category of very good covariant functors $F \colon \bomu \lra \fvect$. Since every object of $\fvect$ is a finite dimensional vector space, the functor 
$
\vgm \lra \vgms
$
that sends a contravariant functor to its dual (defined objectwise) is an equivalence of categories.  Now we prove that the categories $\vgms$ and $\vgvb$ are equivalent. To proceed, we define a pair of functors 
\[
\theta \colon \vgms \rightleftarrows \vgvb : \psi
\]
For $\theta$, let $G$ be an object of $\vgms$. Define $\theta (G) = (E_G, \pi_G, F_{\pi_G})$ as 
\[E_G := \underset{U \in \bomu}{\mbox{colim}} U \times G(U) = \left(\coprod_{U \in \bomu} U \times G(U)\right)/\sim,\]
where $\sim$ is the usual equivalence relation for colimits.   The map $\pi_G$ is defined by  $\pi_G[(x,v)] =x$, and the functor $F_{\pi_{G}}$ is the same as $G$.  Now define the local trivialization 
$\vfi^{\pi_G}_U \colon U \times G(U) \lra \pi_{G}^{-1}(U)$ by  $\vfi^{\pi_G}_U(x, v) = [(x,v)].$  It is straightforward to check that $(E_G, \pi_G, F_{\pi_G})$ is a very good vector bundle. 
%Clearly $\vfi^{\pi_G}_U$ is a homeomorphism because the functor $G$ is very good and every morphism of $\bomu$ is an isotopy equivalence (both  arguments imply that for any $j \colon U \hookrightarrow V$ in $\bomu$, the image $G(j) \colon G(U) \lra G(V)$ is an isomorphism). The axioms of a very good vector bundle are readily verified. The axiom (VB0) comes from the second isomorphism of (\ref{iso_ker}). The remaining axioms ((VB1), (VB2), and (VB3)) immediately follow from the definitions of $\pi_G$ and $\vfi^{\pi_G}_U$.  So $\theta(G)$ is an object of $\vgvb$. On morphisms $\theta$ is defined in the obvious way. That is, if $\eta \colon G \lra G'$ is a morphism of $\vgms$, then $\theta (\eta) = (f_{\eta}, \eta)$, where $f_{\eta}[(x, v)] = [(x, \eta[U](v))]$ for some $U$ containing $x$. 
Now define $\psi$ as  $\psi(E, \pi, F_{\pi}) = F_{\pi}$. By the definitions, one has $\psi\theta = id$ and $\theta\psi \cong id$, which completes the proof. 
\end{proof}

%The following proposition, in which the abelian structure on $\vgms$ is the obvious one, follows immediately from Theorem~\ref{vgm_equivalence_thm}.

%\begin{prop}
%The pair $(\theta, \psi)$ from (\ref{pair_theta_psi}) preserves the abelian structure. So $\vgms$ and $\vgvb$ are equivalent as abelian categories. The same assertion holds for  $
%\vgmsf$ and $\vgvbf, k \geq 2$.
%\end{prop}

%Notice that $\vgm$ and $\vgms$ are not equivalent as abelian categories because of the bad behaviour of the dual functor with respect to the kernel and the cokernel.  

%\begin{rmk}
%Looking at Theorem~\ref{vgm_equivalence_thm}, and the equivalences that appear in the proof of Theorem~\ref{main_thm_paper}, one may ask the question to know whether the category $\vgvbf$  is equivalent to the category $\mathcal{F}_k(\om; \fvect)$ of very good homogeneous functors (of degree $k$) into $\fvect$. This question has a negative answer in general since the category $\fvect$ is not closed under  small limits. 

%(\ref{main_thm_eqn2}) is no longer true when replacing \lq\lq$\mcalc$\rq\rq{} by  \lq\lq$\fvect$\rq\rq{} because of the following reason. In the proof of Theorem~\ref{equiv_cat_thm}, which was done at the end of Section~\ref{vghf_section}, we used Kan extensions which are defined in terms of limits. In particular we used  a functor $\psi_3$ defined as $
%\psi_3(F)(U) = \underset{V \in \ok(U)}{\lim} \ F(V).
%$
%The problem is the fact that this limit could be an infinite dimensional vector space. This happens  for instance, when $k =1$ and $U$ has infinitely many components. 
%\end{rmk}

We close this section with a corollary. In the particular case when $\bou(M)$ is constructed from a triangulation $\tm$ of $M$, the category of very good vector bundles is deeply related to the category of representations of $\pio(M)$. Specifically, one has the following result.

\begin{coro}
Let $\mathcal{T}^{M}$ be a triangulation of $M$. As in Remark~\ref{bofm_butm} take $\mcalb_{\mcalu(\mcalt^M)}$ as the basis for the topology of $M$.
 Assume  that $M$ is  connected. Then one has 
$
\vgvb \simeq f\kbb[\pio(M)]\text{-}\emph{Mod},
$
where $f\kbb[\pio(M)]\text{-}\emph{Mod}$ denotes the category of finite dimensional representations of $\pio(M)$ over $\kbb$. 
\end{coro}

\begin{proof} 
This follows from Theorem~\ref{vgm_equivalence_thm}, Proposition~\ref{equiv_cat_vectn_prop} and Theorem~\ref{crucial_thm}. 
\end{proof}

\textsf{University of Regina, 3737 Wascana Pkwy, Regina, SK S4S 0A2, Canada\\
Department of Mathematics and Statistics\\}
\textit{E-mail address: pso748@uregina.ca}

\textsf{University of Regina, 3737 Wascana Pkwy, Regina, SK S4S 0A2, Canada\\
Department of Mathematics and Statistics\\}
\textit{E-mail address: donald.stanley@uregina.ca}


\addcontentsline{toc}{section}{References}


\begin{thebibliography}{99}
%\bibitem{aro_lam_vol07} G. Arone, P. Lambrechts, and I. Voli\'c , \textit{Calculus of functors, operad formality, and rational homology of embedding spaces}, Acta Math. 199 (2007), no. 2, 153--198.
%\bibitem{aro_tur12} G. Arone and V. Turchin, \textit{On the rational homology of high dimensional analogues of spaces of long knots}, arXiv:1105.1576v4 (2013).
%\bibitem{ber_moe03} C. Berger, I. Moerdijk, \textit{Axiomatic homotopy theory for operads}, Comment. Math. Helv., vol. 78 (2003), no. 4, 805--831.
%\bibitem{board_vogt68} J.M. Boardman and R.M. Vogt, \textit{Homotopy-everything H-spaces}, Bull. Amer. Math.  Soc. 74 (1968), 1117-1122. 
%\bibitem{boavida_weiss} P. Boavida de Brito and M. Weiss, \textit{Manifold calculus and homotopy sheaves}, To appear in  Homology Homotopy Appl.
%\bibitem{bous87} A.K. Bousfield, \textit{On the homology spectral sequence of a cosimplicial space}, Amer. J. Math., vol. 109 (1987), no. 2, 361--394.
%\bibitem{bous_kan72} A.K. Bousfield and D. M. Kan,  \textit{Homotopy limits, completion and localizations}, Springer-Verlag, Berlin, 1972, Lecture Notes in Mathematics, Vol. 304.
%\bibitem{bud07} R. Budney, \textit{Little cubes and long knots}, Topology 46 (2007), no. 1, 1--27.
\bibitem{buo03}  S. Buoncristiano, \textit{Fragments of geometric topology from the sixties},  Geom. Topol. 6 (2003). 
%\bibitem{cat_cot_ram_lon02} A. S. Cattaneo, P. Cotta-Ramusino, R. Longoni, \textit{Configuration spaces and Vassiliev classes in any dimension}, Algebr. Geom. Topol. 2 (2002) 949--1000. 
%\bibitem{cohen76} F. Cohen, T. Lada, and J. May, \textit{The homology of iterated loop spaces}, Springer-Verlag, Berlin, 1976, Lecture Notes in Mathematics, Vol. 533.  
%\bibitem{dwyer_hess10} W. Dwyer, K. Hess, \textit{Long knots and maps between operads}, Geom. Topol., vol.  16 (2012), no. 2, 919--955.
%\bibitem{fht00} Y. F\'elix, S. Halperin, J.C, Thomas, \textit{Rational Homotopy Theory}, Springer-Verlag, 2000.
%\bibitem{ful_har91}   W. Fulton and J. Harris, \textit{Representation theory. A first course.}, Graduate Texts in Mathematics, 129. Readings in Mathematics. Springer-Verlag, New York, 1991. xvi+551 pp. ISBN: 0--387--97527--6; 0--387--97495--4. 
%\bibitem{fress08} B. Fresse, \textit{Modules over Operads and Functors}, Lectures Notes in Mathematics, vol. 1967, Springer-Verlag, Berlin (2009).
\bibitem{gallier08}  J. Gallier, \textit{Notes on Convex Sets, Polytopes, Polyhedra Combinatorial Topology, Voronoi Diagrams and Delaunay Triangulations}, arXiv:0805.0292. 
% Link for the new version of the book of Gallier: http://www.cis.upenn.edu/~jean/gbooks/convexpoly.html
%\bibitem{ger_vor95} M. Gerstenhaber, A. A. Voronov, \textit{Homotopy G-algebras and moduli space operad}, Internat. Math. Res. Notices 1995, no. 3, 141--153 (electronic).
%\bibitem{getzler_jones} E. Getzler, J. D. S. Jones,  \textit{Operads, homotopy algebra and iterated integrals for double loop spaces}, arXiv:hep-th/9403055.
%\bibitem{goe_jar09} P. G. Goerss and J. F. Jardine, \textit{Simplicial Homotopy Theory}, Springer, vol. 174, 2009.
%\bibitem{goodwillie91} T. Goodwillie, Calculus II, \textit{Analytic functors}, K-Theory 5 (1991/1992) 295--332.
%\bibitem{goodwillie03} T. Goodwillie, Calculus III,  \textit{Taylor series}, Geom. Topol. 7 (2003) 645--711. 
\bibitem{good_weiss99} T. G. Goodwillie and M. Weiss, \textit{Embeddings from the point of view of immersion theory: Part II}, Geom. Topol. 3 (1999) 103--118.
%\bibitem{har_lam_tur_vol08} R. Hardt, P. Lambrechts, V. Turchin, and I. Voli\'c, \textit{Real homotopy theory of semi-algebraic sets},  Algebr. Geom. Topol. 11 (2011), no. 5, 2477--2545.
\bibitem{hatcher02} Allen Hatcher, \textit{Algebraic Topology}, Cambridge University Press, 2002. 
\bibitem{hirsch76} M. W. Hirsch, \textit{Differential Topology}, Springer-Verlag, New York Heidelberg, Berlin, 1976.
%\bibitem{hir03} P. Hirschhorn, \textit{Models categories and their localizations}, Mathematical Surveys and Monographs, 99. American Mathematical Society, Providence, RI, 2003.
%\bibitem{hovey99}  M. Hovey, \textit{Model categories}, Mathematical Surveys and Monographs, vol. 63, \textit{American Mathematical Society, Providence, RI}, 1999.
%\bibitem{hovey98} M. Hovey, \textit{Monoidal model categories}, preprint arXiv:math. AT/9803002(1998).
%\bibitem{komawila12} G. Komawila Ngidioni and P. Lambrechts, \textit{Euler series, Stirling numbers and the growth of the homology of the space of long links}, Belgian Mathematical Society-Simon Stevin. Bulletin. (2012) (1370--1444). 
%\bibitem{kont99} M. Kontsevich, \textit{Operads and motives in deformation quantization}, Mosh\'e Flato (1937-1998), Lett. Math. Phys. 48 (1999), no. 1, 35--72.
%\bibitem{lam_tur_vol10} P. Lambrechts, V. Turchin and I. Vol\'c, \textit{The rational homology of spaces of long knots in codimension > 2},  Geom. Topol. 14 (2010) 2151--2187. 
%\bibitem{lam_vol} P. Lambrechts and I. Voli\'c, \textit{Formality of the little N-disks operad} (2014), Memoirs of the American Mathematical Society, Volume 230, Number 1079.
%\bibitem{maclane98} S. Mac Lane, \textit{Categories for the Working Mathematician} (second edition), Graduate Texts in Mathematics 5, Springer Verlag, 1998.
%\bibitem{may92} J. P. May, \textit{Simplicial objects in algebraic topology}, University of Chicago Press, Chicago, IL, 1992.
%\bibitem{may72} P. May, \textit{The geometry of iterated loop spaces}, Lectures Notes in Mathematics, Vol. 271, Springer-Verlag, Berlin-New York, 1972.
%\bibitem{may01} J. P. May, \textit{The additivity of traces in triangulated categories}, Adv. Math. 163 (2001), no. 1, 34--73. 
%\bibitem{mcc01} J. McCleary, \textit{A user's guide to spectral sequences}, Second Edition, Cambridge University Press, 2001.
%\bibitem {mcc_smith02} JE McClure, JH Smith, \textit{A solution of Deligne's Hochschild cohomology conjecture}, from: "Recent progress in homotopy theory(Baltimore, MD, 2000)", Contemp. Math. 293, Amer. Math. Soc. (2002) 153--193 MR1890736.
%\bibitem{mcc_smith04} JE.McClure and JH. Smith,\textit{ Cosimplicial objects and little n-cubes I}, Amer. J . Math., 126 (2004), no. 5,  1109--1153.
%\bibitem{mcc_smith040} JE McClure, JH Smith, \textit{Operads and cosimplicial objects: an introduction}, Axiomatic, enriched and motivic homotopy theory, 133-171, NATO Sci. Ser. II Math. Phys. Chem., 131, Kluwer Acad. Publ., Dordrecht, 2004.
%\bibitem{moriya12} S. Moriya, \textit{Sinha's spectral sequence and homotopical algebra of operads}, arXiv:1210.0996v2 (2012).
%\bibitem{mun_vol09} B. A. Munson and I. Voli\'c, \textit{Cosimplicial model for spaces of links}, Journal of Homotopy and Related Structures, To appear.
\bibitem{mun60}  J. Munkres, \textit{Obstruction to the smoothing of piecewise-differentiable homeomorphisms},   Ann. of Math. (2), 72 (1960), 521--554. 
%\bibitem{muro11} F. Muro, \textit{Homotopy theory of nonsymmetric operads}, Algebr. Geom. Topol. 11 (2011), no. 3,  1541--1599.
%\bibitem{pirash00} T. Pirashvili, \textit{Hodge decomposition for higher order Hochschild homology}, Ann. Sci. Ecole Norm. Sup (4) 33 (2000), no. 2, 151--179.
\bibitem{nee01} A. Neeman, \textit{Triangulated categories}, Annals of Mathematics Studies, 148, Princeton University Press, Princeton, NJ, 2001, viii+449 pp.
\bibitem{pryor15} D. Pryor, \textit{Special open sets in manifold calculus}, Bull. Belg. Math. Soc. Simon Stevin 22 (2015), no. 1, 89--103.
%\bibitem{sakai08} K. Sakai, \textit{Poisson structures on the homology of the space of knots}, Groups, homotopy and configuration spaces, 463--482, Geom. Topol. Monogr., 13, Geom. Topol. Publ., Coventry, 2008.
%\bibitem{sal01} P. Salvatore, \textit{Configurations spaces with summable labels. Cohomological methods in homotopy theory} (Bellatera, 1998), 375--395, Progr. Math. 196, Birkhauser, Basel, 2001. 
%\bibitem{sal06} P. Salvatore, \textit{Knots, operads and double loop spaces}, Int. Math. Res. Not. 2006, Art. ID 13628, 22 pp.
%\bibitem{sal10} P. Salvatore, \textit{The Topological cyclic Deligne conjecture}, Algebr. Geom. Topol. 9 (2009), no. 1, 237--264.
%\bibitem{serre77} Jean-Pierre Serre, \textit{Linear representations of finite groups}, Translated from the second French edition by Leonard L. Scott. Graduate Texts in Mathematics, Vol. 42. Springer-Verlag, New York-Heidelberg, 1977. x+170 pp. ISBN: 0--387--90190--6. 
%\bibitem{shubert49} H. Shubert, \textit{Die eindeutige Zerlegbarkeit eines Knoten in Primknoten}, Sitzungsber. Akad. Wiss. Heidelberg, math.-nat. KI., \textbf{3:57-167}. (1949).
%\bibitem{sin04} D. Sinha, \textit{Manifold-theoretic  compactifications of configuration spaces}, Selecta Math. 10 (2004), no. 3, 391--428.
%\bibitem{sin06} D. Sinha, \textit{Operads and knot spaces}, J. Amer. Math. Soc., 19 (2006), no.2, 461--486 (electronic).
%\bibitem{sin09} D. Sinha, \textit{The topology of spaces of knots: cosimplicial models}, Amer. J. Math. \textbf{131} (2009), no. 4, 945--980.
%\bibitem {sma59} S. Smale, \textit{The classification of immersions of spheres in Euclidean spaces}, Ann. of Math. (2), 69:327--344, 1959.
%\bibitem{songhaf12} P. A. Songhafouo Tsopm\'en\'e, \textit{Formality of Sinha's cosimplicial model for long knots spaces and the Gerstenhaber algebra structure of homology}, Algebr. Geom. Topol. 13 (2013) 2193--2205.
%\bibitem{songhaf13} P. A. Songhafouo Tsopm\'en\'e, \textit{The rational homology of spaces of long links} (2013), arXiv:1312.7280.
\bibitem{paul_don17-2}  P. A. Songhafouo Tsopm\'en\'e, D. Stanley, \textit{Polynomial functors in manifold calculus} (2017), arXiv:1708.02642.
%\bibitem{paul_don17-3}  P. A. Songhafouo Tsopm\'en\'e, D. Stanley, \textit{Classification of homogeneous functors in manifold calculus}, Work in progress. 
\bibitem{paul_don17-4}  P. A. Songhafouo Tsopm\'en\'e, D. Stanley, \textit{Manifold calculus and triangulated categories}, Work in progress.

%\bibitem{spi01} M. Spitzweck, \textit{Operads, Algebras and Modules in General Model Categories}, preprint arXiv:math.AT/ 0101102 (2001).
%\bibitem{turchin12} V. Turchin, \textit{Context-free manifold calculus and the Fulton-MacPherson operad}, To appear in Algebr. Geom. Topol.
%\bibitem{tur10} V. Turchin, \textit{Delooping totalization of a multiplicative operad}, Journal of Homotopy and Related Structures (2013).
%\bibitem{turchin10} V. Turchin, \textit{Hodge-type decomposition in the homology of long knots}, J. Topol. 3 (2010), no. 3, 487--534.
%\bibitem{turchin04} V. Turchin, \textit{On the homology of the spaces of long knots}, from: "Advances in topological quantum field theory", NATO Sci. Ser. II Math. Phys. Chem. 179, Kluwer Acad. Publ., Dordrecht (2004) 23--52.
%\bibitem{turchin07} V. Turchin, \textit{On the other side of the bialgebra of chord diagrams}, J. Knot Theory Ramifications 16 (2007) 575--629. 
%\bibitem{vassiliev90} V. Vassiliev, \textit{Cohomology of knots spaces}, Theory of singularities and its applications, 23--69, Adv. Soviet Math., 1, Amer. Math. Soc., Providence, RI, 1990.
%\bibitem{ver71} J.L. Verdier, \textit{Cat\'egories d\'eriv\'ees}, in Springer Lecture Notes in Mathematics Vol 569, 262--311. 1971.
%\bibitem{weibel94} C. Weibel, \textit{An Introduction to Homological Algebra}, Volume 38 of Cambridge Studies in Advanced Mathematics, Cambridge University Press, Cambridge, 1994. 
%\bibitem{wei96} M. Weiss, \textit{Calculus of Embeddings}, Bull. Amer. Math. Soc. 33 (1996) 177--187. 
\bibitem{wei99} M. Weiss, \textit{Embedding from the point of view of immersion theory I}, Geom. Topol. 3 (1999), 67--101.
%\bibitem{weiss04} M.Weiss, \textit{Homology of spaces of smooth embeddings}, Q. J. Math. 55 (2004), no. 4, 499504.
\end{thebibliography}
\end{document}